\documentclass[twoside,a4paper,reqno,11pt]{amsart} 
\usepackage[top=28mm,right=28mm,bottom=28mm,left=28mm]{geometry}

\usepackage{amsfonts, amsmath, amssymb, mathrsfs, bm, latexsym, stmaryrd, array, hyperref, mathtools, bold-extra, url, xcolor}

\usepackage{verbatim}
\newlength\myverbindent 
\makeatletter
\def\verbatim@processline{%
  \hspace{\myverbindent}\the\verbatim@line\par}
\makeatother

\usepackage{upquote}

\usepackage{placeins}

\renewcommand{\a}{\alpha}
\renewcommand{\b}{\beta}

\newcommand{\e}{\varepsilon}
\renewcommand{\l}{\lambda} 

\newcommand{\s}{\sigma}

\renewcommand{\O}{\Omega}

\newcommand{\Z}{\mathbb{Z}}

\newcommand{\Q}{\mathbb{Q}}

\newcommand{\la}{\langle}
\newcommand{\ra}{\rangle}

\renewcommand{\to}{\rightarrow}

\newcommand{\leqs}{\leqslant}
\newcommand{\geqs}{\geqslant}

\newcommand{\vs}{\vspace{2mm}}

\newcommand{\Aut}{\operatorname{Aut}}
\newcommand{\normeq}{\trianglelefteqslant}

\newcommand{\GL}{\operatorname{GL}}
\newcommand{\PGL}{\operatorname{PGL}}

\makeatletter
\newcommand{\imod}[1]{\allowbreak\mkern4mu({\operator@font mod}\,\,#1)}
\makeatother

\newtheorem{theorem}{Theorem} 
\newtheorem*{conj*}{Conjecture}

\newtheorem{corol}[theorem]{Corollary}

\newtheorem{thm}{Theorem}[section] 
\newtheorem{prop}[thm]{Proposition} 
\newtheorem{lem}[thm]{Lemma}

\theoremstyle{definition}
\newtheorem{rem}[thm]{Remark}
\newtheorem{remk}[theorem]{Remark}

\newtheorem{defff}[theorem]{Definition}
\newtheorem{defn}[thm]{Definition}
\newtheorem{ex}[thm]{Example}

\begin{document}

\title[Fixed-point-free involutions in finite exceptional groups]{On fixed-point-free involutions in actions of finite exceptional groups of Lie type}

\author{Timothy C. Burness}
\address{T.C. Burness, School of Mathematics, University of Bristol, Bristol BS8 1UG, UK}
\email{t.burness@bristol.ac.uk}

\author{Mikko Korhonen}
\address{M. Korhonen, Shenzhen International Center for Mathematics, Southern University of Science and Technology, Shenzhen 518055, Guangdong, P. R. China}
\email{korhonen\_mikko@hotmail.com}

\begin{abstract}
Let $G$ be a nontrivial transitive permutation group on a finite set $\Omega$. By a classical theorem of Jordan, $G$ contains a derangement, which is an element with no fixed points on $\Omega$. Given a prime divisor $r$ of $|\Omega|$, we say that $G$ is $r$-elusive if it does not contain a derangement of order $r$. In a paper from 2011, Burness, Giudici and Wilson essentially reduce the classification of the $r$-elusive primitive groups to the case where $G$ is an almost simple group of Lie type. The classical groups with an $r$-elusive socle have been determined by Burness and Giudici, and in this paper we consider the analogous problem for the exceptional groups of Lie type, focussing on the special case $r=2$. Our main theorem describes all the almost simple primitive exceptional groups with a $2$-elusive socle. In other words, we determine the pairs $(G,M)$, where $G$ is an almost simple exceptional group of Lie type with socle $T$ and $M$ is a core-free maximal subgroup that intersects every conjugacy class of involutions in $T$. Our results are conclusive, with the exception of a finite list of undetermined cases for $T = E_8(q)$, which depend on the existence (or otherwise) of certain almost simple maximal subgroups of $G$ that have not yet been completely classified.
\end{abstract}

\makeatletter
\@namedef{subjclassname@2020}{\textup{2020} Mathematics Subject Classification}
\makeatother


\date{\today}

\maketitle

\setcounter{tocdepth}{1}
\tableofcontents

\section{Introduction}\label{s:intro}

Let $G \leqs {\rm Sym}(\O)$ be a transitive permutation group on a finite set $\Omega$ with $|\Omega| \geqs 2$. By a classical result of Jordan from 1872, namely \cite[Th\'eor\`eme I]{Jordan1872}, there exists an element in $G$ with no fixed points on $\Omega$. Such an element is called a \emph{derangement} and there is an extensive literature on derangements in permutation groups, demonstrating a wide range of connections and applications to other areas of mathematics. For instance, we refer the reader to Serre's article \cite{Serre}, which highlights applications of Jordan's theorem in number theory and topology. Another striking example is given by a celebrated theorem of Fein, Kantor and Schacher \cite{FKS}, which depends on the Classification of Finite Simple Groups. This theorem states that $G$ always contains a derangement of prime power order, and it turns out to have important applications in algebraic number theory concerning the structure of Brauer groups of global field extensions (see \cite{FKS} for the details). 

However, examples show that $G$ does not always contain a derangement of prime order, and such groups are called \emph{elusive}. Elusivity turns out to be a rather restrictive property. For example, if $G$ is primitive then a theorem of Giudici \cite[Theorem 1.1]{Giu} implies that $G$ is elusive if and only if $G = {\rm M}_{11} \wr K$ acting with the product action on $\Omega = \Delta^k$, where $K \leqs {\rm Sym}_k$ is transitive, $k \geqs 1$ and ${\rm M}_{11} < {\rm Sym}(\Delta)$ is primitive with $|\Delta| = 12$. A complete classification of the transitive elusive groups remains out of reach.

A local version of elusivity was introduced in \cite{BGW}, which is defined as follows: for a prime divisor $r$ of $|\Omega|$, we say that $G$ is \emph{$r$-elusive} if every element in $G$ of order $r$ has fixed points on $\O$. In particular, $G$ is elusive if and only if it is $r$-elusive for every prime divisor $r$ of $|\O|$. By \cite[Theorem 2.1]{BGW}, the classification of the $r$-elusive primitive groups is to a large extent reduced to the case where $G$ is an almost simple group (this means that the socle of $G$ is a nonabelian simple group $T$ and we have $T \normeq G \leqs {\rm Aut}(T)$). Moreover, the $r$-elusive primitive groups with socle an alternating or sporadic group are classified in \cite[Theorem 1.1]{BGW}. 

So let us assume $G$ is an almost simple primitive group of Lie type with socle $T$. A detailed analysis of the case where $T$ is a classical group is presented in  \cite{BG2,BG} and the main results provide a complete classification of the groups with an $r$-elusive socle. In this paper, we take the first steps towards extending these results to the exceptional groups of Lie type, focussing here on the special case $r=2$. Indeed, our main result describes the primitive exceptional groups with a $2$-elusive socle. In more abstract terms, this is equivalent to determining the pairs $(G,M)$, where $G$ is an almost simple exceptional group of Lie type with socle $T$, and $M$ is a core-free maximal subgroup of $G$ that intersects every conjugacy class of involutions in $T$. Our results are complete, with the exception of a small (finite) number of undetermined cases when $T = E_8(q)$ and the point stabilizer is an almost simple group with socle ${\rm L}_3(3)$ or ${\rm L}_2(q')$ for some odd prime power $q'$. This seemingly unavoidable ambiguity arises here because it remains a difficult open problem to determine the existence (or otherwise) of maximal subgroups of $G$ of this form.

In order to state our main results, we need to introduce some notation. Suppose that $G$ is an almost simple group with socle $T$, which is a finite simple exceptional group of Lie type over $\mathbb{F}_q$, where $q = p^f$ for some prime $p$. Write $T = (\bar{G}_{\s})'$, where $\bar{G}$ is a simple algebraic group of adjoint type over the algebraic closure of $\mathbb{F}_q$ and $\s$ is an appropriate Steinberg endomorphism of $\bar{G}$.

Let $\mathcal{M}$ be the set of core-free maximal subgroups of $G$. We first define a subset $\mathcal{C}$ of $\mathcal{M}$ as follows (in (IV), we write ${\rm soc}(H)$ for the socle of $H$, which we recall is defined to be the product of the minimal normal subgroups of $H$).

\begin{defff}\label{d:c}
A core-free maximal subgroup $H \in \mathcal{M}$ is contained in $\mathcal{C}$ if and only if one of the following holds:

\vspace{1mm}

\begin{itemize}\addtolength{\itemsep}{0.2\baselineskip}
\item[{\rm (I)}] $H = N_G(\bar{H}_{\s})$, where $\bar{H}$ is a maximal $\sigma$-invariant positive-dimensional closed subgroup of $\bar{G}$;
\item[{\rm (II)}] $H$ is of the same type as $G$ (either a subfield subgroup or a twisted version of $G$);
\item[{\rm (III)}] $H$ is the normalizer of an exotic $r$-local subgroup for some prime $r \neq p$;
\item[{\rm (IV)}] $T = E_8(q)$, $p > 5$ and ${\rm soc}(H) = \operatorname{Alt}_5 \times \operatorname{Alt}_6$ (the \emph{Borovik subgroup}).
\end{itemize}
\end{defff}

We then define $\mathcal{S} = \mathcal{M} \setminus \mathcal{C}$, so that $\mathcal{M}$ is a disjoint union 
\[
\mathcal{M} = \mathcal{C} \cup \mathcal{S}.
\]
By combining earlier work due to Borovik \cite{Bor} and Liebeck--Seitz \cite{LS90}, we know that every subgroup in $\mathcal{S}$ is almost simple. At the time of writing, all the subgroups in $\mathcal{C}$ have been determined, up to conjugacy, and the same is true for the maximal subgroups in $\mathcal{S}$ so long as $T \ne E_7(q), E_8(q)$. We refer the reader to Remark \ref{r:CS} for a more detailed description of the subgroups comprising the collections $\mathcal{C}$ and $\mathcal{S}$.

We begin by stating a simplified version of our main result.

\begin{theorem}\label{t:main0}
Let $G \leqs {\rm Sym}(\O)$ be a finite almost simple primitive permutation group with point stabilizer $H$ and socle $T$, which is a simple exceptional group of Lie type in characteristic $p$. If $T$ is $2$-elusive, then either

\vspace{1mm}

\begin{itemize}\addtolength{\itemsep}{0.2\baselineskip}
\item[{\rm (i)}] $(G,H)$ is known; or
\item[{\rm (ii)}] $T = E_8(q)$, $H$ is almost simple with socle $S$ and either $(S,p) = ({\rm L}_3(3), 13)$ or $S = {\rm L}_2(r)$ for some odd prime power $r$. 
\end{itemize}
\end{theorem}

A more detailed version of Theorem \ref{t:main0} is presented as Theorem \ref{t:main} below, where we refer to the subgroup collections $\mathcal{C}$ and $\mathcal{S}$ introduced above. Note that Tables \ref{tab:main}, \ref{tab:main2} and \ref{tab:main3} can be found at the end of the paper in Section \ref{s:tab}. In addition, the relevant groups with $|\O|$ odd are recorded in Theorem \ref{t:odd}, which is a special case of a theorem of Liebeck and Saxl \cite{LS} on primitive groups of odd degree.

\begin{theorem}\label{t:main}
Let $G \leqs {\rm Sym}(\O)$ be a finite almost simple primitive permutation group with point stabilizer $H$ and socle $T$, which is a simple exceptional group of Lie type in characteristic $p$. Set $H_0 = H \cap T$ and assume $|\O|$ is even. 

\vspace{1mm}

\begin{itemize}\addtolength{\itemsep}{0.2\baselineskip}
\item[{\rm (i)}] If $H \in \mathcal{C}$, then $T$ is not $2$-elusive if and only if $(T,H_0)$ is in Table \ref{tab:main}.
\item[{\rm (ii)}] If $T \ne E_8(q)$ and $H \in \mathcal{S}$, then $T$ is $2$-elusive if and only if $(T,H_0)$ is in Table \ref{tab:main2}.
\item[{\rm (iii)}] If $T = E_8(q)$ and $H \in \mathcal{S}$ has socle $S$, then $T$ is $2$-elusive only if $p$ is odd and either $S = {\rm L}_2(p^e)$ with $7 \leqs p^e \leqs 2621$, or if $(H_0,p)$ is in Table \ref{tab:main3}.
\end{itemize}
\end{theorem}

\begin{remk}\label{r:main}
We record a couple of comments on the statement of Theorem \ref{t:main}.

\vspace{1mm}

\begin{itemize}\addtolength{\itemsep}{0.2\baselineskip}
\item[{\rm (a)}] Typically, we find that $T$ is often $2$-elusive when $H \in \mathcal{C}$ and so it is more efficient to list the exceptions in Table \ref{tab:main}. On the other hand, there are fewer groups where $T$ is $2$-elusive and $H \in \mathcal{S}$. This explains the contrasting way we have chosen to state parts (ii) and (iii) in Theorem \ref{t:main}, in comparison with part (i).

\item[{\rm (b)}] In part (iii), we note that $T = E_8(q)$ has two conjugacy classes of involutions when $p$ is odd (see \cite[Table 4.5.1]{GLS} and Section \ref{ss:invols}). So if $S = {\rm L}_2(p^e)$ with $p$ odd, then $T$ is $2$-elusive only if $H_0$ contains an involutory diagonal or field automorphism of $S$. It is also worth noting that at the time of writing, there is not a single known example of a maximal subgroup $H \in \mathcal{S}$ with socle ${\rm L}_2(p^e)$. Furthermore, it is conjectured that no such example exists, see \cite[p.560]{SeitzTesterman} and \cite[Conjecture 3.3]{CravenPSL2}. Similarly, it remains an open problem to determine if there exists a maximal subgroup $H$ of $G$ corresponding to any of the cases $(H_0,p)$ listed in Table \ref{tab:main3}.
\end{itemize}
\end{remk}

As a corollary, we obtain the following result by combining Theorem \ref{t:main} with \cite[Theorem 1.5.1]{BG} and the main results in \cite{BG2,BGW} for $r=2$.

\begin{corol}\label{c:main2}
Let $G \leqs {\rm Sym}(\O)$ be a finite almost simple primitive permutation group with point stabilizer $H$ and socle $T$. Assume $|\O|$ is even. Then $T$ is $2$-elusive if and only if one of the following holds:

\vspace{1mm}

\begin{itemize}\addtolength{\itemsep}{0.2\baselineskip}
\item[{\rm (i)}] $T$ is an alternating or sporadic group and either $G$ is $2$-elusive and $(G,H)$ is recorded in \cite{BGW}, or $(G,H)$ is one of the cases in Table \ref{tab:not2}.

\item[{\rm (ii)}] $T$ is a classical group and the possibilities for $(G,H)$ are determined in \cite{BG2,BG}.

\item[{\rm (iii)}] $T$ is an exceptional group of Lie type and the possibilities for $(G,H)$ are determined in Theorem \ref{t:main}.
\end{itemize}
\end{corol}

{\scriptsize
\begin{table}
\[
\begin{array}{lll} \hline
G & H & \mbox{Conditions} \\ \hline
{\rm Sym}_n & {\rm Sym}_k \times {\rm Sym}_{n-k} & \mbox{$H$ intransitive, $n \equiv 2 \imod{4}$, $1 \leqs k < n/2$ odd} \\
{\rm Sym}_6 & {\rm Sym}_5 & \mbox{$H$ primitive} \\
{\rm PGL}_2(9) & 3^2{:}8 & \\
{\rm Alt}_6.2^2 & 2 \times 5{:}4, \, 3^2{:}{\rm SD}_{16} & \\
{\rm He}.2 & {\rm Sp}_4(4).4, \, 5^2{:}4{\rm Sym}_4 & \\
{\rm J}_2.2 & ({\rm Alt}_5 \times {\rm D}_{10}).2, \, 5^2{:}(4 \times {\rm Sym}_3) & \\
{\rm J}_3.2 & {\rm L}_2(16).4 & \\
{\rm O'N}.2 & 3^3{:}2^{1+4}.{\rm D}_{10}.2 & \\ \hline
\end{array}
\]
\caption{The special cases arising in part (i) of Corollary \ref{c:main2}}
\label{tab:not2}
\end{table}
}

\begin{remk}
Let $G \leqs {\rm Sym}(\O)$ be an almost simple primitive group with socle $T$ and point stabilizer $H$, where $T$ is an alternating or sporadic group. The $2$-elusive groups of this form  are determined in \cite{BGW}, but some additional work is required to find all of the examples where $T$ is $2$-elusive. Indeed, if $G$ is $2$-elusive then so is $T$, but the converse is false, in general. For example, if we consider the natural action of $G = {\rm Sym}_6$ on $\O = \{1,\ldots, 6\}$, then $T = {\rm Alt}_6$ is $2$-elusive, but $G$ is not. By inspecting the relevant proofs in \cite{BGW}, it is a straightforward exercise to determine the complete list of pairs $(G,H)$ that arise in this way, and these are the cases recorded in Table \ref{tab:not2}.
\end{remk}

By combining Theorem \ref{t:main} with \cite[Theorem 4.1.7]{BG}, we obtain the following corollary on maximal parabolic subgroups of simple groups of Lie type, which may be of independent interest. For a classical group $G$, recall that each maximal parabolic subgroup is the stabilizer of a totally singular $m$-dimensional subspace of the natural module $V$, which we denote by $P_m$ (if $G = {\rm L}_n(q)$ is a linear group, then we view all subspaces of $V$ as totally singular). In addition, we adopt the standard $P_m$ notation for the maximal parabolic subgroups of exceptional groups (this agrees with the usual Bourbaki labelling of the simple roots for $G$, as given in \cite[11.4]{Humphreys}). Also note that in Table \ref{tab:parab}, we use the notation $(a)_2$ to denote the largest power of $2$ dividing $a$.

\begin{corol}\label{c:parmain}
Let $G$ be a simple group of Lie type and let $H = P_m$ be a maximal parabolic subgroup of $G$. Then either $H$ contains a representative of every conjugacy class of involutions in $G$, or $(G,m)$ is one of the cases recorded in Table \ref{tab:parab}.
\end{corol}

{\scriptsize
\begin{table}
\[
\begin{array}{ll} \hline
G & \mbox{Conditions} \\ \hline
{\rm L}_n(q) & \mbox{$n$ even, $mq$ odd, and either $q \equiv 3 \imod{4}$ or $(n)_2 > (q-1)_2$} \\
{\rm U}_n(q), \, n \geqs 3 & \mbox{$n$ even, $m = n/2$ and $(n)_2 < (q+1)_2$} \\
{\rm PSp}_n(q), \, n \geqs 4 &  \mbox{$q \equiv 3 \imod{4}$ and $m$ odd} \\
{\rm P\O}_n^{+}(q), \, n \geqs 8 & \mbox{$q \equiv 3 \imod{4}$ and either $m \geqs n/2-1$, or $m$ is odd and $n \equiv 0 \imod{4}$} \\
{\rm P\O}_n^{-}(q), \, n \geqs 8 & \mbox{$q \equiv 7 \imod{8}$, $n \equiv 2 \imod{4}$ and $m$ odd} \\
\O_n(q), \, n \geqs 7 & \mbox{$q \equiv 3 \imod{4}$ and $m = (n-1)/2$} \\
E_7(q) & \mbox{$q \equiv 3 \imod{4}$ and $m \in \{2,5,7\}$} \\ \hline
\end{array}
\]
\caption{The special cases $(G,m)$ arising in Corollary \ref{c:parmain}}
\label{tab:parab}
\end{table}
}

Finally, by combining Corollary \ref{c:main2} with \cite[Theorem 2.1]{BGW}, we get the following result for arbitrary primitive permutation groups.

\begin{corol}\label{c:main3}
Let $L \leqs {\rm Sym}(\O)$ be a finite primitive permutation group with socle $S$ and assume $|\O|$ is even. Then $S$ is $2$-elusive if and only if the following hold:

\vspace{1mm}

\begin{itemize}\addtolength{\itemsep}{0.2\baselineskip}
\item[{\rm (i)}] $L \leqs G \wr {\rm Sym}_k$ acting with its product action on $\Omega = \Delta^k$ for some $k \geqs 1$, where $G \leqs {\rm Sym}(\Delta)$ is an almost simple primitive group with socle $T$ and point stabilizer $H$; and

\item[{\rm (ii)}] $T$ is $2$-elusive on $\Delta$, so $(G,H)$ is one of the cases recorded in Corollary \ref{c:main2}.
\end{itemize}
\end{corol}

\vs

In future work, we will investigate the $r$-elusive actions of almost simple exceptional groups of Lie type for all odd primes $r$.

\vs

The structure of this paper is as follows. We begin in Section \ref{s:prel} by presenting a number of preliminary results, which we will need in the proof of Theorem \ref{t:main}. In particular, we fix notation and we introduce our algebraic group setup, which we will use throughout the paper. We also recall some of the main results on the conjugacy classes of maximal subgroups and involutions in the simple exceptional groups (both finite and algebraic), and we briefly explain how we will apply computational methods, working with {\sc Magma} \cite{magma}. The remainder of the paper is dedicated to the proof of Theorem \ref{t:main}. The cases where the point stabilizer $H$ is a parabolic or subfield subgroup are straightforward and they are treated in Sections \ref{ss:parsub} and \ref{ss:subf}, respectively. The low rank groups are then handled in Section \ref{s:proof_low}, which reduces the proof of Theorem \ref{t:main} to the groups with socle $T = F_4(q)$, $E_6^{\e}(q)$, $E_7(q)$ or $E_8(q)$, and these four cases are handled in Sections \ref{s:proof_F4} - \ref{s:e8}, respectively. Finally, the relevant tables referred to in the statement of Theorem \ref{t:main} are presented in Section \ref{s:tab}.

We conclude the introduction by briefly outlining some of the methods used in the proofs of our main results. Let $G$ be an almost simple exceptional group of Lie type
with socle $T$ and let $H$ be a core-free maximal subgroup with $|\Omega| = [G:H]$ even. We will consider each possibility for $G$ and $H$ in turn, with the aim of determining whether or not $H_0 = H \cap T$ intersects every $T$-class of involutions in $T$ (in other words, our goal to determine whether or not the action of $T$ on $\Omega = G/H$ is $2$-elusive). In a handful of cases, we find that the number of $T$-classes of involutions in $T$ exceeds the number of $H_0$-classes of involutions in $H_0$, which immediately implies that $T$ is not $2$-elusive. But in all other cases, we need to study the fusion of $H_0$-classes in $T$ in order to determine whether or not every $T$-class of involutions has a representative in $H_0$.

Given an involution $g \in H_0$, we can usually identify the $T$-class of $g$ by computing the dimension of the fixed point space $C_V(g)$ of $g$ on the adjoint module $V = \mathcal{L}(\bar{G})$ for the ambient simple algebraic group $\bar{G}$ (see Section \ref{ss:invols}). In particular, if we have expressed $g$ explicitly as a product of root elements, then we can compute $\dim C_V(g)$ using {\sc Magma}, as discussed in Section \ref{ss:comp}. For example, this is typically the approach we take when $H_0$ is a subgroup of maximal rank, such as the normalizer of a maximal torus in $T$. In other cases, we will often compute $\dim C_V(g)$ by appealing to information available in the literature on the composition factors of the restriction $V \downarrow H_0$. And in cases where this information is not readily available, then a different argument is required, which may involve working with an explicit construction of $H_0$ and $V \downarrow H_0$.

Recall that we partition the set of maximal subgroups of $G$ into two parts, denoted $\mathcal{C}$ and $\mathcal{S}$. If $H \in \mathcal{C}$ is of the form $H = N_G(\bar{H}_{\s})$ for some $\sigma$-invariant positive-dimensional closed subgroup $\bar{H}$ of $\bar{G}$, then we can often work with a description of the composition factors of $V \downarrow \bar{H}$, which is readily available in the literature (for example, the tables in \cite[Chapter 12]{Thomas} provide a convenient source). Similarly, if $H \in \mathcal{S}$, then $H$ is almost simple and the possibilities for $H$ have been determined up to conjugacy for $T \ne E_7(q), E_8(q)$; in particular, the structure of 
$V \downarrow H_0$ is known and as before we can use this to determine the fusion of $H_0$-classes in $T$. However, it remains an open problem to determine the subgroups $H \in \mathcal{S}$ when $T = E_7(q)$ or $E_8(q)$, despite substantial progress in recent years. Here the possibilities for the socle of $H$ have been narrowed down to a fairly short list of candidates and to handle these cases we typically take a computational approach, working with \emph{feasible characters} as in \cite{Litt}, which we discuss in Section \ref{ss:feasible}. In order to rule out the possibility of a $2$-elusive action of $T$ on $G/H$, the goal is to compute all the feasible characters of $H_0$ on $V$ and then show that none of them are consistent with both 
\begin{itemize}\addtolength{\itemsep}{0.2\baselineskip}
\item[{\rm (a)}] the maximality of $H$ as a subgroup of $G$; and
\item[{\rm (b)}] $H_0$ intersecting every $T$-class of involutions in $T$.
\end{itemize}
If this approach is feasible, then we can use it to rule out the possibility of a $2$-elusive action, without determining whether or not such a maximal subgroup of $G$ exists.

\vs

\noindent \textbf{Acknowledgements.} We thank an anonymous referee for their careful reading of an earlier version of the paper and for sharing many useful comments and suggestions. We also thank David Craven and Alastair Litterick for helpful discussions. MK was supported by NSFC grant 12350410360.

\section{Preliminaries}\label{s:prel}

\subsection{Notation}\label{ss:nota}

We begin by fixing some of the general notation we will use throughout the paper. Further notation will be introduced as and when needed.

Let $n$ be a positive integer and let $A$ and $B$ be groups. We denote a cyclic group of order $n$ by $Z_n$, and often just by $n$, and we will use $[n]$ to denote an unspecified solvable group of order $n$. We write $A{:}B$ for an unspecified split extension (semidirect product) of $A$ and $B$, where $A$ is a normal subgroup. Similarly, $A.B$ denotes an unspecified extension of $A$ and $B$ (possibly nonsplit) and we use $A \circ B$ to denote a central product of $A$ and $B$. And for positive integers $a$ and $b$, we write $(a,b)$ for the highest common factor of $a$ and $b$. 

We adopt the notation for finite simple groups from \cite{KL}; for example, we write ${\rm L}_n(q) = {\rm L}_n^{+}(q) = {\rm PSL}_n(q)$ and ${\rm U}_n(q) = {\rm L}_n^{-}(q) = {\rm PSU}_n(q)$ for linear and unitary groups. In some situations, it will also be convenient to adopt the Lie notation for classical groups, so we will write $A_{n-1}(q)$ and $A_{n-1}^-(q)$ for ${\rm L}_n(q)$ and ${\rm U}_n(q)$, etc.

Our notation for matrices is also fairly standard. First we write $I_n$ for the identity matrix of size $n$, and we use $J_n$ for a unipotent Jordan block of size $n$. We denote a block-diagonal matrix with diagonal components $A_1, \ldots, A_t$ by writing $(A_1, \ldots, A_t)$. For instance, $(I_2, -I_3)$ is the $5 \times 5$ diagonal matrix ${\rm diag}(1,1,-1,-1,-1)$. For a positive integer $k$, we use $J_n^k$ to denote the block-diagonal matrix $(J_n, \ldots, J_n)$, which is a unipotent matrix with $k$ Jordan blocks of size $n$. This extends naturally to the notation $(J_{n_1}^{k_1}, \ldots, J_{n_t}^{k_t})$ for arbitrary unipotent matrices and it will be convenient to adopt the shorthand $(n_1^{k_1}, \ldots, n_t^{k_t})$ for such a matrix. So for example, $(2^3, 1^2)$ denotes an $8 \times 8$ unipotent matrix with $3$ Jordan blocks of size $2$, and $2$ Jordan blocks of size $1$.

\subsection{Setup}\label{ss:setup}

Here we fix our basic setup, which we will adopt for the remainder of the paper. 

Let $G \leqs {\rm Sym}(\O)$ be a finite almost simple primitive permutation group with point stabilizer $H$ and socle $T$. Set $H_0 = H \cap T$ and assume that $T$ is an exceptional group of Lie type over $\mathbb{F}_q$, where $q = p^f$ for some prime $p$ and positive integer $f$. We can write $T = (\bar{G}_{\s})'$, where $\bar{G}$ is a simple algebraic group of adjoint type over the algebraic closure $K = \bar{\mathbb{F}}_q$ and $\bar{G}_{\s}$ is the subgroup of fixed points of an appropriate Steinberg endomorphism $\s$ of $\bar{G}$. If $\bar{G}_{{\rm sc}}$ denotes the simply connected cover of $\bar{G}$ and we write $\s$ for the corresponding Steinberg endomorphism of $\bar{G}_{{\rm sc}}$, then 
$T \cong \left( \bar{G}_{{\rm sc}} \right)_{\s} / Z(\bar{G}_{{\rm sc}})_{\s}$,
unless $T  = {}^2G_2(3)'$, ${}^2F_4(2)'$ or $G_2(2)'$.

We now give a more precise description of $T$, following \cite{SteinbergNotes} and \cite[Chapter 2]{GLS}, which we also refer the reader to for more details. 

Let $\Phi$ be the root system of $\bar{G}$ and choose a base $\Delta = \{\alpha_1, \ldots, \alpha_{\ell}\}$ of simple roots in $\Phi$, where we adopt the standard Bourbaki labelling (see \cite[11.4]{Humphreys}). Let $\Phi^+$ be the corresponding set of positive roots. We may assume that $\bar{G}$ is a Chevalley group as defined in \cite{SteinbergNotes}, generated by the set of root elements $x_{\alpha}(t)$ with $\alpha \in \Phi$ and $t \in K$, which are constructed via reduction modulo $p$, as in \cite[Chapter 3]{SteinbergNotes}.

For $\alpha \in \Phi$, we write $U_{\alpha} = \langle x_{\alpha}(t) \, :\,  t \in K \rangle$ for the root subgroup corresponding to $\alpha$. And for each $t \in K^{\times}$ we define 
\[ 
w_{\alpha}(t) = x_{\alpha}(t) x_{-\alpha}(-t^{-1}) x_{\alpha}(t),\;\;  
h_{\alpha}(t) = w_{\alpha}(t) w_{\alpha}(1)^{-1},\;\; w_{\alpha} = w_{\alpha}(1).
\]
Then 
\[
\bar{T} = \langle h_{\alpha}(t) \,:\, \alpha \in \Phi, \, t \in K^{\times} \rangle = \la h_{\alpha}(t) \,:\, \alpha \in \Delta, \, t \in K^{\times}\ra
\]
is a maximal torus of $\bar{G}$ and $W = N_{\bar{G}}(\bar{T}) / \bar{T}$ is the Weyl group of $\bar{G}$.

We have $w_{\alpha}^2 = h_{\alpha}(-1)$ and $w_{-\alpha} = w_{\alpha}^{-1}$ for all $\alpha \in \Phi$. In addition, $h_{\alpha}(st) = h_{\alpha}(s)h_{\alpha}(t)$ and $h_{-{\alpha}}(t) = h_{\alpha}(t^{-1})$ for all $\alpha \in \Phi$ and $s,t \in K^{\times}$. Recall that $\dim \bar{T} = |\Delta|$ is called the \emph{rank} of $\bar{G}$.

For $\alpha \in \Phi$, we write $s_{\alpha} \in W$ for the image of the element $w_{\alpha} \in N_{\bar{G}}(\bar{T})$, which we refer to as the reflection corresponding to $\alpha$. Note that $W = \la s_{\alpha} \,:\, \alpha \in \Phi\ra = \la s_{\alpha} \,:\, \alpha \in \Delta \ra$.

Let $\s_q: \bar{G} \rightarrow \bar{G}$ be the Frobenius endomorphism of $\bar{G}$ corresponding to the field automorphism $t \mapsto t^q$ of $K$, which maps $x_{\alpha}(t) \mapsto x_{\alpha}(t^q)$ for all $t \in K$ and $\alpha \in \Phi$. In the \emph{untwisted case}, we have $\s = \s_q$ and the group $O^{p'}(\bar{G}_{\s})$ is  generated by the set of root elements $\{x_{\alpha}(t)\,:\, \alpha \in \Phi,\, t \in \mathbb{F}_q\}$. And we note that $O^{p'}(\bar{G}_{\s}) = T$ unless $\bar{G}_{\s} = G_2(2)$, in which case $O^{p'}(\bar{G}_{\s}) = \bar{G}_{\s} = T.2$.

In the \emph{twisted case}, the groups that arise are the \emph{Steinberg groups} and the \emph{Suzuki--Ree groups}. For Steinberg groups, we have $\s = \tau \s_q$, where $\tau$ is a graph automorphism of $\bar{G}$. Here $\tau$ corresponds to a permutation $\tau'$ of the root system $\Phi$ with $\tau'(\Delta) = \Delta$, and as in Corollary (b) of \cite[Theorem 29]{SteinbergNotes} we have 
\[
\tau(x_{\alpha}(t)) = x_{\tau'(\alpha)}(\varepsilon_{\alpha}t),
\]
where $\varepsilon_{\alpha} = \pm$ for all $\alpha \in \Phi$, and $\varepsilon_{\alpha} = +$ for all $\alpha \in \pm \Delta$. (The values of $\varepsilon_{\alpha}$ for $\alpha \notin \pm \Delta$ depend on the structure constants of the Chevalley basis used in the construction of $\bar{G}$.)

In the case of exceptional groups, the Steinberg groups that arise are $T = {}^3D_4(q)$ and $T = {}^2E_6(q)$. If $T = {}^3D_4(q)$, then $\bar{G}$ is of type $D_4$ and we take $\tau$ to be a triality graph automorphism corresponding to the following permutation $\tau'$ of $\Delta = \{\a_1,\a_2,\a_3,\a_4\}$: 
\[ 
\alpha_1 \mapsto \alpha_3 \mapsto \alpha_4 \mapsto \alpha_1, \;\; \alpha_2 \mapsto \alpha_2. 
\]
For example, notice that $T$ contains the elements 
$x_{\alpha_1}(s) x_{\alpha_3}(s^q) x_{\alpha_4}(s^{q^2})$ and $x_{\alpha_2}(t)$ for all $s \in \mathbb{F}_{q^3}$ and $t \in \mathbb{F}_q$.
Similarly, if $T = {}^2E_6(q)$ then $\bar{G} = E_6$ and we take $\tau$ to be an involutory graph automorphism defined by the following permutation of $\Delta$:
\[
\alpha_1 \mapsto \alpha_6 \mapsto \alpha_1, \;\; \alpha_3 \mapsto \alpha_5 \mapsto \alpha_3, \;\; \alpha_2 \mapsto \alpha_2, \;\; \alpha_4 \mapsto \alpha_4.
\]
Then $T$ contains elements such as $x_{\alpha_1}(s)x_{\alpha_6}(s^q)$,  $x_{\alpha_3}(s) x_{\alpha_5}(s^q)$ and $x_{\alpha_2}(t)$, $x_{\alpha_4}(t)$ for all
$s \in \mathbb{F}_{q^2}$ and $t \in \mathbb{F}_q$.

For the Suzuki--Ree groups we have $\s = \psi \s_q$, where $\psi:\bar{G} \to \bar{G}$ is an exceptional isogeny; in this case $q = p^f$ with $f$ odd, and either $p = 2$ and $T \in \{ {}^2B_2(q), {}^2F_4(q)' \}$, or $p = 3$ and $T = {}^2G_2(q)'$ (note that if $T = {}^2B_2(q)$ then we may assume $q \geqs 8$ since ${}^2B_2(2) = 5{:}4$ is solvable). Here $\psi$ is defined via a certain involution $\alpha \mapsto \alpha^\vee$ on $\Phi$ which swaps short and long roots, and preserves angles between the simple roots, see Corollary (b) of \cite[Theorem~29]{SteinbergNotes}. Then  
\[
\psi(x_{\alpha}(t)) = \begin{cases} x_{\alpha^\vee}(\varepsilon_{\alpha} t) & \mbox{if  $\alpha$ is long} \\ 
x_{\alpha^\vee}(\varepsilon_{\alpha} t^2) & \mbox{if $\alpha$ is short,} 
\end{cases}
\]
where once again we have $\varepsilon_{\alpha} = \pm$ for all $\alpha \in \Phi$, and $\varepsilon_{\alpha} = +$ for all $\alpha \in \pm \Delta$.

For more detailed information on the Steinberg groups and the Suzuki--Ree groups, we refer the reader to \cite[Chapter 11]{SteinbergNotes} and \cite[Chapter 2]{GLS}.

\begin{rem}\label{r:ordering}
As stated above, we adopt the standard Bourbaki labelling of the simple roots $\Delta = \{\a_1, \ldots, \a_{\ell}\}$ for $\bar{G}$. But in certain places, we will also need a labelling for the complete set of positive roots $\Phi^+$ (for instance, see Example \ref{ex:e6t1} and the proofs of Propositions \ref{p:f4_2}, \ref{p:e6_2} and Lemmas \ref{l:e7_mr}, \ref{l:e8_2}, \ref{l:e8_3}). It will be convenient to adopt the same ordering of the roots as used by {\sc Magma} \cite{magma} and we will do this consistently throughout the paper. 

Let us briefly explain how the {\sc Magma} ordering of roots is determined. First recall that the \emph{height} of a positive root $\a \in \Phi^{+}$ is the positive integer ${\rm ht}(\a) = \sum_{i} c_i$, where $\a = \sum_{i=1}^{\ell} c_i\a_i$. The positive roots are first ordered by height, so $\operatorname{ht}(\alpha) < \operatorname{ht}(\beta)$ implies $\alpha < \beta$. And then roots of equal height are ordered lexicographically as follows. For positive roots $\alpha, \beta \in \Phi^+$, we have $\alpha - \beta = \sum_{i = 1}^{\ell} d_i \a_i$ for some integers $d_i \in \mathbb{Z}$. If ${\rm ht}(\alpha) = {\rm ht}(\beta)$, then we define $\alpha < \beta$ if and only if $d_k > 0$, where $k$ is minimal such that $d_k \neq 0$. This defines a total order on $\Phi^+$ with the property that the simple roots $\a_1, \ldots, \a_{\ell}$ are the first $\ell$ roots in this ordering. We extend the notation for simple roots by writing $\alpha_i$ for the $i$-th root in $\Phi^+$ with respect to the ordering defined above.
\end{rem}

\subsection{Algebraic groups}\label{ss:alg}

Given the setup introduced in the previous section, we will often work with algebraic groups and their representations in this paper. In doing so, we will typically follow Jantzen's notation from \cite{Jantzen}, some of which is briefly recalled below.

Let $\bar{H}$ be a connected semisimple algebraic group of rank $\ell$ over $K = \bar{\mathbb{F}}_q$. Often we will denote $\bar{H}$ by its type; for example, $\bar{H} = A_4A_1$ means that $\bar{H}$ is a connected semisimple group of type $A_4A_1$. With respect to a fixed maximal torus and a complete set $\{\alpha_1, \ldots, \alpha_{\ell}\}$ of simple roots for $\bar{H}$, we denote the fundamental dominant  weights by $\{\varpi_1, \ldots, \varpi_{\ell} \}$. Throughout the paper, we will always adopt the standard Bourbaki labelling for the $\a_i$ and $\varpi_i$, as described in \cite[11.4]{Humphreys}. The Lie algebra of $\bar{H}$ is denoted by $\mathcal{L}(\bar{H})$ and we refer to it as the \emph{adjoint module} for $\bar{H}$.

By an $\bar{H}$-module, we will always mean a rational module defined over $K$. For a dominant weight $\lambda$, we use $V_{\bar{H}}(\lambda)$ to denote the Weyl module with highest weight $\lambda$. Similarly, $L_{\bar{H}}(\lambda)$ is the irreducible $\bar{H}$-module with highest weight $\lambda$. When there is no confusion, we will also use the notation $V(\lambda)$ and $L(\lambda)$. Similarly, if  $\bar{H} = A_1$, then we will write $c \varpi_1 = c$ for $c \in \Z$, so $V_{\bar{H}}(c \varpi_1) = V_{\bar{H}}(c)$ and $L_{\bar{H}}(c \varpi_1) = L_{\bar{H}}(c)$ for $c \geqs 0$.

For simply connected simple $\bar{H}$ of exceptional type, other than $E_8$, we define the \emph{minimal module} $V_{\mathrm{min}}$ to be the Weyl module $V_{\bar{H}}(\varpi_1), V_{\bar{H}}(\varpi_4), V_{\bar{H}}(\varpi_1), V_{\bar{H}}(\varpi_7)$ for $\bar{H} = G_2, F_4, E_6, E_7$, respectively. Note that $V_{\mathrm{min}}$ has dimension $7$, $26$, $27$, $56$ for $\bar{H} = G_2, F_4, E_6, E_7$, respectively.

Let $W_1$, $W_2$, $\ldots$, $W_t$ be $\bar{H}$-modules. Then for an $\bar{H}$-module $V$, we will use the notation $V = W_1/W_2/\cdots/W_t$ to denote that $V$ has the same composition factors as the direct sum $W_1 \oplus W_2 \oplus \cdots \oplus W_t$.

\subsection{Computational methods}\label{ss:comp}

Throughout this paper, we will often need to identify the conjugacy class of a given involution $g \in \bar{G}$, and in some cases we will adopt a computational approach to do this, working with {\sc Magma} \cite{magma} (version V2.28-11). Typically, the relevant element $g$ will be written as a product of elements of the form $h_{\alpha}(-1)$, $w_{\alpha}$ and $x_{\alpha}(\pm 1)$, with respect to the notation in Section \ref{ss:setup}. As we will explain in Section \ref{ss:invols}, we can determine the 
$\bar{G}$-class of $g$ by computing the dimension of the fixed point space 
\[
C_V(g) = \{ v \in V \,: \, gv = v \}
\]
with respect to the minimal module $V = V_{\mathrm{min}}$ or the adjoint module $V = \mathcal{L}(\bar{G})$, as defined in Section \ref{ss:alg}.

Let us explain how we can compute $\dim C_V(g)$ using {\sc Magma}. As before, let $\bar{G}_{{\rm sc}}$ be the simply connected cover of $\bar{G}$, and define elements $x_{\alpha}'(t)$, $h_{\alpha}'(t)$, $w_{\alpha}'$ in $\bar{G}_{{\rm sc}}$ in the same way we defined the analogous elements in $\bar{G}$ (see Section \ref{ss:setup}). There is an isogeny $\bar{G}_{{\rm sc}} \rightarrow \bar{G}$ mapping $x_{\alpha}'(t) \mapsto x_{\alpha}(t)$, so for the relevant computations we can work with a suitable element $g' \in \bar{G}_{{\rm sc}}$. In the cases we are interested in, the element $g'$ is defined over the prime field $\mathbb{F}_p$ since it can be written as a product of root elements of the form $x_{\alpha}'(\pm 1)$, so $g' \in (\bar{G}_{{\rm sc}})(p)$, which is the group of $\mathbb{F}_p$-rational points of $\bar{G}_{{\rm sc}}$. Furthermore, as a $K[(\bar{G}_{{\rm sc}})(p)]$-module, $V$ is defined over $\mathbb{F}_p$, which means that $V = K \otimes_{\mathbb{F}_p} V_0$ for some absolutely irreducible $\mathbb{F}_p[(\bar{G}_{{\rm sc}})(p)]$-module $V_0$. 

As a consequence, it suffices to determine the action of $g'$ on $V_0$ and then read off $\dim C_{V_0}(g')$, since tensoring by $K$ does not change the dimension of the fixed point space. 

\begin{ex}
For instance, the following {\sc Magma} code verifies that in characteristic $p = 2$, the involution $g = x_{\alpha_1}(1) x_{\alpha_4}(1)$ of $\bar{G} = F_4$ has a fixed point space of dimension $14$ on $V = V_{\mathrm{min}}$, and dimension $28$ on $V = \mathcal{L}(\bar{G})$: 

\vs

{\small
\begin{verbatim}
G := GroupOfLieType("F4", GF(2) : Isogeny := "SC");
g := elt<G | <1,1>, <4,1>>;
r := HighestWeightRepresentation(G, [0,0,0,1]);
A := Matrix(r(g)); // action of g on V_min
Dimension(Kernel(A-1)); // output: 14
s := AdjointRepresentation(G);
B := Matrix(s(g)); // action of g on Lie(G)
Dimension(Kernel(B-1)); // output: 28
\end{verbatim}
}
\end{ex}

Note that since $g$ is an involution, the dimension of $C_V(g)$ uniquely determines the Jordan normal form of $g$ on $V$. For instance, in the above example we see that $g$ has Jordan form $(2^{12}, 1^2)$ on $V_{\mathrm{min}}$ and $(2^{24}, 1^4)$ on $\mathcal{L}(\bar{G})$, and we will often denote this by writing $\mathcal{L}(\bar{G}) \downarrow g = (2^{24},1^4)$, for example. As explained in the next section, this allows us to conclude that $g$ is contained in the $\bar{G}$-class of involutions labelled  $A_1\tilde{A_1}$. 

In the same way, we can calculate the action of $g \in \bar{G}$ on any given Weyl module $V = V_{\bar{G}}(\lambda)$. We will often be interested in performing such calculations when $p=2$, in which case the involutions are unipotent elements. There will also be some cases where we need to do this when $p$ is odd. In the latter setting, each involution is semisimple and it is helpful to observe that we can work over the rational numbers $\mathbb{Q}$ in order to perform the relevant computations. This is advantageous because the computations over $\mathbb{Q}$ allow us to deduce results over $K$, which are independent of the choice of (odd) characteristic $p$. 

Let us explain why we can work over $\mathbb{Q}$ for computations when $p$ is odd. To see this, first let $\bar{G}_{\mathbb{Q}}$ be a simply connected Chevalley group over $\mathbb{Q}$ of the same type as $\bar{G}_{{\rm sc}}$, defined using the same structure constants as $\bar{G}_{{\rm sc}}$. For each $\a \in \Phi$ and $t \in \Q$, let $x_{\alpha}^{\Q}(t)$ be the corresponding root element in $\bar{G}_{\mathbb{Q}}$ and let $V_{\mathbb{Q}}$ be the $\bar{G}_{\mathbb{Q}}$-module with the same highest weight as $V$. Let $\bar{G}_{\Z}$ be the subgroup of $\bar{G}_{\mathbb{Q}}$ generated by the set $\{ x_{\alpha}^{\mathbb{Q}}(t) \,:\, \a \in \Phi, \, t \in \Z\}$. From the Chevalley construction, $V_{\Q}$ contains a $\bar{G}_{\Z}$-invariant lattice $V_{\Z}$.  Furthermore, there exists a homomorphism $$\pi: \bar{G}_{\Z} \rightarrow \bar{G}_{{\rm sc}}$$ defined by $x_{\alpha}^{\Q}(t) \mapsto x_{\alpha}'(t)$ for all $\alpha \in \Phi$ and $t \in \Z$. 

Let $\rho_{\Z}: \bar{G}_{\Z} \rightarrow \GL(V_{\Z})$ be the representation (over $\Z$) corresponding to $V_{\Z}$, and let $\rho: \bar{G}_{{\rm sc}} \rightarrow \GL(V)$ be the representation (over $K$) corresponding to $V$. From the construction of $\bar{G}_{{\rm sc}}$ via reduction modulo $p$ \cite[Chapter 3]{SteinbergNotes}, we can identify $V = K \otimes_{\Z} V_{\Z}$ and we have 
\[
\rho(\pi(x)) = {{\rm Id}}_K \otimes \rho_{\Z}(x)
\]
for all $x \in \bar{G}_{\Z}$.
 
Recall that the element $g' \in \bar{G}_{{\rm sc}}$ we are interested in can be expressed as a product of root elements of the form $x_{\alpha}'(\pm 1)$, say
\[
g' = x_{\beta_1}'(c_1) \cdots x_{\beta_t}'(c_t)
\]
for some roots $\beta_i \in \Phi$ and integers $c_i \in \{1,-1\}$. Now define 
\[
g_{\mathbb{Q}} = x_{\beta_1}^{\mathbb{Q}}(c_1) \cdots x_{\beta_t}^{\mathbb{Q}}(c_t)
\] 
and note that $g' = \pi(g_{\Q})$. 

The following result now justifies our calculations over $\Q$. In the statement, $V$ is an arbitrary Weyl module $V_{\bar{G}}(\l)$.

\begin{lem}\label{l:rationalcompute}
Assume that $p > 2$ and that $g_{\Q}$ acts as an involution on $V_{\Q}$. Then $g' = \pi(g_{\Q})$ acts as an involution on $V$ and we have $\dim C_V(g') = \dim C_{V_{\mathbb{Q}}}(g_{\mathbb{Q}})$.
\end{lem}

\begin{proof}
First observe that $g_{\Q} \in \bar{G}_{\Z}$, so the lattice $V_{\Z}$ is $g_{\Q}$-invariant. Since $g_{\Q}$ acts as an involution, it is well known (see for example \cite[Theorem 74.3]{CurtisReiner} or \cite[Appendix A]{BorovoiTimashev}) that there is a basis for $V_{\Z}$ such that the corresponding matrix for $g_{\Q}$ is block-diagonal of the form $(I_a, -I_b, A^c)$ for some non-negative integers $a,b$ and $c$, where $A$ is the $2 \times 2$ anti-diagonal matrix ${\rm antidiag}(1,1)$.

Now $\rho(g') = {\rm Id}_K \otimes \rho_{\Z}(g_{\mathbb{Q}})$, so the action of $g'$ on $V$ is given by the reduction modulo $p$ of this matrix. Since the matrix $A$ is similar to ${\rm diag}(1,-1)$ over any field of characteristic $\neq 2$, in particular over $K$ and over $\Q$, it is clear that $g'$ also acts as an involution on $V$. Moreover, we have $\dim C_V(g') = \dim C_{V_{\Q}}(g_{\Q})$ as required.
\end{proof}

\begin{ex}
Let $\bar{G} = E_7$ with $p > 2$ and consider the following elements:
\[
g_1 = h_{\alpha_1}(-1), \;\; g_2 = w_{\alpha_2} w_{\alpha_5} w_{\alpha_7}, \;\; g_3 = h_{\alpha_1}(-1) w_{\alpha_2} w_{\alpha_5} w_{\alpha_7}.
\]
It is clear that each $g_i$ acts as an involution on the adjoint module $V = \mathcal{L}(\bar{G})$. By appealing to Lemma \ref{l:rationalcompute}, the following {\sc Magma} code shows that $\dim C_V(g_i) = 69,79,63$ for $i = 1,2,3$ respectively: 

\vs

{\small
\begin{verbatim}
G := GroupOfLieType("E7", RationalField() : Isogeny := "SC");
g1 := TorusTerm(G,1,-1);
g2 := elt<G|2> * elt<G|5> * elt<G|7>;
g3 := g1*g2;
r := AdjointRepresentation(G);
A1 := Matrix(r(g1)); // action of g1 on Lie(G)
A2 := Matrix(r(g2)); // action of g2 on Lie(G)
A3 := Matrix(r(g3)); // action of g3 on Lie(G)
Dimension(Kernel(A1-1)); // output: 69
Dimension(Kernel(A2-1)); // output: 79
Dimension(Kernel(A3-1)); // output: 63
\end{verbatim}
}

\vs

\noindent As we will explain in the next section, this computation allows us to conclude that the involutions $g_1$, $g_2$ and $g_3$ are contained in the $\bar{G}$-classes of type $A_1D_6$, $E_6T_1$ and $A_7$, respectively.
\end{ex}

\subsection{Involutions}\label{ss:invols} 

Recall that our main aim is to classify the almost simple primitive groups with point stabilizer $H$ and socle $T$, an exceptional group of Lie type, with the property that $T$ is $2$-elusive. This is essentially equivalent to determining the pairs $(T,H_0)$, where $T$ is a simple exceptional group of Lie type and $H_0 = H \cap T$ for some core-free maximal subgroup $H$ of a group with socle $T$ such that $H_0$ intersects every $T$-class of involutions in $T$. So in order to study this problem, we require detailed information on both the conjugacy classes of involutions in simple exceptional groups, as well as the maximal subgroups of the almost simple exceptional groups. In this section, we focus on the involution classes, and we will turn to the maximal subgroups in Section \ref{ss:subs}.

The study of involutions divides naturally into two cases, according to the parity of the underlying characteristic $p$. Of course, if $p = 2$ then the involutions are unipotent elements of $\bar{G}$, while for $p \neq 2$ they are semisimple.  There are some significant differences between these two cases, so in the proofs of our main results, we will often treat the cases $p = 2$ and $p \neq 2$ separately.

We first consider some of the small rank exceptional groups.

\begin{lem}\label{l:b2g2classes}
If $T = {}^2B_2(q)$ or ${}^2G_2(q)'$, then $T$ has a unique conjugacy class of involutions.
\end{lem}

\begin{proof}
The fact that ${}^2B_2(q)$ and ${}^2G_2(q)$ have a unique conjugacy class of involutions was originally proved in \cite[Proposition 7]{Suzuki} and \cite[p.63]{Ward}, respectively. This establishes the lemma except for ${}^2G_2(3)' \cong \mathrm{L}_2(8)$, in which case the result is clear.
\end{proof}

Next we consider the Steinberg triality group. We refer the reader to Table \ref{table:3D4q} for further information on the two classes of unipotent involutions when $p=2$ (in the table, $V$ denotes the natural $8$-dimensional module for $\bar{G} = D_4$).

\begin{lem}\label{l:3d4qclasses}
Suppose that $T = {}^3D_4(q)$.

\vspace{1mm}

\begin{itemize}\addtolength{\itemsep}{0.2\baselineskip}
\item[{\rm (i)}] If $q$ is odd, then $T$ has a unique conjugacy class of involutions.
\item[{\rm (ii)}] If $q$ is even, then $T$ has two conjugacy classes of involutions, labelled $A_1$ and $A_1^3$.
\end{itemize}
\end{lem}

\begin{proof}
As noted in \cite[Lemma 2.3(i)]{K}, claim (i) follows from results in \cite{GorensteinHarada}. Alternatively, one can argue directly as follows: In $\bar{G} = D_4 = {\rm SO}_8(K) / \langle \pm I_8 \rangle$, each involution is conjugate to the image of a diagonal matrix in ${\rm SO}_8(K)$, and a calculation shows that there are $4$ conjugacy classes of involutions in $\bar{G}$, only one of which is invariant under a triality graph automorphism of $\bar{G}$. Finally, for part (ii) we refer to \cite{Spaltenstein}.
\end{proof}

{\scriptsize
\begin{table}
\[
\begin{array}{llll} \\ \hline
\mbox{Class} & \mbox{Representative} & |C_T(x)| & V \downarrow x \\ \hline
A_1 & x_{\alpha_1}(1) & q^{12}(q^6 - 1) & (2^2, 1^4) \\
A_1^3 & x_{\alpha_1}(1)x_{\alpha_3}(1)x_{\alpha_4}(1) & q^{10}(q^2 - 1) & (2^4) \\\hline
\end{array}
\]
\caption{The involution classes in $T = {}^3D_4(q)$ for $p=2$} 
\label{table:3D4q}
\end{table}
}

In each of the remaining cases, the ambient algebraic group $\bar{G}$ is simple of exceptional type. We begin by discussing the conjugacy classes of involutions in $\bar{G}$. 

First assume $p = 2$. Here the conjugacy classes of involutions in $\bar{G}$ are recorded in Table \ref{table:evenpinvolutions}, where we adopt the labelling of the classes from the tables in \cite[Chapter 22]{LS_book}. The information in Table \ref{table:evenpinvolutions} is verified as follows. Firstly, the Jordan forms on $V_{\mathrm{min}}$ and $\mathcal{L}(\bar{G})$ can be read off from the tables in \cite{Law} (see Remark \ref{r:nota} below). As a consequence, we observe that the $\bar{G}$-class of a unipotent involution is uniquely determined by its Jordan form on $V_{\mathrm{min}}$ and $\mathcal{L}(\bar{G})$. We will often use this fact to determine the class of a given involution in $\bar{G}$. 

In most cases, the representatives listed in Table \ref{table:evenpinvolutions} are standard ones that can be found in the literature. In any case, by computing the action of the given elements on $V_{\mathrm{min}}$ and $\mathcal{L}(\bar{G})$ (for example, with the aid of {\sc Magma}, as described in Section \ref{ss:comp}), one can verify that the representatives listed in Table \ref{table:evenpinvolutions} are correct by inspecting \cite{Law}. In the cases where $\bar{G}$ admits a graph automorphism $\tau$ (or an exceptional isogeny $\psi$), we have also indicated in Table \ref{table:evenpinvolutions} whether or not the given class is invariant under $\tau$ (or $\psi$). This information is clear from the table, since the representatives listed are either fixed by $\tau$ (or $\psi$), or mapped to another representative.

{\scriptsize
\begin{table}
\[
\begin{array}{llllll} \hline
\bar{G} & \mbox{Class} & \mbox{Representative} & V_{\mathrm{min}} \downarrow x & \mathcal{L}(\bar{G}) \downarrow x & \mbox{Notes} \\ \hline
G_2 & \tilde{A_1} & x_{\alpha_1}(1)  & (2^3, 1) & (2^6, 1^2) & \\
& A_1 & x_{\alpha_2}(1) & (2^2, 1^3) & (2^6, 1^2) & \\
&&&&&\\
F_4 & A_1 & x_{\alpha_1}(1) & (2^6, 1^{14}) & (2^{16}, 1^{20}) & \\
& \tilde{A_1} & x_{\alpha_4}(1) & (2^{10}, 1^6) & (2^{16}, 1^{20}) & \\
& A_1\tilde{A_1} & x_{\alpha_1}(1)x_{\alpha_4}(1) & (2^{12}, 1^2) & (2^{24}, 1^4) & \mbox{$\psi$-invariant} \\
& (\tilde{A_1})_2 & x_{\alpha_2+\alpha_3}(1) x_{\alpha_2+2\alpha_3}(1) & (2^{10}, 1^6) & (2^{21}, 1^{10}) & \mbox{$\psi$-invariant} \\ 
&&&&&\\
E_6 & A_1 & x_{\alpha_2}(1) & (2^6, 1^{15}) & (2^{22}, 1^{34}) & \mbox{$\tau$-invariant} \\
& A_1^2 & x_{\alpha_1}(1)x_{\alpha_6}(1) & (2^{10},1^7) & (2^{32}, 1^{14}) & \mbox{$\tau$-invariant} \\
& A_1^3 & x_{\alpha_1}(1)x_{\alpha_2}(1)x_{\alpha_6}(1) & (2^{12}, 1^3) & (2^{38}, 1^2) & \mbox{$\tau$-invariant} \\
&&&&&\\
E_7 & A_1 & x_{\alpha_1}(1) & (2^{12}, 1^{32}) & (2^{34}, 1^{65}) & \\
 & A_1^2 & x_{\alpha_1}(1)x_{\alpha_2}(1) & (2^{20}, 1^{16}) & (2^{52}, 1^{29}) & \\
 & (A_1^3)^{(1)} & x_{\alpha_2}(1)x_{\alpha_5}(1)x_{\alpha_7}(1) & (2^{28}) & (2^{53}, 1^{27}) & \\
 & (A_1^3)^{(2)} & x_{\alpha_3}(1)x_{\alpha_5}(1)x_{\alpha_7}(1) & (2^{24}, 1^8)      & (2^{62}, 1^9) & \\
& A_1^4 & x_{\alpha_2}(1)x_{\alpha_3}(1)x_{\alpha_5}(1)x_{\alpha_7}(1) & (2^{28})           & (2^{63}, 1^7) & \\
&&&&&\\
E_8 & A_1 & x_{\alpha_1}(1) & & (2^{58}, 1^{132}) \\
 & A_1^2 & x_{\alpha_1}(1)x_{\alpha_4}(1) & & (2^{92}, 1^{64}) \\
 & A_1^3 & x_{\alpha_1}(1)x_{\alpha_4}(1)x_{\alpha_6}(1) & & (2^{110}, 1^{28}) \\
 & A_1^4 & x_{\alpha_1}(1)x_{\alpha_4}(1)x_{\alpha_6}(1)x_{\alpha_8}(1) & & (2^{120}, 1^8) \\ \hline
\end{array}
\]
\caption{The involution classes in $\bar{G}$ of exceptional type, $p = 2$}\label{table:evenpinvolutions}
\end{table}
}

\begin{rem}\label{r:nota}
In Table \ref{table:evenpinvolutions}, we use the labelling of unipotent classes from \cite[Chapter 22]{LS_book} and we note that this differs slightly from the labels used by Lawther in \cite{Law}, which is our main reference for Jordan block sizes. Specifically, to avoid any confusion, we note that for $p = 2$ and $\bar{G} = E_7$, the class $(A_1^3)^{(1)}$ is denoted by $(3A_1)''$ in \cite{Law}, and the class $(A_1^3)^{(2)}$ is denoted by $(3A_1)'$ in \cite{Law}.
\end{rem}

Now assume $p > 2$. Here the conjugacy classes of involutions in $\bar{G}$ are listed in Table \ref{table:oddpinvolutions}, following \cite[Table 4.3.1]{GLS}. In each case, the class of an involution $x$ is labelled by the structure of the connected component $C_{\bar{G}}(x)^{\circ}$, which for exceptional $\bar{G}$ is uniquely determined by the dimension of the centralizer. Since each involution is semisimple, we have 
\[
\dim C_{\bar{G}}(x) = \dim C_{\mathcal{L}(\bar{G})}(x)
\]
by \cite[9.1]{Borel}, so we can determine the $\bar{G}$-class of $x$ by calculating its action on the adjoint module $\mathcal{L}(\bar{G})$. A similar observation is made in \cite[Proposition 1.2]{LS99}. As we did for the case $p = 2$, we also list explicit representatives for each class of involutions in Table \ref{table:oddpinvolutions}, and we indicate the classes that are invariant under the relevant maps $\tau$ and $\psi$. 

{\scriptsize
\begin{table}
\[
\begin{array}{lllccl} \hline
\bar{G} & \mbox{Class} & \mbox{Representative} & \dim C_{V_{\mathrm{min}}}(x) & \dim C_{\mathcal{L}(\bar{G})}(x) & \mbox{Notes} \\ \hline
G_2 & A_1 \tilde{A_1} & h_{\alpha_1}(-1)h_{\alpha_2}(-1) & 3 & 6 & \mbox{$\psi$-invariant if $p = 3$} \\
&&&&&\\
F_4 & A_1C_3 & h_{\alpha_1}(-1) & 14 & 24 & \\
& B_4 & h_{\alpha_4}(-1) & 10 & 36 & \\
&&&&&\\
E_6 & A_1A_5 & h_{\alpha_2}(-1) & 15 & 38 & \mbox{$\tau$-invariant} \\
& D_5T_1 & h_{\alpha_1}(-1)h_{\alpha_6}(-1) & 11 & 46 & \mbox{$\tau$-invariant} \\
&&&&&\\
E_7 & A_1D_6 & h_{\alpha_1}(-1) & \mbox{$32$ or $24$} & 69 & \\
& E_6T_1 & w_{\alpha_2} w_{\alpha_5} w_{\alpha_7}  & 0 & 79 & \\
& A_7 & h_{\alpha_1}(-1)w_{\alpha_2} w_{\alpha_5} w_{\alpha_7} & 0 & 63 & \\
&&&&&\\
E_8 & A_1E_7 & h_{\alpha_1}(-1) & & 136 & \\
& D_8 & h_{\alpha_1}(-1)h_{\alpha_2}(-1) & & 120 & \\ \hline
\end{array}
\]
\caption{The involution classes in $\bar{G}$ of exceptional type, $p \ne 2$}\label{table:oddpinvolutions}
\end{table}
}

\begin{rem}\label{r:e77}
Recall that our simple algebraic group $\bar{G}$ is of adjoint type, which means that $V_{\mathrm{min}}$ is not necessarily a $\bar{G}$-module when $\bar{G}$ is of type $E_6$ or $E_7$. However, in these cases we can lift each involution $g \in \bar{G}$ to an element $g' \in \bar{G}_{{\rm sc}}$ of order $2$ or $4$ in the simply connected cover of $\bar{G}$, and the action on $V_{\mathrm{min}}$ presented in Tables \ref{table:evenpinvolutions} and \ref{table:oddpinvolutions} corresponds to the action of $g'$. More precisely, $g'$ can be chosen to be an involution, except for elements in the classes labelled $E_6T_1$ and $A_7$ in $\bar{G} = E_7$ with $p \ne 2$. In these two cases, $g'$ has order $4$ and 
\[
(g')^2 = z = h_{\alpha_2}'(-1) h_{\alpha_5}'(-1) h_{\alpha_7}'(-1)
\]
generates the center of $\bar{G}_{{\rm sc}}$. (Here $h_{\alpha}'(t)$ is defined in the same way as $h_{\alpha}(t)$ for $\bar{G}$.) 

In addition, we note that if $\bar{G} = E_7$ and $p \ne 2$, then there are two possible choices for the lift $g'$, namely $g'$ and $g'z$. Since $z$ acts as $-I_{56}$ on $V_{\rm min}$, the value of $\dim C_{V_{\rm min}}(g')$ may depend on the choice of lift $g'$. This issue occurs only for involutions of type $A_1D_6$ in $\bar{G} = E_7$, as indicated in Table \ref{table:oddpinvolutions}. 
\end{rem}

We can now use the following lemma to describe the conjugacy classes of involutions in $\bar{G}_{\s}$ in terms of the classes in $\bar{G}$. Note that the lemma also includes the cases $T = {}^2B_2(q)$ and $T = {}^3D_4(q)$.

\begin{lem}\label{l:Gsigmaclasses}
Assume that $T \neq G_2(2)'$, and let $g \in \bar{G}$ be an involution. Then the following statements hold:

\vspace{1mm}

\begin{itemize}\addtolength{\itemsep}{0.2\baselineskip}
\item[{\rm (i)}] $g^{\bar{G}} \cap T$ is nonempty if and only if $g^{\bar{G}}$ is $\s$-invariant.
\item[{\rm (ii)}] If $T$ is untwisted, then $g^{\bar{G}} \cap T$ is nonempty.
\item[{\rm (iii)}] If $g^{\bar{G}} \cap T$ is nonempty, then $g^{\bar{G}} \cap T$ consists of a single $T$-class.
\end{itemize}
\end{lem}

\begin{proof}
First assume $T \in \{{}^2G_2(3)', {}^2F_4(2)'\}$. For $T = {}^2F_4(2)'$ it is easy to check that $T$ has two conjugacy classes of involutions, which belong to the $\bar{G}$-classes labelled $A_1\tilde{A_1}$ and $(\tilde{A_1})_2$. So the result holds in this case (see Table \ref{table:evenpinvolutions}). And for $T = {}^2G_2(3)'$, the result is clear since $T$ has a unique conjugacy class of involutions (Lemma \ref{l:b2g2classes}).

In the remaining cases $T \cong (\bar{G}_{{\rm sc}})_{\s} / Z(\bar{G}_{{\rm sc}})_{\s}$, where $\bar{G}_{{\rm sc}}$ is the simply connected cover of $\bar{G}$. Claim (i) follows from \cite[I, 2.7(a)]{SpringerSteinberg}. Claim (ii) is true more generally, but for the purposes of this proof it is sufficient to observe that each class representative given in Tables \ref{table:evenpinvolutions} and \ref{table:oddpinvolutions} can be written as a product of root elements of the form $x_{\alpha}(\pm 1)$. Then part (ii) follows since any such root element is fixed by a Frobenius endomorphism.

For (iii) in the case $p = 2$, we refer to the tables in \cite[Chapter 22]{LS_book} when $\bar{G}$ is of exceptional type; alternatively, see \cite{AS}. For $p = 2$ and $T \in \{{}^2B_2(q), {}^3D_4(q)\}$, the claim follows from Lemmas \ref{l:b2g2classes} and \ref{l:3d4qclasses}. And for $p > 2$, we refer to \cite[Table 4.5.1]{GLS}, noting that $x^T = x^{\bar{G}_{\s}}$ for each involution $x \in T$ (see \cite[Theorem 4.2.2(j)]{GLS}). 
\end{proof}

To summarize the results in the case where $\bar{G}$ is of exceptional type and $T \ne G_2(2)'$, Lemma \ref{l:Gsigmaclasses} implies that the conjugacy classes of involutions in $T$ are as described in Table \ref{table:evenpinvolutions} (for $p=2$) and Table \ref{table:oddpinvolutions} (for $p \ne 2$). Therefore, we can determine the $T$-class of each involution in $T$ just by computing the dimension of its fixed point space on $V_{\mathrm{min}}$ or $\mathcal{L}(\bar{G})$. Note that one only needs to do this computation for both modules when $(\bar{G},p) = (F_4,2)$. Finally, for the record we note that $G_2(2)' \cong {\rm U}_3(3)$ has a unique class of involutions.

\begin{rem}\label{r:TvsInnT}
For the most part, there is no distinction between classes of involutions in the socle $T = (\bar{G}_{\s})'$ and the almost simple group $\widetilde{G} = {\rm Inndiag}(T) = \bar{G}_{\s}$. Clearly, if the index $[\widetilde{G}:T]$ is odd, then every involution in $\widetilde{G}$ is contained in $T$. And for $T$ exceptional, $[\widetilde{G}:T]$ is even if and only if $T = G_2(2)'$, ${}^2F_4(2)'$, or $E_7(q)$ with $q$ odd.

\vspace{1mm}

\begin{itemize}\addtolength{\itemsep}{0.2\baselineskip}
\item[{\rm (a)}] If $T = G_2(2)'$ then $\widetilde{G}$ has two classes of involutions, only one of which is contained in $T$. More precisely, the involutions in $T$ are those of type $A_1$ (long root elements), while the class in $\widetilde{G} \setminus T$ consists of involutions of type $\tilde{A_1}$ (short root elements). 

\item[{\rm (b)}] If $T = {}^2F_4(2)'$, then every involution in $\widetilde{G} = {}^2F_4(2)$ is contained in $T$.

\item[{\rm (c)}] Now suppose $T = E_7(q)$ and $q$ is odd. Here there are three $\bar{G}$-classes of involutions, with corresponding centralizers $A_1D_6$, $E_6T_1.2$ and $A_7.2$, and according to \cite[Table 4.5.1]{GLS} there are three classes in $T$ and five in $\widetilde{G}$. More precisely, there is a single $\widetilde{G}$-class of involutions of type $A_1D_6$, which is contained in $T$. Each of the two remaining $\bar{G}$-classes correspond to two $\widetilde{G}$-classes, only one of which is contained in $T$ (the splitting of $x^{\bar{G}}$ into two $\widetilde{G}$-classes corresponds to the fact that $[C_{\bar{G}}(x): C_{\bar{G}}(x)^{\circ}] = 2$). For example, if $q \equiv \e \imod{4}$, then the involutions $x \in \tilde{G}$ with $|C_{\widetilde{G}}(x)| = 2|{\rm SL}^{\e}_8(q)|$ are contained in $T$, while those with $|C_{\widetilde{G}}(x)| = 2|{\rm SL}^{-\e}_8(q)|$ are in $\widetilde{G} \setminus T$.  
\end{itemize}
\end{rem}

\begin{rem}\label{t:inv_fuse}
Let $x \in T$ be an involution and set $A = {\rm Aut}(T)$. By considering the above description of the conjugacy classes of involutions in $T$, it is not difficult to show that 
either 

\vspace{1mm}

\begin{itemize}\addtolength{\itemsep}{0.2\baselineskip}
\item[{\rm (a)}] $x^A = x^T$, or 

\item[{\rm (b)}] $T = F_4(q)$, $p=2$ and $x$ is in one of the $T$-classes labelled $A_1$ or $\tilde{A}_1$ in Table \ref{table:evenpinvolutions}, in which case $x^A = y^T \cup z^T$ with $y \in A_1$ and $z \in \tilde{A}_1$. 
\end{itemize}
\end{rem}

To conclude our discussion of involutions, let us consider the cases where $T$ has a unique class of involutions. As explained above, this property holds if and only if $T$ is one of the following:
\[
{}^2B_2(q), \; {}^2G_2(q)',\; {}^3D_4(q) \mbox{ ($q$ odd)}, \; G_2(q) \mbox{ ($q$ odd)}.
\]

\begin{lem}\label{l:uniqueclasselusive}
Suppose that $T$ has a unique conjugacy class of involutions. Then $T$ is $2$-elusive if and only if $|\O|$ is even.
\end{lem}

\begin{proof}
This follows immediately from \cite[Lemma 2.1]{LSh}, which implies that $|H \cap T|$ is even for every maximal subgroup $H$ of $G$.
\end{proof}

\begin{rem}
We refer the reader to Theorem \ref{t:odd} below for a complete description of the relevant groups with $|\O|$ odd. 
\end{rem}

Now suppose that $T$ has two or more conjugacy classes of involutions. Clearly, if there are more conjugacy classes of involutions in $T$ than in $H_0$, then $T$ is not $2$-elusive. With this simple observation in mind, we record the following elementary result, which will be useful later.

\begin{lem}\label{l:conj}
Let $L$ be a finite group with a normal subgroup $A$ of odd order. Then $L$ and $L/A$ have the same number of conjugacy classes of involutions.
\end{lem}

\begin{proof}Let $t_1A,\ldots, t_kA$ be representatives for the conjugacy classes of involutions in $L/A$. Since $A$ has odd order, we can assume that each $t_i$ is an involution in $L$. To prove the lemma, it will suffice to show that each involution in $L$ is conjugate to a unique $t_i$. To this end, let $x \in L$ be an involution. Then $xA$ is conjugate to $t_iA$ for a unique $i$, so $x$ is conjugate to $t_i a$ for some $a \in A$. Set $J = \la A, t_i \ra = A{:}\la t_i \ra = A{:}\la t_ia \ra$. Then $\la t_i \ra$ and $\la t_ia \ra$ are $J$-conjugate by Sylow's theorem and the result follows.\end{proof}

As a special case of Lemma \ref{l:conj}, note that if $L/A$ has a unique conjugacy class of involutions, then $L$ also has a unique class of involutions.

\begin{rem}\label{r:inv_class}
In the proof of Theorem \ref{t:main}, we will also need some information on the conjugacy classes of involutions in the finite classical groups. So let $L$ be an almost simple classical group over $\mathbb{F}_q$ with socle $S$ and let $\widetilde{L} = {\rm Inndiag}(S)$ be the subgroup of ${\rm Aut}(S)$ generated by the inner and diagonal automorphisms of $S$.

\vspace{1mm}

\begin{itemize}\addtolength{\itemsep}{0.2\baselineskip}
\item[{\rm (a)}] If $q$ is odd, then detailed information on the conjugacy classes of involutions in $\widetilde{L}$ and $S$ is presented in \cite[Table 4.5.1]{GLS}.

\item[{\rm (b)}] In even characteristic, Aschbacher and Seitz \cite{AS} provide an in-depth analysis of the involution classes in all groups of Lie type, see also \cite[Chapters 6 and 7]{LS_book}. For classical groups, each involution $x \in \widetilde{L}$ is contained in $S$ (unless $\widetilde{L} = {\rm Sp}_4(2) = S.2$) and we observe that $x$ has Jordan form $(2^{k},1^{n-2k})$ on the natural module $V$ for $S$, where $n = \dim V$ and $1 \leqs k \leqs n/2$. 

For linear and unitary groups, the $S$-class and $\widetilde{L}$-class of $x$ coincide, and it is uniquely determined by $k$. Similarly, if $S$ is a symplectic group, then there is a unique such class for $k$ odd, denoted $b_k$ in \cite{AS}, whereas there are two classes for each even $k$, labelled $a_k$ and $c_k$ (the $a$-type involutions have the property that $\b(v,v^x) = 0$ for all $v \in V$, where $\b$ is the defining symplectic form on $V$).

If $S = \O_{n}^{\e}(q)$ is an even-dimensional orthogonal group, then $k$ is even and there are two classes when $k<n/2$, denoted $a_k$ and $c_k$. If $k = n/2$, then there is a unique class $c_{n/2}$ when $\e=-$, whereas there are three such classes when $\e=+$, denoted by $a_{n/2}$, $a_{n/2}'$ and $c_{n/2}$ (here the first two $S$-classes are fused in the full isometry group $S.2 = {\rm O}_{n}^{+}(q)$). Note that in this case we have $S < \operatorname{Sp}_n(q)$, and for all $k$ even the class $a_k$ (respectively $c_k$) corresponds to the class denoted by $a_k$ (respectively $c_k$) in $\operatorname{Sp}_n(q)$.

\item[{\rm (c)}] If $p=2$ and $S$ is a symplectic or orthogonal group, then an involution $x \in S$ acts as a unipotent element on the natural module $V$ of $S$. In terms of the \emph{distinguished normal form} used to describe the unipotent classes in \cite{LS_book}, we note that $V \downarrow x$ admits the following orthogonal decomposition according to the $S$-class of $x$ as described in part (b):
\begin{align*}
&a_k:\ V \downarrow x = W(1)^{n/2-k} \perp W(2)^{k/2},  \\ 
&b_k:\ V \downarrow x = W(1)^{n/2-k} \perp W(2)^{(k-1)/2} \perp V(2), \\ 
&c_k:\  V \downarrow x = W(1)^{n/2-k} \perp W(2)^{(k-2)/2} \perp V(2)^2. 
\end{align*} 
For further details, and a description of the summands arising here, we refer the reader to \cite[Chapters 6 and 7]{LS_book} (and Theorem 7.3 in particular).
\end{itemize}
\end{rem}

\subsection{Subgroup structure}\label{ss:subs}

The purpose of this section is recall some of the main results from the literature on maximal subgroups of almost simple exceptional groups of Lie type. 

As before, let $G$ be an almost simple exceptional group of Lie type over $\mathbb{F}_q$ with socle $T = (\bar{G}_{\s})'$, where $q = p^f$ and $p$ is a prime. Let $\mathcal{M} = \mathcal{C} \cup \mathcal{S}$ be the set of core-free maximal subgroups of $G$, where $\mathcal{C}$ is the collection of maximal subgroups $H$ of the following types (see Definition \ref{d:c}):

\vspace{1mm}

\begin{itemize}\addtolength{\itemsep}{0.2\baselineskip}
\item[{\rm (I)}] $H = N_G(\bar{H}_{\s})$ for some maximal $\sigma$-invariant positive-dimensional closed subgroup $\bar{H}$ of $\bar{G}$;
\item[{\rm (II)}] $H$ is of the same type as $G$ (either a subfield subgroup or a twisted version of $G$);
\item[{\rm (III)}] $H$ is the normalizer of an exotic $r$-local subgroup for some prime $r \neq p$;
\item[{\rm (IV)}] $T = E_8(q)$, $p > 5$, and ${\rm soc}(H) = {\rm Alt}_5 \times {\rm Alt}_6$ (the \emph{Borovik subgroup}).
\end{itemize}

\vspace{1mm}

Here the subgroups arising in (I) and (II) are described in \cite{LS04}. The maximal subgroups in case (III) were classified (up to conjugacy) in \cite{CLSS}, and the subgroup in case (IV) was first described by Borovik \cite{Bor}. By combining results of Borovik \cite{Bor} and Liebeck--Seitz \cite{LS90}, it follows that the maximal subgroups in $\mathcal{S}$ are almost simple.

\begin{rem}\label{r:CS}
Through the work of many authors, spanning several decades, the maximal subgroups of $G$ have been determined (up to conjugacy) when $T \ne E_7(q), E_8(q)$. For example, Craven's recent paper \cite{Craven} completely classifies the maximal subgroups for the groups with socle $T = F_4(q)$, $E_6(q)$ or ${}^2E_6(q)$. Since our main result (Theorem \ref{t:main}) is stated in terms of the collections $\mathcal{C}$ and $\mathcal{S}$, here we provide some more details on the subgroups arising in each collection:

\vspace{1mm}

\begin{itemize}\addtolength{\itemsep}{0.2\baselineskip}
\item[{\rm (a)}] $T = {}^2B_2(q)$ or ${}^2G_2(q)'$: The maximal subgroups were determined up to conjugacy by Suzuki \cite{Suzuki} and Kleidman \cite{K0}, in the two respective cases. For $T \ne {}^2G_2(3)'$, the subgroups in $\mathcal{M}$ are conveniently recorded in \cite[Tables 8.16 and 8.43]{BHR} and we note that $\mathcal{S}$ is empty. And if $T = {}^2G_2(3)' \cong {\rm L}_2(8)$, then every core-free maximal subgroup of $G$ is solvable and so once again $\mathcal{S}$ is empty.

\item[{\rm (b)}] $T = {}^3D_4(q)$: The subgroups in $\mathcal{M}$ were determined up to conjugacy by Kleidman \cite{K} and they are listed in \cite[Table 8.51]{BHR}. The collection $\mathcal{S}$ is empty for all $q$.

\item[{\rm (c)}] $T = {}^2F_4(q)'$: First assume $q \geqs 8$. The subgroups in $\mathcal{M}$ were determined up to conjugacy by Malle \cite{Malle}, modulo the omission of three classes of maximal subgroups isomorphic to ${\rm PGL}_2(13)$ when $G = {}^2F_4(8)$ (see \cite[Remark 4.11]{Craven}). Here the collection $\mathcal{S}$ is empty unless $G = {}^2F_4(8)$, in which case it comprises the aforementioned maximal subgroups isomorphic to ${\rm PGL}_2(13)$.

For $q=2$, the conjugacy classes of maximal subgroups of $G$ were determined by Wilson \cite{Wil84}, noting the omission of a unique class of maximal subgroups ${\rm SU}_3(2).2$ of ${}^2F_4(2)$. We see that $\mathcal{S}$ contains five classes of subgroups when $G = T$, which are isomorphic to ${\rm L}_3(3).2$ (two classes), $A_6.2^2$ (two classes) and ${\rm L}_2(25)$. And for $G = T.2$ we note that $\mathcal{S}$ comprises a unique class of subgroups isomorphic to ${\rm L}_2(25).2$. 

\item[{\rm (d)}] $T= G_2(q)'$: The maximal subgroups of $G$ were determined up to conjugacy by Cooperstein \cite{Coop} (for $p=2$) and Kleidman \cite{K0} (for $p \ne 2$), and they are conveniently listed in \cite[Table 8.30]{BHR} (for $p=2$ and $q \geqs 4$) and \cite[Tables 8.41, 8.42]{BHR} (for $p \ne 2$). 

If $p=2$, then the collection $\mathcal{S}$ is empty unless $q=4$, in which case it comprises two classes of subgroups $H$ with $H \cap T = {\rm L}_2(13)$ or ${\rm J}_2$. Note that if $q=2$ then $T = G_2(2)' \cong {\rm U}_3(3)$ and $G$ has a unique class of nonsolvable maximal subgroups $H$ with $H \cap T = {\rm L}_3(2)$, which are contained in the collection $\mathcal{C}$. Similarly, if $p=3$ then $\mathcal{S}$ is empty unless $q=3$, in which case it contains a class of subgroups $H$ with $H \cap T = {\rm L}_2(13)$. Finally, if $p \geqs 5$ then the subgroups in $\mathcal{S}$ can be read off from \cite[Table 8.41]{BHR} (the relevant cases are labelled $\mathcal{S}_1$ in the first column of \cite[Table 8.41]{BHR}).

\item[{\rm (e)}] $T = F_4(q)$, $E_6(q)$ or ${}^2E_6(q)$: The subgroups in $\mathcal{M}$ have been determined up to conjugacy by Craven in \cite{Craven}. More specifically, the subgroups in $\mathcal{S}$ are recorded in Tables 1, 2 and 3 of \cite{Craven}, while the subgroups comprising $\mathcal{C}$ are presented in \cite[Tables 7 and 8]{Craven} for $T = F_4(q)$ and \cite[Tables 9 and 10]{Craven} for $T = E_6(q)$ and ${}^2E_6(q)$, respectively.

\item[{\rm (f)}] $T = E_7(q)$: The subgroups in $\mathcal{C}$ are recorded up to conjugacy in \cite[Table 4.1]{Craven2}. By the main theorem of \cite{Craven2}, the collection $\mathcal{S}$ comprises the subgroups in \cite[Table 1.1]{Craven2}, which are described up to conjugacy, together with the possible inclusion of some additional subgroups with socle ${\rm L}_2(q')$ for $q' \in \{7,8,9,13\}$.

\item[{\rm (g)}] $T = E_8(q)$: Here the collection $\mathcal{C}$ comprises $8$ conjugacy classes of maximal parabolic subgroups, together with the maximal rank subgroups of the form $N_G(\bar{H}_{\s})$ recorded in \cite[Tables 5.1, 5.2]{LSS}. In addition, $\mathcal{C}$ contains several classes of non-maximal rank subgroups of the form $N_G(\bar{H}_{\s})$, where the possibilities for ${\rm soc}(\bar{H}_{\s})$ are listed in \cite[Table 3]{LS03}, as well as subfield subgroups of type $E_8(q_0)$ and the exotic local subgroups appearing in \cite[Table 1]{CLSS}. And for $p > 5$, the collection $\mathcal{C}$ also contains the Borovik subgroup $H$ with ${\rm soc}(H) = {\rm Alt}_5 \times {\rm Alt}_6$. In particular, all the subgroups in $\mathcal{C}$ have been determined up to conjugacy.

At the time of writing, it remains an open problem to determine the subgroups in $\mathcal{S}$, even up to isomorphism. However, by combining earlier results of several authors \cite{Craven0,Craven_alt,LS99,LS98,Litt} with Craven's ongoing work \cite{CravenE8}, we know that if $S$ denotes the socle of a subgroup in $\mathcal{S}$, then one of the following holds (here ${\rm Lie}(p)$ denotes the set of finite simple groups of Lie type in characteristic $p$):

\vspace{1mm}

\begin{itemize}\addtolength{\itemsep}{0.2\baselineskip}
\item $S \in {\rm Lie}(p)$ and either $S = {\rm L}_2(q_0)$ with $7 \leqs q_0 \leqs (2,q-1)\cdot 1312$, or 
\[
S \in \{ {\rm L}_3^{\e}(3), \, {\rm L}_3^{\e}(4), \, {\rm U}_3(8), \, {\rm U}_4(2), \, {}^2B_2(8) \}.
\]

\item $S \not\in {\rm Lie}(p)$ and either 

\vspace{1mm}

\begin{itemize}\addtolength{\itemsep}{0.2\baselineskip}
\item $S = {\rm Alt}_6$ $(p \ne 5)$, ${\rm Alt}_7$ $(p \ne 2)$, ${\rm M}_{11}$ $(p =3,11)$, ${\rm J}_3$ $(p=2)$ or ${\rm Th}$ $(p=3)$;

\item $S = {\rm L}_2(q')$ with $q' \in \{7,8,11,13,16,17,19,25,29,31,32,41,49,61\}$; or

\item $S = {\rm L}_3(3)$, ${\rm L}_3(5)$, ${\rm L}_4(5)$, ${\rm U}_3(3)$, ${\rm U}_4(2)$, ${\rm PSp}_4(5)$, ${}^2B_2(8)$, ${}^2B_2(32)$, ${}^3D_4(2)$ or ${}^2F_4(2)'$.
\end{itemize}
\end{itemize}

\vspace{1mm}

It is worth noting that the collection $\mathcal{S}$ is known to be nonempty for certain values of $q$. For example, if $G = E_8(2)$ then the main theorem of \cite{PS} shows that $\mathcal{S}$ contains unique conjugacy classes of subgroups isomorphic to ${\rm L}_3(5).2$ and ${\rm PGSp}_4(5)$.
\end{itemize}
\end{rem}

Next we look more closely at the subgroups of type (I) in the definition of the collection $\mathcal{C}$. Here $H = N_G(\bar{H}_{\s})$ for some $\s$-stable positive-dimensional closed subgroup $\bar{H}$ of $\bar{G}$ and the possibilities for $\bar{H}$ can be divided into the following cases (recall that a closed subgroup $\bar{H}$ of $\bar{G}$ has \emph{maximal rank} if $\bar{H}^{\circ}$ contains a maximal torus of $\bar{G}$):

\vspace{1mm}

\begin{itemize}\addtolength{\itemsep}{0.2\baselineskip}
\item[{\rm (a)}] $\bar{H}$ is a parabolic subgroup;
\item[{\rm (b)}] $\bar{H}^\circ$ is a maximal torus;
\item[{\rm (c)}] $\bar{H}$ is reductive of maximal rank and $\bar{H}^\circ$ is not a torus;
\item[{\rm (d)}] $\bar{H}$ is reductive, not of maximal rank.
\end{itemize}

\vspace{1mm}

In general, the structure of the finite group $\bar{H}_{\s}$ depends on the choice of  $\bar{G}$-conjugate of $\bar{H}$. Let us explain this in more general terms.

Let $X$ be a group that $\sigma$ acts on. We denote by $H^1(\sigma, X)$ the equivalence classes of $X$ under the relation $\sim$ defined by $x \sim y$ if and only if $x = \sigma(g)^{-1}yg$ for some $g \in X$. The following result is \cite[2.7]{SpringerSteinberg}.

\begin{lem}\label{l:sigmaclasses}
Let $X$ be a $\sigma$-invariant closed subgroup of $\bar{G}$, and let $\mathcal{X}$ be the set of $\sigma$-invariant $\bar{G}$-conjugates of $X$. Denote by $\mathcal{X}/\bar{G}_{\s}$ the set of $\bar{G}_{\s}$-classes in $\mathcal{X}$. Then there is a bijection 
\[
\mathcal{X}/\bar{G}_{\s} \to H^1(\sigma, N_{\bar{G}}(X)/N_{\bar{G}}(X)^\circ),
\]
which maps the class $[X^g]$ to the image of $\sigma(g)g^{-1}$ in $H^1(\sigma, N_{\bar{G}}(X)/N_{\bar{G}}(X)^\circ)$.
\end{lem}

In particular, if $N_{\bar{G}}(X)$ is connected, then Lemma \ref{l:sigmaclasses} implies that all of the $\sigma$-invariant $\bar{G}$-conjugates of $X$ are conjugate in $\bar{G}_{\s}$. In the context of maximal subgroups, one special case where this applies is when $\bar{H}$ is a parabolic subgroup, in which case $N_{\bar{G}}(\bar{H}) = \bar{H}$ is connected. 

For each maximal subgroup $H = N_G(\bar{H}_{\s})$ as in case (I), we have $\bar{H} = N_{\bar{G}}(\bar{H})$ by the maximality of $\bar{H}$. Therefore, Lemma \ref{l:sigmaclasses} implies that the $\bar{G}_{\s}$-classes of $\s$-invariant conjugates of $\bar{H}$ are in bijection with the set $H^1(\s, \bar{H}/\bar{H}^\circ)$. More precisely, a $\s$-invariant conjugate $\bar{H}^g$ with $g \in \bar{G}$ corresponds to the image of the element $w = \sigma(g)g^{-1}$ in $H^1(\s, \bar{H}/\bar{H}^\circ)$. Furthermore,
\begin{equation}\label{e:bar}
\left( \bar{H}^g \right)_{\s} = \left( \bar{H}_{w\s} \right)^g \mbox{ and } \left(\bar{G}_{\s}\right)' = \left( \left(\bar{G}_{w\s} \right)' \right)^g,
\end{equation}
where $w\s$ denotes the Steinberg endomorphism $x \mapsto \sigma(x)^w$ of $\bar{G}$. In particular, if we are interested in studying the $2$-elusivity of $T$, then it will be sufficient to consider involutions in the group $\widetilde{H_0} = N_{\bar{G}}(\bar{H}_{w\s}) \cap \left(\bar{G}_{w\s} \right)'$. We formalize this observation in the following lemma.

\begin{lem}\label{l:tildeH0}
Let $\bar{H}$ be a maximal $\sigma$-invariant positive-dimensional closed subgroup of $\bar{G}$ and let $\bar{H}^g$ be a $\s$-invariant conjugate of $\bar{H}$, corresponding to the image of $w = \s(g)g^{-1}$ in $H^1(\s, \bar{H}/\bar{H}^\circ)$. Assume $H = N_{G}(\left( \bar{H}^g \right)_{\s})$ is a maximal subgroup of $G$ and set 
\[
\widetilde{H_0} = N_{\bar{G}}(\bar{H}_{w\s}) \cap \left(\bar{G}_{w\s} \right)'.
\]
Then the following statements hold:

\vspace{1mm}

\begin{itemize}\addtolength{\itemsep}{0.2\baselineskip}
\item[{\rm (i)}]  We have $(\widetilde{H_0})^g = H_0$.
\item[{\rm (ii)}] Assume that $T \neq G_2(2)'$. Then $T$ is $2$-elusive on $\Omega = G/H$ if and only if $\widetilde{H_0}$ meets every $\s$-invariant conjugacy class of involutions in $\bar{G}$.
\end{itemize}
\end{lem}

\begin{proof}
The first claim follows from equation \eqref{e:bar} above. And in view of (i), the second claim is an immediate consequence of Lemma \ref{l:Gsigmaclasses}.
\end{proof}

\begin{rem}\label{r:e7_inv}
Suppose $\bar{G}$ is of exceptional type and $\bar{G}_{\s} \ne G_2(2), {}^2G_2(3), {}^2F_4(2)$. If $\bar{G} \ne E_6,E_7$, then $\bar{G}$ is simply connected and
\[
\left(\bar{G}_{w\s} \right)' = \bar{G}_{w\s} \mbox{ and } \widetilde{H_0} = N_{\bar{G}_{w\s}}(\bar{H}_{w\s}).
\]
In the proof of Theorem \ref{t:main}, no issues arise when $\bar{G} = E_6$ since the index of $(\bar{G}_{w\s})'$ in $\bar{G}_{w\s}$ is $(3,q \pm 1)$ and thus every involution in $\bar{G}_{w\s}$ is contained in $(\bar{G}_{w\s})'$. However, if $\bar{G} = E_7$ then we have $[\bar{G}_{w\s} : (\bar{G}_{w\s})'] = (2,q-1)$ and so there is a difference between the classes of involutions in $\bar{G}_{w\s}$ and $(\bar{G}_{w\s})'$ when $q$ is odd, as highlighted in Remark \ref{r:TvsInnT}. But here we can verify that a given element $x \in N_{\bar{G}_{w\s}}(\bar{H}_{w\s})$ is contained in $\left(\bar{G}_{w\s} \right)'$ by checking that it lifts to an element $x' \in \left(\bar{G}_{{\rm sc}}\right)_{w'\s}$, where $w'$ is a lift of $w$ to $\bar{G}_{{\rm sc}}$. This is simply a consequence of the fact that $$\left(\bar{G}_{{\rm sc}}\right)_{w'\s} / Z(\bar{G}_{{\rm sc}})_{w'\s} \cong \left(\bar{G}_{w\s} \right)',$$ where the isomorphism is induced by an isogeny $\bar{G}_{{\rm sc}} \rightarrow \bar{G}$. In proofs, we will often verify $x \in \left(\bar{G}_{w\s} \right)'$ by identifying a specific lift $x' \in \left(\bar{G}_{{\rm sc}}\right)_{w'\s}$ that centralizes both $w'$ and $\s$. These calculations can be done with {\sc Magma}, and appear for example in the proof of Lemma \ref{l:e7_mr}. 
\end{rem}

\begin{ex}\label{ex:e6t1}
To illustrate such a calculation with an example, consider $\bar{G} = E_7$ with $q$ odd, and let $\bar{H} = E_6T_1.2$. (This is discussed in Case 1.6 of the proof of Lemma \ref{l:e7_mr}, which also contains more details.) Here we can take $\bar{H}$ to be the normalizer of a Levi factor of type $E_6$: $$\bar{H}^\circ = \langle \bar{T}, U_{\alpha} : \alpha \in \Phi' \rangle,$$ where $\bar{T}$ is a maximal torus of $\bar{G}$ and $\Phi'$ is the root subsystem of $\Phi$ with base $\{\alpha_1, \ldots, \alpha_6\}$. Then $\bar{H} = \bar{H}^\circ{:}\langle w \rangle$, where $$w = w_{\alpha_1} w_{\alpha_2} w_{\alpha_5} w_{\alpha_7} w_{\alpha_{37}} w_{\alpha_{55}} w_{\alpha_{61}}$$ corresponds to the longest element of the Weyl group of $E_7$. Here $\alpha_i$ denotes the $i$-th root in the ordering of $\Phi$ used in {\sc Magma} (see Remark \ref{r:ordering}).

The following {\sc Magma} code verifies that in the simply connected cover $\bar{G}_{\rm sc}$, a lift $w'$ of $w$ centralizes $w_{\alpha}'$ and $h_{\alpha}'(-1)$ in $\bar{G}_{\rm sc}$ for all $\alpha \in \Phi$. (The fact that $w'$ centralizes $h_{\alpha}(-1)$ is also clear from the fact that $w'$ acts as $\alpha \mapsto -\alpha$ on $\Phi$.)

\vs

{\small
\begin{verbatim}
G := GroupOfLieType("E7", RationalField() : Isogeny := "SC");
w := elt<G|1>*elt<G|2>*elt<G|5>*elt<G|7>*elt<G|37>*elt<G|55>*elt<G|61>;
{w*elt<G|i> eq elt<G|i>*w : i in [1..63]}; // output: {true}
{w*TorusTerm(G,i,-1) eq TorusTerm(G,i,-1)*w : i in [1..63]}; // output: {true}
\end{verbatim}
}

\vs

Moreover, as $w_{\alpha}'$ and $h_{\alpha}'(-1)$ are clearly fixed in $\bar{G}_{\rm sc}$ by $\s$ for all $\alpha \in \Phi$, we can conclude that $$w_{\alpha}, h_{\alpha}(-1) \in (\bar{G}_{w\s})'$$ for all $\alpha \in \Phi$. Define $t = h_{\alpha_2}(-1)$ and \[
h = w_{\alpha_2} w_{\alpha_{28}} w_{\alpha_{38}} w_{\alpha_{46}} \in \bar{H}^{\circ},
\] 
so $t,h \in \bar{H} \cap (\bar{G}_{w\s})'$ and $h$ corresponds to the longest element of the Weyl group of $E_6$. With the following {\sc Magma} code, we verify that $t$, $w$, and $thw$ are involutions in $\bar{G}$. Moreover, we calculate that in the adjoint representation, these elements have respective fixed point spaces of dimensions $69$, $79$ and $63$:

\vs

{\small
\begin{verbatim}
G := GroupOfLieType("E7", RationalField() : Isogeny := "SC");
w := elt<G|1>*elt<G|2>*elt<G|5>*elt<G|7>*elt<G|37>*elt<G|55>*elt<G|61>;
h := elt<G|2>*elt<G|28>*elt<G|38>*elt<G|46>;
t := TorusTerm(G,2,-1);
r := AdjointRepresentation(G);

A1 := Matrix(r(t));
A2 := Matrix(r(t*h*w));
A3 := Matrix(r(w));

[Order(i) eq 2 : i in [A1,A2,A3]]; // output: [true,true,true]
[Dimension(Kernel(i-1)) : i in [A1,A2,A3]]; // output: [ 69, 79, 63 ]
\end{verbatim}
}

\vs

We conclude then from Table \ref{table:oddpinvolutions} that $t$, $thw$, $w$ are involutions of type $A_1D_6$, $E_6T_1$, $A_7$, respectively. Consequently $\bar{H} \cap (\bar{G}_{w\s})'$ meets every $\bar{G}$-class of involutions.
\end{ex}

Typically, the set $H^1(\sigma, \bar{H}/\bar{H}^\circ)$ is very small (and often trivial) unless $\bar{H}^\circ$ is a maximal torus. So next we will look more closely at the normalizers of maximal tori in $\bar{G}_{\s}$, referring the reader to \cite[3.3]{Carter} for more details. 

First recall that a \emph{maximal torus} of $\bar{G}_{\s}$ is a subgroup of the form $\bar{T}_{\s}$, where $\bar{T}$ is a $\s$-invariant maximal torus of $\bar{G}$. In order to describe the $\bar{G}_{\s}$-classes of $\s$-invariant maximal tori, let us first observe that $\sigma$ acts on the Weyl group $W = N_{\bar{G}}(\bar{T}) / \bar{T}$, since $\bar{T}$ is $\sigma$-invariant, and we refer to $H^1(\s, W)$ as the set of $\sigma$-conjugacy classes of $W$. Note that if $\s = \s_q$ is a Frobenius endomorphism, then $\sigma(w_{\alpha}) = w_{\alpha}$ for all $\alpha \in \Phi$. So in the untwisted case, $H^1(\s, W)$ is just the set of conjugacy classes of $W$. In the general setting, we have the following result, which is a special case of Lemma \ref{l:sigmaclasses} (see Propositions 3.3.1, 3.3.2 and 3.3.3 in \cite{Carter}). Here $\pi: N_{\bar{G}}(\bar{T}) \rightarrow W$ is the quotient map.

\begin{lem}\label{l:torus1}
Let $g \in \bar{G}$. Then the map $\bar{T}^g \mapsto \pi(\s(g)g^{-1})$ defines a bijection from the set of $\bar{G}_{\s}$-classes of $\sigma$-invariant maximal tori in $\bar{G}$ to the set $H^1(\s,W)$.
\end{lem}

Let $w \in W$ and write $w = \pi(n)$, where $n = \s(g)g^{-1} \in N_{\bar{G}}(\bar{T})$ for some $g \in \bar{G}$. Then under the bijection of Lemma \ref{l:torus1}, the $\s$-class of $w$ corresponds to the $\bar{G}_{\s}$-class of $\bar{T}_w := \bar{T}^g$. And as before (see \eqref{e:bar}), we have 
\[
\left( \bar{T}_w \right)_{\s} = \left( \bar{T}_{n\s} \right)^g \mbox{ and } N_{\bar{G}_{\s}}(\bar{T}_w) = \left( N_{\bar{G}}(\bar{T})_{n\s} \right)^g,
\]
where $n\s$ denotes the Steinberg endomorphism $x \mapsto \sigma(x)^n$ of $\bar{G}$. In particular, for computations involving $N_{\bar{G}_{\s}}(\bar{T}_w)$, it will often be more convenient to work with the $\bar{G}$-conjugate $N_{\bar{G}}(\bar{T})_{n\s}$ instead.

In order to describe the structure of the normalizer of $\bar{T}_w$ in $\bar{G}_{\s}$, we define the $\s$-centralizer of $w \in W$ by 
\[
C_{W,\s}(w) = \{x \in W \,:\, \sigma(x)^{-1} w x = w\}.
\]
Then the following result is \cite[Proposition 3.3.6]{Carter}.

\begin{lem}\label{l:torus2}
We have $N_{\bar{G}_{\s}}(\bar{T}_w) / \left( \bar{T}_w \right)_{\s} \cong C_{W,\s}(w)$ for all $w \in W$.
\end{lem}

In particular, as a consequence of Lemma \ref{l:torus2}, we deduce that 
\[
N_{\bar{G}}(\bar{T})_{n\s} / \bar{T}_{n\s} \cong C_{W,\s}(w),
\]
with an isomorphism induced by the quotient map $\pi: N_{\bar{G}}(\bar{T}) \rightarrow W$.

The following result, which is due to Tits \cite{Tits}, will also be useful.

\begin{thm}\label{t:titsweyl}
Let $\bar{T}$ be a maximal torus and set $W_0 = \la w_{\a} \,:\, \a \in \Phi\ra$. Then $N_{\bar{G}}(\bar{T}) = \bar{T} W_0$. Furthermore, if $p = 2$, then $W \cong W_0$ and $N_{\bar{G}}(\bar{T}) = \bar{T}{:}W_0$ is a split extension.
\end{thm}

Finally, we conclude this section by presenting Theorem \ref{t:odd} below, which describes all the core-free maximal subgroups of odd index in an almost simple exceptional group of Lie type. This is a special case of a more general result of Liebeck and Saxl \cite{LS} on primitive permutation groups of odd degree.
 
In the second column of Table \ref{tab:odd}, we refer to the \emph{type of $H$}, which gives an approximate description of the structure of $H$ (working with the Lie notation for classical groups). In each case, the precise structure is readily available in the literature. For example, we refer the reader to \cite[Tables  5.1, 5.2]{LSS} in the cases where $H$ is a non-parabolic maximal rank subgroup. Note that for $T=E_6(q)$ we write $P_1$ and $P_6$ for representatives of the two $T$-classes of maximal parabolic subgroups with Levi factors of type $D_5(q)$. Also observe that in the final column of Table \ref{tab:odd} we write ``graphs'' to indicate that the given subgroup is maximal only if $G$ contains a graph (or graph-field) automorphism of $T$, and similarly if we write ``no graphs''.

\begin{thm}\label{t:odd}
Let $G \leqs {\rm Sym}(\O)$ be a finite almost simple primitive permutation group with point stabilizer $H$ and socle $T$, which is a simple exceptional group of Lie type over $\mathbb{F}_q$. Then $|\O|$ is odd if and only if one of the following holds:

\vspace{1mm}

\begin{itemize}\addtolength{\itemsep}{0.2\baselineskip}
\item[{\rm (i)}] $q$ is even and $H$ is a parabolic subgroup of $G$.
\item[{\rm (ii)}] $q$ is odd and $H$ is a subfield subgroup over $\mathbb{F}_{q_0}$, where $q=q_0^k$ and $k$ is an odd prime.
\item[{\rm (iii)}] $q$ is odd and $(G,H)$ is one of the cases recorded in Table \ref{tab:odd}.
\end{itemize}
\end{thm}

{\scriptsize
\begin{table}
\[
\begin{array}{lll} \hline
T & \mbox{Type of $H$}  & \mbox{Conditions} \\ \hline
E_8(q) & A_1(q)^8, \, D_8(q), \, D_4(q)^2 & \\
& (q-\e)^8 & q \equiv \e \imod{4} \\
E_7(q) & A_1(q)^7,\, A_1(q)D_6(q), \, A_1(q)^3D_4(q) & \\
& (q-\e)^7 & q \equiv \e \imod{4} \\
E_6(q) & P_1,\, P_6 & \mbox{no graphs} \\
& D_5(q) \times (q-1) & \mbox{graphs} \\
& D_4(q) \times (q-1)^2 & \\
& (q-1)^6 & q \equiv 1 \imod{4} \\ 
{}^2E_6(q) & D_5^{-}(q) \times (q+1),\, D_4(q) \times (q+1)^2 & \\
& (q+1)^6 & q \equiv 3 \imod{4} \\ 
F_4(q) & B_4(q),\, D_4(q) & \\
{}^3D_4(q) & G_2(q),\, A_1(q)A_1(q^3) & \\
& A_2^{\e}(q) \times (q^2 + \e q +1) & q \equiv \e \imod{4} \\
G_2(q) & A_1(q)^2 &  \\
& A_2^{\e}(q) & q \equiv \e \imod{4} \\
& (q-\e)^2 & \mbox{$p=3$, $q \geqs 9$, $q \equiv \e \imod{4}$, graphs} \\ 
& G_2(2),\, 2^3.{\rm L}_3(2) & q=p \equiv \pm 3 \imod{8} \\
{}^2G_2(q)' & A_1(q) & q \geqs 27 \\ 
& 2^3{:}7 & q = 3 \\ \hline
\end{array}
\]
\caption{The groups with $|\O|$ odd in Theorem \ref{t:odd}(iii)}
\label{tab:odd}
\end{table}
}

\subsection{Feasible characters}\label{ss:feasible}

As before, let $G$ be an almost simple exceptional group of Lie type with socle $T = (\bar{G}_{\s})'$ and write $\mathcal{M} = \mathcal{C} \cup \mathcal{S}$ for the set of core-free maximal subgroups of $G$. Recall that the subgroups in $\mathcal{M}$ have been determined up to conjugacy, with the exception of a finite number of  open cases involving the collection $\mathcal{S}$ when $T = E_7(q)$ or $E_8(q)$ (see Remark \ref{r:CS}). We conclude this preliminary section by explaining how the theory of feasible characters can be used to study $2$-elusivity in the presence of these undetermined cases. As before, we write $K$ for the algebraic closure of $\mathbb{F}_q$, where $q = p^f$ with $p$ a prime.

We begin with a general definition. Let $H$ be a finite group and let $V$ be a finite-dimensional $\bar{G}$-module. Following Litterick \cite[Definition 3.2]{Litt}, we say that a $KH$-module $V_0$ with Brauer character $\chi$ is a \emph{feasible decomposition} of $H$ on $V$ if $\dim V_0 = \dim V$ and each character value $\chi(x)$ for a $p'$-element $x \in H$ corresponds to the trace on $V$ of some semisimple element in $\bar{G}$. Moreover, this correspondence should be compatible with the associated power maps. In this situation, $\chi$ is called a \emph{feasible character}. Clearly, if $H$ is a finite subgroup of $\bar{G}$, then the restriction $V \downarrow H$ is a feasible decomposition of $H$ on $V$. However, it is not true that every feasible decomposition corresponds to a finite subgroup of $\bar{G}$.

In \cite{Litt}, Litterick considers the case where $H$ is a finite simple group, which is not a group of Lie type in characteristic $p$, and he develops a computational approach for finding the list of feasible decompositions of $H$ with respect to the minimal module $V_{\operatorname{min}}$ and the adjoint module $\mathcal{L}(\bar{G})$. This describes the possible embeddings of such a finite simple group $H$ into $\bar{G}$, which in turn severely limits the possibilities for the socle of a maximal subgroup of $G$ in class $\mathcal{S}$.

In this paper, we will use Litterick's {\sc Magma} code in \cite{LittGithub} to find feasible characters in certain cases where $H$ is almost simple, but not necessarily simple. Referring to the proof of Theorem \ref{t:main}, this will help us to overcome some of the difficulties that arise when $T = E_7(q)$ or $E_8(q)$, and the point stabilizer $H$ is an almost simple maximal subgroup in the class $\mathcal{S}$, where a complete list of such subgroups is not currently known.

In order to develop these ideas further, we need to introduce some additional notation and terminology. 

Let $H$ be a finite subgroup of $\bar{G}$. Following Craven \cite[Definition 3.2]{CravenPSL2}, we say that $H$ is \emph{strongly imprimitive} if $H < L < \bar{G}$ for some subgroup $L$ of $\bar{G}$ such that 

\vspace{1mm}

\begin{itemize}\addtolength{\itemsep}{0.2\baselineskip}
\item[{\rm (a)}] $L$ is a positive-dimensional maximal closed subgroup of $\bar{G}$; and
\item[{\rm (b)}] $L$ is both $\sigma$-stable and $N_{{\rm Aut}^+(\bar{G})}(H)$-stable, where ${\rm Aut}^{+}(\bar{G})$ is the group generated by the inner, graph and $p$-power field automorphisms of $\bar{G}$ (note that the latter are automorphisms of $\bar{G}$ as an abstract group, but not as an algebraic group).
\end{itemize}

\vspace{1mm}

Now each automorphism of $\bar{G}_{\s}$ extends to an element of ${\rm Aut}^{+}(\bar{G})$. Therefore, if $H < \bar{G}_{\s}$ is strongly imprimitive with $L$ as above, then $N_{\Aut(\bar{G}_{\s})}(H)$ is contained in $N_{\Aut(\bar{G}_{\s})}(L_{\s})$ (see \cite[p.12]{CravenPSL2}). As a consequence, we obtain the following result.

\begin{lem}\label{l:stronglyimprimitive}
Let $H$ be a subgroup of $G$ such that $H \cap T$ is a strongly imprimitive subgroup of $\bar{G}_{\s}$. Then $H$ is not contained in the collection $\mathcal{S}$.
\end{lem}

The next lemma, which follows from \cite[Proposition 4.5]{CravenPSL2}, is our main tool for showing that a given subgroup of $G$ is strongly imprimitive. Note that if $\bar{G} = E_7$ or $E_8$, then the adjoint module $\mathcal{L}(\bar{G})$ is irreducible unless $(\bar{G},p) = (E_7,2)$.

\begin{lem}\label{l:liefixed}
Let $H$ be a subgroup of $G$ and assume that the $\bar{G}$-module $V = \mathcal{L}(\bar{G})$ is irreducible. If $H$ fixes a nonzero element of $V$, then $H$ is strongly imprimitive. 
\end{lem}

In order to effectively apply Lemma \ref{l:liefixed}, we need a condition which forces a $KH$-module to have a nonzero fixed point. The following cohomological condition was introduced by Litterick in \cite[Chapter 6, p.70]{Litt}, where it is referred to as property (\textbf{P}).

\begin{defn}\label{d:P}
Let $H$ be a finite group and let $V_0$ be a finite-dimensional $KH$-module with composition factors $W_1, \ldots, W_t$. Let $m \geqs 0$ be the number of trivial composition factors and let $W_i^*$ be the dual of $W_i$. Then we say that $V_0$ (or the Brauer character of $V_0$) has property (\textbf{P}) if and only if

\vspace{1mm}

\begin{itemize}\addtolength{\itemsep}{0.2\baselineskip}
\item[{\rm (a)}] $\sum_{i = 1}^t \dim H^1(H,W_i) \geqs m$; and
\item[{\rm (b)}] If $\sum_{i = 1}^t \dim H^1(H,W_i) = m$, then for some $i \in \{1, \ldots, t\}$ we have $H^1(H,W_i) = 0$ and $H^1(H,W_i^*) \ne 0$, or $H^1(H,W_i) \ne 0$ and $H^1(H,W_i^*) = 0$.
\end{itemize}
\end{defn}

Then the key result here is the following.

\begin{prop}\label{p:condP}
Let $H$ be a subgroup of $G$ and assume that $V = \mathcal{L}(\bar{G})$ is an irreducible $\bar{G}$-module. If the restriction of $V$ to $H \cap T$ does not have property {\rm (\textbf{P})}, then $H$ is not contained in the collection $\mathcal{S}$.
\end{prop}

\begin{proof}
By \cite[Proposition 3.6]{Litt}, if (\textbf{P}) does not hold, then $H \cap T$ has a $1$-dimensional trivial submodule on $V$. The result now follows from Lemmas \ref{l:stronglyimprimitive} and \ref{l:liefixed}.
\end{proof}

\begin{rem}\label{r:cohomology}
Let $H$ be a subgroup of $G$ and let $V_0$ be a feasible decomposition of $H$ on  the adjoint module $V = \mathcal{L}(\bar{G})$. Then we can use {\sc Magma} to check property (\textbf{P}) in Definition \ref{d:P}. To do this, first observe that  each composition factor $W_i$ is defined over a finite field, so we have $W_i \cong K \otimes_{\mathbb{F}_{q_i}} W_i'$ for some absolutely irreducible $\mathbb{F}_{q_i}H$-module $W_i'$, where $q_i$ is some power of $p$. Then 
\[
H^1(H,W_i) \cong K \otimes_{\mathbb{F}_{q_i}} H^1(H, W_i')
\]
and thus $\dim_K H^1(H,W_i) = \dim_{\mathbb{F}_{q_i}} H^1(H,W_i')$. Cohomology groups over finite fields can be computed with {\sc Magma} and this allows us to compute $\dim_K H^1(H,W_i)$ for all $i$.
\end{rem}

\begin{ex} 
For instance, the following {\sc Magma} code first constructs all seven absolutely irreducible modules $W$ of $H = {\rm PGL}_2(7)$ in characteristic $p = 3$. We then calculate that each cohomology group $H^1(H,W)$ is trivial, with a single exception, which is $1$-dimensional.

\vs

{\small
\begin{verbatim}
H := PGL(2,7);
M := AbsolutelyIrreducibleModules(H, GF(3));
[CohomologicalDimension(x,1) : x in M]; // output: [ 0, 0, 0, 0, 0, 1, 0 ]
\end{verbatim}
}
\end{ex}

We are now in a position to summarize our approach for studying the $2$-elusive problem when $T = E_7(q)$ or $E_8(q)$ and the point stabilizer $H$ is contained in the collection $\mathcal{S}$. Set $V = \mathcal{L}(\bar{G})$ and assume that $(\bar{G},p) \ne (E_7,2)$, so $V$ is an irreducible $\bar{G}$-module. Suppose $H$ is an almost simple subgroup in $\mathcal{S}$ with $H_0 = H \cap T$. We can   proceed as follows:

\vspace{1mm}

\begin{itemize}\addtolength{\itemsep}{0.2\baselineskip}
\item[{\rm (a)}] First we use Litterick's {\sc Magma} code \cite{LittGithub} to determine all the feasible decompositions $V_0$ for the action of $H_0$ on $V$. 

\item[{\rm (b)}] For each feasible decomposition in (a), we proceed as in Remark \ref{r:cohomology} to determine if $V_0$ has property (\textbf{P}).

\item[{\rm (c)}] Given a feasible decomposition $V_0$ with property (\textbf{P}), we then study the composition factors of $V_0$ in order to  compute $\dim C_{V_0}(x)$ for each involution $x \in H_0$. See Remark \ref{r:cvx} below for more details.

\item[{\rm (d)}] By examining Tables \ref{table:evenpinvolutions} and  \ref{table:oddpinvolutions}, we see that $\dim C_V(x)$ uniquely determines the $T$-class of each involution $x \in H_0$. So from (c), we can determine if there are any 
feasible characters with property (\textbf{P}) such that $H_0$ meets every $T$-class of involutions.

\item[{\rm (e)}] If no such feasible characters are identified in (d), we can conclude that there is no action of $G$ with point stabilizer $H \in \mathcal{S}$ such that $H \cap T = H_0$ and $T$ is $2$-elusive.
\end{itemize}

\begin{rem}\label{r:cvx}
In step (c) above, we need to compute $\dim C_{V_0}(x)$ for each involution $x \in H_0$. If $p \ne 2$, then $x$ is semisimple and we can read off $\dim C_{V_0}(x)$ from the Brauer character $\chi$ of $V_0$. Indeed, we have 
\[
\dim C_{V_0}(x) = \frac{1}{2}(\dim V_0 + \chi(x))
\]
for every involution $x \in H_0$. 

For $p=2$, the calculation is more difficult and our approach will depend on the situation. For example, if $V_0$ is irreducible as a $KH_0$-module (and thus absolutely irreducible since $K$ is algebraically closed), then we can usually construct $V_0$ in {\sc Magma} via the function  \path{AbsolutelyIrreducibleModules}, which constructs all absolutely irreducible modules in a given characteristic. In some cases, we find that $|H_0|$ is large and constructing all absolutely irreducible modules is not feasible, but in our situation it is important to note that $\dim V_0 \leqs 248$. Indeed, this allows us to appeal to the work of Hiss and Malle \cite{HissMalle,HissMalleCorr}, where all the absolutely irreducible modules of dimension at most $250$ of all quasisimple finite groups are determined. In particular, the possible composition factors of $V_0 \downarrow \operatorname{soc}(H_0)$ are known and can be constructed. Then by Frobenius-Nakayama reciprocity, the $KH_0$-module $V_0$ can in turn be constructed as a simple quotient of an induced module $\operatorname{Ind}_{\operatorname{soc}(H_0)}^{H_0}(W)$, see Example \ref{e:psl4-5}.

We note that in many cases, explicit generators for certain quasisimple groups with respect to certain low-dimensional absolutely irreducible modules are available in the Web Atlas \cite{WebAt}. In other cases, we can often obtain the relevant module as a composition factor of a suitable permutation module corresponding to the action of $H_0$ on the cosets of a maximal subgroup, working with the {\sc Magma} function \texttt{PermutationModule} to construct the relevant permutation module. We illustrate this approach in Example \ref{e:psl4-5}.
\end{rem}

\begin{ex}\label{ex:e8cfc}
To illustrate the approach outlined above, let us assume $T = E_8(q)$ and $(H_0,p) = (\PGL_2(7), 3)$, $(\PGL_2(11),5)$ or $(\PGL_2(13),7)$. We present the following information in Tables \ref{table:PGL2_7_char3}, \ref{table:PGL2_11_char5} and \ref{table:PGL2_13_char7}, which has been computed using {\sc Magma}:

\vspace{1mm}

\begin{itemize}\addtolength{\itemsep}{0.2\baselineskip}
\item[(a)] The absolutely irreducible $KH_0$-modules $W$ are listed according to their dimension and they are indexed alphabetically. For example, if $(H_0, p) = (\PGL_2(7), 3)$ then $H_0$ has exactly three $6$-dimensional absolutely irreducible modules, denoted by $6_a$, $6_b$ and $6_c$. In the first row of each table, we record $\dim H^1(H_0,W)$.
		
\item[(b)] In each case, $H_0$ has exactly two classes of involutions, with  representatives denoted $t_1$ and $t_2$, where $C_{H_0}(t_1) = {\rm D}_{2(p-1)}$ and $C_{H_0}(t_2) = {\rm D}_{2(p+1)}$. In the second and third row of each table, we present the character values $\chi(t_i)$ for each irreducible $KH_0$-module in (a).
		
\item[(c)] The remaining rows in each table give the multiplicities of composition factors for all the feasible characters of $H_0$ on $V = \mathcal{L}(\bar{G})$ with  property (\textbf{P}). In general, there can be many feasible characters that do not have property (\textbf{P}). For example, if $(H_0,p) = (\PGL_2(7),3)$ then we calculate that there are $49$ feasible characters of $H_0$ on $V$, but only $14$ of them have property (\textbf{P}) and they are denoted \texttt{1}-\texttt{14} in Table \ref{table:PGL2_7_char3}.

\item[(d)] In the final column of each table, we use the symbol $(\star)$ to denote the feasible characters with property (\textbf{P}) and the additional condition that   $H_0$ intersects both classes of involutions in $T$. So for example, if $p=5$ and there exists a subgroup $H \in \mathcal{S}$ with $H_0 = \PGL_2(11)$, then $T$ is $2$-elusive if and only if $V \downarrow H_0$ has composition factors $1_a$, $10_a$, $10_b$, $10_c$, $11_a$ and $11_b$, with respective multiplicities $12$, $2$, $2$, $2$, $14$ and $2$. But it remains an open problem to determine whether or not such a maximal subgroup actually exists.
\end{itemize}
\end{ex}

{\scriptsize
\begin{table}
\[
\begin{array}{r|*{8}{c}}
&  1_a &  1_b &  6_a & 6_b & 6_c & 7_a & 7_b & \\ \hline
\dim H^1(H_0,W)  &  0 &  0 &  0 &  0 &  0 &  1 &  0 & \\
\chi(t_1)        &  1 &  2 &  0 &  0 &  0 &  1 &  2 & \\
\chi(t_2)        &  1 &  1 &  1 &  2 &  2 &  2 &  2 & \\  \hline                  
\texttt{1}                 &  5 &  9 &  5 &  3 &  3 & 10 & 14 &  \\
\texttt{2}                &  6 &  8 &  5 &  3 &  3 &  9 & 15 &  \\
\texttt{3}                &  6 & 10 &  7 &  3 &  3 &  9 & 13 &  \\
\texttt{4}                 &  7 &  9 &  7 &  1 &  5 &  8 & 14 &  \\
\texttt{5}                 &  7 &  9 &  7 &  5 &  1 &  8 & 14 &  \\
\texttt{6}                 &  7 & 10 &  0 &  7 &  7 &  8 & 13 & (\star)  \\
\texttt{7}                 &  7 & 12 &  2 &  7 &  7 &  8 & 11 & (\star)  \\
\texttt{8}                 & 13 &  1 &  5 &  3 &  3 & 18 &  6 & (\star)  \\
\texttt{9}                 & 14 &  0 &  5 &  3 &  3 & 17 &  7 & (\star)  \\
\texttt{10}                & 14 &  2 &  7 &  3 &  3 & 17 &  5 & (\star)  \\
\texttt{11}                & 15 &  1 &  7 &  1 &  5 & 16 &  6 & (\star)  \\
\texttt{12}                & 15 &  1 &  7 &  5 &  1 & 16 &  6 & (\star)  \\
\texttt{13}                & 15 &  2 &  0 &  7 &  7 & 16 &  5 &  \\
\texttt{14}                & 15 &  4 &  2 &  7 &  7 & 16 &  3 &  \\  \hline
\end{array}
\]
\caption{$H_0 = \PGL_2(7)$, $p = 3$}
\label{table:PGL2_7_char3}
\end{table}
}

{\scriptsize
\begin{table}
\[
\begin{array}{r|*{10}{c}}
& 1_a & 1_b & 10_a & 10_b & 10_c & 10_d & 10_e & 11_a & 11_b & \\ \hline
\dim H^1(H_0,W)  &  0 &  0 &  0 &  0 &  0 &  0 &  0 &  1 &  0 & \\
\chi(t_1) &  1 &  4 &  0 &  0 &  0 &  0 &  0 &  1 &  4 & \\
\chi(t_2) &  1 &  1 &  3 &  3 &  2 &  2 &  2 &  4 &  4 & \\ \hline
\texttt{1}                &  4 &  8 &  2 &  2 &  2 &  0 &  0 &  6 & 10 &  \\
\texttt{2}                 &  5 &  8 &  1 &  4 &  0 &  1 &  1 &  5 & 10 &  \\
\texttt{3}                 &  5 &  8 &  3 &  2 &  0 &  1 &  1 &  5 & 10 &  \\
\texttt{4}                 & 12 &  0 &  2 &  2 &  2 &  0 &  0 & 14 &  2 & (\star) \\ \hline
\end{array}
\]
\caption{$H_0 = \PGL_2(11)$, $p = 5$}
\label{table:PGL2_11_char5}
\end{table}
}

{\scriptsize
\begin{table}
\[
\begin{array}{r|*{10}{c}}
& 1_a & 1_b & 12_a & 12_b & 14_a & 14_b & 14_c & 14_d & 14_e & \\ \hline
\dim H^1(H_0,W)  &  0 &  0 &  1 &  0 &  0 &  0 &  0 &  0 &  0 & \\
\chi(t_1) &  1 &  1 &  0 &  0 &  2 &  2 &  5 &  5 &  5 & \\
\chi(t_2) &  1 &  6 &  5 &  2 &  0 &  0 &  0 &  0 &  0 & \\ \hline
\texttt{1}               &  1 &  1 &  7 &  3 &  1 &  1 &  1 &  3 &  3 &  \\
\texttt{2}               &  2 &  2 &  4 &  0 &  2 &  2 &  2 &  4 &  4 &  \\
\texttt{3}               &  3 &  1 &  8 &  3 &  0 &  1 &  3 &  2 &  2 &  \\
\texttt{4}               &  3 &  3 &  8 &  4 &  0 &  0 &  5 &  1 &  1 &  \\
\texttt{5}               &  4 &  2 &  5 &  0 &  1 &  2 &  4 &  3 &  3 &  \\
\texttt{6}               &  4 &  2 &  5 &  0 &  3 &  0 &  4 &  3 &  3 &  \\
\texttt{7}               &  4 &  4 &  5 &  1 &  1 &  1 &  6 &  2 &  2 &  \\
\texttt{8}               &  4 &  4 &  5 &  1 &  4 &  6 &  0 &  1 &  1 & (\star) \\ \hline
\end{array}
\]
\caption{$H_0 = \PGL_2(13)$, $p = 7$}
\label{table:PGL2_13_char7}
\end{table}
}

\begin{ex}\label{e:psl4-5}
Here we present an example where generators for the relevant representation are not available in the Web Atlas \cite{WebAt}. Suppose that $\bar{G} = E_8$, and denote the adjoint module by $V = \mathcal{L}(\bar{G})$. We consider the case where $p = 2$ and $H_0$ has socle $S = {\rm L}_4(5)$, noting that the existence of an embedding ${\rm L}_4(5) < E_8(4)$ is proved in \cite[Section 5]{CLSS}. Then \cite[6.329]{Litt} implies that $V \downarrow S$ is irreducible. It follows from the discussion in \cite{HissMalleCorr} that in characteristic $p = 2$, there are two absolutely irreducible modules of dimension $248$ for $S$, and both of these are defined over $\mathbb{F}_2$. 

Now up to conjugacy, there is a unique maximal subgroup $J < S$ of index $806$. This maximal subgroup arises from a parabolic subgroup of $\operatorname{SL}_4(5)$ with structure $5^4{:}\operatorname{SL}_2(5)^2{:}4$ (this is the stabilizer of a $2$-dimensional subspace of the natural module). From the natural coset action of $S$ on $S/J$, we can construct the permutation module $\mathbb{F}_2[S/J]$, and it turns out that the two irreducible $\mathbb{F}_2[S]$-modules of dimensional $248$ arise as composition factors of $\mathbb{F}_2[S/J]$. The following {\sc Magma} code constructs these two composition factors.

\vs

{\small
\begin{verbatim}

S := PSL(4,5);
L := MaximalSubgroups(S);
L := [x : x in L | (#S div x`order) eq 806]; 

// unique maximal subgroup of index 806, up to conjugacy
J := L[1]`subgroup; 

// permutation module corresponding to the coset action on S/J
V := PermutationModule(S,J,GF(2)); 

C := CompositionFactors(V);

// contains two non-isomorphic irreducible modules of dimension 248.
C := [x : x in C | Dimension(x) eq 248]; 
M1 := C[1];
M2 := C[2];
IsIsomorphic(M1,M2); // output: false
\end{verbatim}
}

\vs

Continuing this example, we also need to consider the possibility that 
$S < H_0 \leqs {\rm Aut}(S)$. In this case since $W = V \downarrow S$ is irreducible, it follows from Frobenius-Nakayama reciprocity that $V \downarrow H_0$ arises as a simple quotient of the induced module $\operatorname{Ind}_{S}^{H_0}(W)$. Consider, for example, the case where $H_0 = S.2$ is the unique index $2$ subgroup of $\operatorname{PGL}_4(5)$. Here each $248$-dimensional irreducible $\mathbb{F}_2[S]$-module $W$ extends to $H_0$ and we find that ${\rm Ind}_{S}^{H_0}(W)$ is uniserial, with two composition factors of dimension $248$. Hence, the extension of $W$ to $H_0$ is unique, up to isomorphism, and it can be constructed with the following {\sc Magma} code (continuing from the code above).

\vs

{\small
\begin{verbatim}
G0 := PGL(4,5);
H0 := sub<G0 | S, G0.1^2>; // index 2 in PGL(4,5)

I1 := Induction(M1,H0);
I2 := Induction(M2,H0);

N1 := I1/Socle(I1);
N2 := I2/Socle(I2);

[IsIrreducible(x) : x in [N1,N2]]; // output: [true,true]
[Dimension(x) : x in [N1,N2]]; // output: [248,248]
IsIsomorphic(N1,N2); // output: false
\end{verbatim}
}

\vs

\end{ex}

\section{Parabolic subgroups}\label{ss:parsub}

We are now ready to begin the proof of Theorem \ref{t:main} and we start by handling the groups where the point stabilizer $H$ is a maximal parabolic subgroup. Here, and for the remainder of the paper, we freely adopt all of the notation and terminology introduced in Section \ref{s:prel}. 

Recall that every $\sigma$-invariant parabolic subgroup of $\bar{G}$ is conjugate to a standard parabolic $P_I$, where the subset $I \subseteq \Delta = \{\a_1, \ldots, \a_{\ell}\}$ is $\s$-stable (note that if $T = (\bar{G}_{\s})'$ is untwisted, then every subset of $\Delta$ is $\s$-stable). We have a Levi decomposition $P_I = U_I{:}L_I$, with unipotent radical 
\[
U_I = \prod_{\alpha \in \Phi^+ \setminus \Z I} U_{\alpha}
\]
and Levi factor 
\[
L_I = \langle \bar{T}, U_{\alpha} \,:\, \alpha \in \Phi \cap \Z I \rangle
\]
for some maximal torus $\bar{T}$ of $\bar{G}$. Both $U_I$ and $L_I$ are $\sigma$-invariant and we have 
\[
(P_I)_{\s} = (U_I)_{\s}{:}(L_I)_{\s}.
\]
Note that $L_I$ is a reductive group with root system $\Phi \cap \Z I$. In addition,  $L_I = Z(L_I)^\circ L_I'$ and the derived subgroup $L_I'$ is semisimple (or trivial). 

Recall that we use the standard Bourbaki labelling of the simple roots $\alpha_i$, as given in \cite[11.4]{Humphreys}. For each simple root $\a_m$, we will adopt the standard notation $P_m$ for the maximal parabolic subgroup $P_{\Delta \setminus \{\a_m\}}$ of $\bar{G}$ (and similarly for the corresponding subgroup of $T = (\bar{G}_{\s})'$ if the subset $\Delta \setminus \{\a_m\}$ is $\s$-invariant). 

The following easy observation will be useful in the proof of Proposition \ref{p:parab} below.

\begin{lem}\label{l:oddeasy}
Let $V$ be an $n$-dimensional vector space over $\mathbb{F}_q$, where $n$ is odd and $q \equiv 3 \imod{4}$. Then there is no element $g \in {\rm GL}(V)$ such that $g^2 = -I_n$.
\end{lem}

\begin{proof}
If $g^2 = -I_n$ for some $g \in {\rm GL}(V)$, then by taking determinants we get $\det(g)^2 = -1$. But since $\det(g) \in \mathbb{F}_q$, this implies that $-1$ is a square in $\mathbb{F}_q$, which is incompatible with the condition $q \equiv 3 \imod{4}$.
\end{proof}

\begin{prop}\label{p:parab}
If $H$ is a parabolic subgroup, then either 

\vspace{1mm}

\begin{itemize}\addtolength{\itemsep}{0.2\baselineskip}
\item[{\rm (i)}] every involution in $T$ has fixed points; or 

\item[{\rm (ii)}] $T = E_7(q)$, $q \equiv 3 \imod{4}$ and $H_0 = P_2$, $P_5$ or $P_7$.
\end{itemize}

\vspace{1mm}

\noindent Furthermore, if (ii) holds then every involution in $H_0$ is of type $A_1D_6$.
\end{prop}

\begin{proof}
If $q$ is even, then $|\O|$ is odd (see Theorem \ref{t:odd}) and thus every involution in $T$ has fixed points. For the remainder, we may assume $q$ is odd. 

If $T = {}^2G_2(q)'$ or ${}^3D_4(q)$ then $T$ has a unique class of involutions (see Lemmas \ref{l:b2g2classes} and \ref{l:3d4qclasses}) and the result follows from Lemma \ref{l:uniqueclasselusive}. In the remaining cases, $\bar{G}$ is  of exceptional type and either $T$ is untwisted or $T = {}^2E_6(q)$. Since $H_0 = \bar{H}_{\s} \cap T$ contains $\bar{T}_{\s} \cap T$, where $\bar{T}$ is a $\s$-stable maximal torus of $\bar{G}$, by inspecting Table \ref{table:oddpinvolutions} we deduce that $H_0$ meets every $T$-class of involutions, with the possible exception of the groups with $T = E_7(q)$ (note that in the latter case, we have $[\bar{G}_{\s}:T] = 2$).

So let us assume $T = E_7(q)$ and $q$ is odd, in which case $H_0 = P_m$ for some $m \in \{1, \ldots, 7\}$ and the involution classes in $T$ are listed in Table \ref{table:oddpinvolutions}. Recall (see Remark \ref{r:e7_inv}) that here an involution $x \in \bar{G}_{\sigma}$ is contained in $T$ if and only if it lifts to a $\sigma$-invariant element in the simply connected cover $\bar{G}_{{\rm sc}}$ of $\bar{G}$. With this in mind, we first observe that $h_{\alpha_1}(-1) \in H_0$, so by inspecting Table \ref{table:oddpinvolutions} we deduce that $H_0$ contains an involution of type $A_1D_6$. 

If $m \in \{1,3,4,6\}$, then $H_0$ also contains the elements $w_{\alpha_2}$, $w_{\alpha_5}$ and $w_{\alpha_7}$, which in turn implies that $H_0$ contains representatives of the classes labelled $E_6T_1$ and $A_7$ in Table \ref{table:oddpinvolutions}. 

So we may assume $m \in \{2,5,7\}$. We first consider $q \equiv 1 \imod{4}$, in which case $\l^2 = -1$ for some $\l \in \mathbb{F}_q$. Then $\l^q = \l$, so it follows that $H_0$ contains the involutions 
\[
x = h_{\alpha_2}(\l)h_{\alpha_5}(\l)h_{\alpha_7}(\l),\;\; 
y = h_{\alpha_1}(-1)h_{\alpha_2}(\l)h_{\alpha_5}(\l)h_{\alpha_7}(\l).
\]
A {\sc Magma} computation (see Section \ref{ss:comp}) shows that $\dim C_V(x) = 79$ and $\dim C_V(y) = 63$ with respect to the adjoint module $V = \mathcal{L}(\bar{G})$. So by inspecting Table \ref{table:oddpinvolutions} we conclude that $x$ and $y$ are of type $E_6T_1$ and $A_7$, respectively, and this means that every involution in $T$ has fixed points.

Finally, let us assume $T = E_7(q)$, $q \equiv 3 \imod{4}$ and $H_0 = P_m$ with $m \in \{2,5,7\}$. Here $\bar{H} = P_I = U_I{:}L_I$, where $I = \Delta \setminus \{\a_m\}$ and $L_I'$ is semisimple of type $E_6$, $A_4A_2$, $A_6$ for $m=2,5,7$, respectively. To complete the proof of the proposition, we will show that $H_0$ does not contain involutions of type $E_6T_1$ or $A_7$. (We remark that the elements $x$ and $y$ defined in the previous paragraph are not in $H_0$, since the property $\l^q = -\l$ implies that $x,y \in \bar{G}_{\s}\setminus T$.)

Seeking a contradiction, suppose that $g \in H_0$ is an involution of type $E_6T_1$ or $A_7$. Let $\bar{G}_{{\rm sc}}$ be the simply connected cover of $\bar{G}$. Now $g$ lifts to an element $g' \in \left(\bar{G}_{{\rm sc}}\right)_{\s}$ of order $4$ such that 
\[
(g')^2 = h_{\alpha_2}'(-1) h_{\alpha_5}'(-1) h_{\alpha_7}'(-1)
\]
generates the center of $\bar{G}_{{\rm sc}}$ (see Remark \ref{r:e77}). Here $(g')^2$ acts as a scalar on every irreducible $\bar{G}_{{\rm sc}}$-module, and on the $56$-dimensional minimal module $V_{\operatorname{min}}$ it acts as $-I_{56}$. 

Let $P$ be the parabolic subgroup of $\bar{G}_{{\rm sc}}$ corresponding to $P_{I}$, with Levi factorization $P = U{:}L$. Then $g' \in P_{\s}$ and every $2$-element of $P_{\s} = U_{\s}{:}L_{\s}$ can be conjugated into $L_{\s}$, so we can assume that $g'$ is contained in the Levi factor $L_{\s}$. 

The composition factors of $V_{\operatorname{min}} \downarrow L'$ are given in \cite[Table 13.4]{Thomas} and the submodule structure for $V_{\operatorname{min}} \downarrow L$ is identical. It follows that $V_{\operatorname{min}} \downarrow L$ has an odd-dimensional composition factor (specifically, one of dimension $27,15,7$ for $L' = E_6, A_4A_2, A_6$, respectively). 

Since the highest weight of $V_{\operatorname{min}}$ is $p$-restricted, we have $V_{\operatorname{min}} = K \otimes_{\mathbb{F}_q} V_0$ for some absolutely irreducible $\mathbb{F}_q[(\bar{G}_{{\rm sc}})_{\s}]$-module $V_0$. Moreover, by inspecting \cite[Table 13.4]{Thomas}, we see that the highest weight of each composition factor of $V_{\operatorname{min}} \downarrow L'$ is also $p$-restricted, so the dimensions of the composition factors of $V_0 \downarrow L_{\s}$ and $V_{\operatorname{min}} \downarrow L$ are the same. In particular, $V_0 \downarrow L_{\s}$ has a composition factor of odd dimension. But we have already noted that $(g')^2$ acts as $-I_{56}$ on $V$, so by appealing to Lemma \ref{l:oddeasy} we reach a contradiction.
\end{proof}

\begin{rem}
An alternative approach to the proof of Proposition \ref{p:parab} is as follows. Let $H$ be a parabolic subgroup, and let $\chi = 1^T_{H_0}$ be the corresponding permutation character for $T$. For $p$ odd, we could use \cite[Corollary 3.2]{LLS2} to show that $\chi(x) > 0$ for each involution $x \in T$, except when (ii) holds in Proposition \ref{p:parab}. However, evaluating the expression for $\chi(x)$ in \cite[Corollary 3.2]{LLS2} for each involution $x$ is a non-trivial calculation, so we prefer the more direct approach we have adopted in the proof of Proposition \ref{p:parab}.
\end{rem}

\section{Subfield subgroups and twisted versions}\label{ss:subf}

In this short section, we prove Theorem \ref{t:main} in the cases where $H \in \mathcal{C}$ is of type (II) in Definition \ref{d:c}, which means that $H$ is either a subfield subgroup, or a twisted version of $G$. We begin by handling the subfield subgroups.

\begin{prop}\label{p:subfield}
Suppose $H$ is a subfield subgroup over $\mathbb{F}_{q_0}$, where $q=q_0^k$ and $k$ is a prime. Then either $qk$ is odd, in which case $|\O|$ is odd, or $T$ is $2$-elusive. In particular, every involution in $T$ has fixed points.
\end{prop}

\begin{proof}
We can assume that $|\O|$ is even. If $T$ has a unique conjugacy class of involutions, then $T$ is $2$-elusive by Lemma \ref{l:uniqueclasselusive}. So this takes care of the groups with $T = {}^2B_2(q)$ or ${}^2G_2(q)'$ (see Lemma \ref{l:b2g2classes}), and also $T = {}^3D_4(q)$ with $q$ odd (Lemma \ref{l:3d4qclasses}). 

In the remaining cases, by inspecting Tables \ref{table:3D4q},  \ref{table:evenpinvolutions} and \ref{table:oddpinvolutions}, we see that every conjugacy class of involutions in $T$ has a representative that can be written as a product of root elements of the form $x_{\alpha}( \pm 1)$. Since the scalars $\pm 1$ are contained in the prime field $\mathbb{F}_p$, it follows that the subfield subgroup $H$ intersects every conjugacy class of involutions in $T$ and therefore $T$ is $2$-elusive.
\end{proof}

\begin{prop}\label{p:twisted}
Suppose that $H$ is of the same type as $T$, but a twisted version. Then $T$ is not $2$-elusive if and only if $T = F_4(q)$, $p=2$ and $H_0 = {}^2F_4(q_0)$ with $q = q_0^2$.
\end{prop}

\begin{proof}
By arguing as in the proof of Proposition \ref{p:subfield}, we may assume $T$ has at least two classes of involutions. Therefore, $q=q_0^2$ and either $(T,H_0) = (F_4(q), {}^2F_4(q_0))$ with $p=2$, or $(T,H_0) = (E_6(q), {}^2E_6(q_0))$.

If $T = F_4(q)$ and $p=2$, then neither of the two $T$-classes of root elements are $\psi$-invariant (see Table \ref{table:evenpinvolutions}), where $\psi$ is an exceptional isogeny of the algebraic group $\bar{G} = F_4$. Therefore, $H_0 = {}^2F_4(q_0) = C_T(\psi)$ does not contain any involutions of type $A_1$ nor $\tilde{A_1}$. 

Now assume $T = E_6(q)$ and $H_0 = {}^2E_6(q_0) = C_T(\gamma)$, where $\gamma = \s_{q_0}\tau$ is an involutory graph-field automorphism of $T$. As noted in the proof of Proposition \ref{p:subfield}, each class of involutions in $T$ is stable under the field automorphism $\s_{q_0}$. And by inspecting Tables \ref{table:evenpinvolutions} and \ref{table:oddpinvolutions}, we see that each class is also stable under $\tau$, whence $H_0$ meets every $T$-class of involutions and thus $T$ is $2$-elusive.
\end{proof}

\section{The low rank groups}\label{s:proof_low}

In this section, our goal is to complete the proof of Theorem \ref{t:main} for the low rank groups with socle one of the following:
\[
{}^2B_2(q), \; {}^2G_2(q)',\; G_2(q)', \; {}^2F_4(q)', \; {}^3D_4(q).
\]

\begin{prop}\label{l:1}
If $T \in \{{}^2B_2(q), {}^2G_2(q)', G_2(q)'\}$ and $|\O|$ is even, then $T$ is $2$-elusive unless $(T,H_0) = (G_2(4), {\rm L}_{2}(13))$.
\end{prop}

\begin{proof}
If $T \in \{ {}^2B_2(q), {}^2G_2(q)'\}$, or if $T = G_2(q)'$ with $q=2$ or $q$ odd, then $T$ has a unique class of involutions and the result follows from Lemma \ref{l:uniqueclasselusive}. So for the remainder we may assume $T = G_2(q)$ and $q \geqs 4$ is even, in which case $T$ has two classes of involutions corresponding to long and short root elements, labelled $A_1$ and $\tilde{A}_1$ in Table \ref{table:evenpinvolutions}. The possibilities for $H$ are conveniently listed in \cite[Table 8.30]{BHR} and we note that $H$ is non-parabolic since $|\O|$ is even (see Theorem \ref{t:odd}). In view of Proposition \ref{p:subfield}, we may also assume that $H$ is not a subfield subgroup. 

First assume that $H \in \mathcal{S}$, so $q = 4$ and $H_0 = {\rm L}_{2}(13)$ or ${\rm J}_2$ (see Remark \ref{r:CS}(d)). If $H_0 = {\rm L}_{2}(13)$ then $H_0$ has a unique class of involutions and thus $T$ is not $2$-elusive (for the record, the involutions in $H_0$ are of type $A_1$). On the other hand, if $H_0 = {\rm J}_2$ then $H_0$ has two classes of involutions and one can check (with the aid of {\sc Magma} \cite{magma}, for example) that they are not fused in $T$, so $T$ is $2$-elusive. For the remainder, we may assume $H \in \mathcal{C}$.

Suppose $H_0 = {\rm SL}_{3}^{\e}(q).2 = N_T(\bar{H}_{\s})$, where $\bar{H} = A_2.2$ is a $\s$-stable maximal rank subgroup of $\bar{G}$. If $V$ denotes the $6$-dimensional irreducible module $L_{\bar{G}}(\varpi_1)$ for $\bar{G}$, then $V \downarrow \bar{H}^{\circ} = W \oplus W^*$, where $W$ is the natural module for $\bar{H}^{\circ} = A_2$. If $x \in \bar{H}^{\circ}$ is an involution, then it has Jordan form $(2,1)$ on $W$, and hence $(2^2,1^2)$ on $V$. On the other hand, if $y \in \bar{H}$ is an involutory graph automorphism of $\bar{H}^{\circ}$, then $y$ interchanges $W$ and $W^*$, and therefore has Jordan form $(2^3)$ on $V$. It follows that $x$ and $y$ are not $T$-conjugate and thus $T$ is $2$-elusive.

Finally, if $H_0 = {\rm SL}_{2}(q) \times {\rm SL}_2(q)$ then involutions of the form $(x,1)$ and $(1,y)$ in $H_0$ are not $T$-conjugate (indeed, the two ${\rm SL}_2(q)$ factors are generated by long and short root elements, respectively). Therefore $T$ is $2$-elusive.   
\end{proof}

Next we turn to the large Ree groups with socle $T = {}^2F_4(q)'$. Here the maximal subgroups of $G$ were determined by Wilson \cite{Wil84} (for $q=2$, noting the omission of a maximal subgroup ${\rm SU}_3(2).2$ of ${}^2F_4(2)$) and Malle \cite{Malle} (for $q \geqs 8$, noting the omission of $3$ conjugacy classes of maximal subgroups ${\rm PGL}_2(13)$ in ${}^2F_4(8)$, as observed by Craven \cite[Remark 4.11]{Craven}). For $q=2$, recall that $\mathcal{S}$ comprises the maximal subgroups with socle ${\rm Alt}_6$, ${\rm L}_2(25)$ or ${\rm L}_3(3)$ (see Remark \ref{r:CS}(c)). And for $q \geqs 8$, the collection $\mathcal{S}$ is empty unless $G = {}^2F_4(8)$, in which case it comprises the three conjugacy classes of subgroups isomorphic to ${\rm PGL}_2(13)$.

The special case $q=2$ can be handled using {\sc Magma}.

\begin{prop}\label{l:2f42}
Suppose $T = {}^2F_4(2)'$ and $|\O|$ is even.

\vspace{1mm}
 
\begin{itemize}\addtolength{\itemsep}{0.2\baselineskip}
\item[{\rm (i)}] If $H \in \mathcal{C}$, then $T$ is $2$-elusive unless  
$H_0 = 3^{1+2}{:}{\rm D}_8$ or $13{:}6$.
\item[{\rm (ii)}] If $H \in \mathcal{S}$, then $T$ is $2$-elusive if and only if  
$H_0 = {\rm Alt}_6.2^2$.
\end{itemize}
\end{prop}

\begin{proof}
This is a straightforward {\sc Magma} computation. To do this, we first use the function \texttt{AutomorphismGroupSimpleGroup} to construct $G$ as a permutation group of degree $1755$. We then use the command \texttt{MaximalSubgroups} to construct a representative of each conjugacy class of core-free maximal subgroups $H$ of $G$. We can then take a set of representatives for the conjugacy classes of involutions $x \in H_0 = H \cap T$ and we can read off the corresponding $T$-class by computing $|C_T(x)|$. In this way, it is easy to verify the result.

Finally, note that the maximality of $H$ implies that $G = T.2$ for $H_0 \in \{3^{1+2}{:}{\rm D}_8, 13{:}6\}$ in part (i), whereas $G = T$ in (ii) with $H_0 =  {\rm Alt}_6.2^2$.
\end{proof}

\begin{prop}\label{l:3}
Suppose $T = {}^2F_4(q)$ and $|\O|$ is even, where $q \geqs 8$.

\vspace{1mm}
 
\begin{itemize}\addtolength{\itemsep}{0.2\baselineskip}
\item[{\rm (i)}] If $H \in \mathcal{C}$, then $T$ is $2$-elusive unless $H_0 = {\rm PGU}_3(q).2$, ${\rm SU}_3(q).2$, $(q+1)^2{:}{\rm GL}_2(3)$ or $(q^2 \pm \sqrt{2q^3}+q \pm \sqrt{2q}+1){:}12$.
\item[{\rm (ii)}] If $H \in \mathcal{S}$, then $T$ is not $2$-elusive. 
\end{itemize}
\end{prop}

\begin{proof}
First recall that $T$ has two classes of involutions, say $x^T$ and $y^T$, labelled $A_1\tilde{A}_1$ and $(\tilde{A}_1)_2$ in Table \ref{table:evenpinvolutions}, with $|C_T(x)| = q^9|{\rm SL}_2(q)|$ and $|C_T(y)| = q^{10}|{}^2B_2(q)|$, as recorded in  \cite[Table 22.2.5]{LS_book}. In particular, $y$ does not commute with any element in $T$ of order $3$. By Theorem \ref{t:odd}(i) and Proposition \ref{p:subfield}, we may assume that $H$ is neither a parabolic nor a subfield subgroup of $T$. We now inspect the remaining possibilities in turn.

First assume $H \in \mathcal{S}$, so $G = {}^2F_4(8)$ and $H = {\rm PGL}_{2}(13)$, as noted above. Here $H$ has two classes of involutions, say $a^H$ and $b^H$, with $|C_H(a)| = 24$ and $|C_H(b)| = 28$, and it follows that $a$ is in the $T$-class labelled $A_1\tilde{A}_1$ since $|C_T(y)|$ is indivisible by $3$. In order to determine the $T$-class of $b$, we can work with Craven's construction $H < G < {\rm GL}_{26}(8)$ in {\sc Magma}, which is defined in terms of the action of $G$ on the $26$-dimensional minimal module $V_{26}$ over $\mathbb{F}_8$. Here explicit matrices generating $H$ are given in the supplementary file \path{ConstructPGL213in2F48.txt} to \cite{Craven}. We find that both $a$ and $b$ have Jordan form $(2^{12}, 1^2)$ on $V_{26}$, so by inspecting Table \ref{table:evenpinvolutions} we deduce that both involutions are in the $\bar{G}$-class labelled $A_1\tilde{A}_1$. In particular, every involution in $H$ is of type $A_1\tilde{A}_1$ and thus $T$ is not $2$-elusive.

For the remainder, we may assume $H \in \mathcal{C}$, noting that the possibilities for $H$ are recorded in the main theorem of \cite{Malle}. 

First assume $H_0 = {\rm PGU}_3(q).2$. If $t \in H_0$ is an involutory graph automorphism of ${\rm PGU}_3(q)$, then $|C_{{\rm PGU}_3(q)}(t)| = |{\rm SL}_2(q)|$ is divisible by $3$, which places $t$ in the class labelled $A_1\tilde{A}_1$. And if $t \in {\rm PGU}_3(q)$ is an involution, then $|C_{{\rm PGU}_3(q)}(t)| = q^3(q+1)$ does not divide $|C_T(y)| = q^{10}|{}^2B_2(q)|$, so once again $t$ is an $A_1\tilde{A}_1$ involution. It follows that $T$ is not $2$-elusive. The case $H_0 = {\rm SU}_3(q).2$ is entirely similar and the same conclusion holds. 

Next assume $H_0 = {}^2B_2(q) \wr {\rm Sym}_2$ and set $B_0 = {}^2B_2(q)^2$. In this paragraph, we remind the reader that $C_{\ell}$ denotes the simple algebraic group ${\rm Sp}_{2\ell}(K)$, rather than a cyclic group of order $\ell$ (for the latter we use $Z_{\ell}$ (or just $\ell$), as stated in Section \ref{ss:nota}). If $t \in H_0 \setminus B_0$ is an involution, then $C_{B_0}(t) = {}^2B_2(q)$ and we deduce that $t$ is in the class $(\tilde{A}_1)_2$. If we take $t = (t_1,t_2) \in B_0$, where both $t_1$ and $t_2$ are involutions, then we find that $t \in C_2 \times C_2 < C_4$ embeds in $C_4$ as a $c_4$-type involution in the notation of \cite{AS} (see Remark \ref{r:inv_class}(b) and the proof of Lemma 4.17 in \cite{BT}). Then by considering the embedding of $C_4$ in $\bar{G}$, we deduce that $t$ is in the $T$-class labelled $A_1\tilde{A}_1$ and we conclude that $T$ is $2$-elusive. The case $H_0 = {\rm Sp}_4(q){:}2$ is similar: $c_2$-type  involutions in ${\rm Sp}_4(q)<H_0$ are in the $A_1\tilde{A}_1$ class, while those of type $a_2$ are in $(\tilde{A}_1)_2$. 

To complete the proof, we may assume $H$ is the normalizer of a maximal torus of $G$. First assume $H_0 = (q^2 \pm \sqrt{2q^3}+q \pm \sqrt{2q}+1){:}12$. Here Lemma \ref{l:conj} implies that $H_0$ has a unique class of involutions and thus 
$T$ is not $2$-elusive. (In fact, by appealing to \cite[Corollary 4.4]{BT}, we see that every involution in $H_0$ is contained in the $T$-class labelled $A_1\tilde{A}_1$.)

Next suppose $H_0 = N_T(S) = S{:}[96]$, where $S = (q +\e\sqrt{2q}+1)^2$ and $\e = \pm$, with $q \geqs 32$ if $\e=-$. By \cite[Corollary 4.4]{BT}, there is an involution $t_1 \in H_0$ that inverts the maximal torus $S$ and this is in the $T$-class labelled $A_1\tilde{A}_1$. In addition, since
\[
S < {}^2B_2(q)^2 < {}^2B_2(q) \wr {\rm Sym}_2 < T
\]
we see that $H_0$ also contains an involution $t_2$ interchanging the two factors of $S$. This implies that $|C_{H_0}(t_2)|$ is divisible by $q + \e\sqrt{2q}+1$, which places $t_2$ in the $T$-class labelled $(\tilde{A}_1)_2$. In particular, $T$ is $2$-elusive.

Finally, let us assume $H_0 = N_T(S) = S{:}{\rm GL}_2(3)$, where $S =  (q+1)^2$. Let $t \in H_0$ be an involution. If $t$ inverts $S$, then  \cite[Corollary 4.4]{BT} implies that $t$ is contained in the $T$-class $A_1\tilde{A}_1$. If not, then $t$ either inverts one of the factors of $S$ (and centralizes the other), or it interchanges the two factors. In both cases, it follows that $|C_{H_0}(t)|$ is divisible by $q+1$ and this means that $t$ is of type $A_1\tilde{A}_1$. We conclude that every involution in $H_0$ is of type $A_1\tilde{A}_1$ and thus $T$ is not $2$-elusive.
\end{proof}

To conclude our analysis of the low rank groups, we may assume $T = {}^3D_4(q)$. Here the maximal subgroups of $G$ were determined by Kleidman \cite{K} and we refer to \cite[Table 8.51]{BHR} for a convenient list of the subgroups that arise.

\begin{prop}\label{l:4}
Suppose $T = {}^3D_4(q)$ and $|\O|$ is even. Then $H \in \mathcal{C}$ and $T$ is $2$-elusive unless $q$ is even and either $H_0 = (q^2 \pm q +1)^2.{\rm SL}_2(3)$ or $(q^4-q^2+1).4$, or $q \geqs 4$ and $H_0 = {\rm PGL}_3^{\e}(q)$ with $q \equiv \e \imod{3}$.
\end{prop}

\begin{proof} 
First recall that $T$ has a unique class of involutions when $q$ is odd (see Lemma \ref{l:3d4qclasses}). So in view of Lemma \ref{l:uniqueclasselusive}, we may assume that $q$ is even. Here $T$ has two classes of involutions, labelled $A_1$ and $A_1^3$ in Table \ref{table:3D4q}, with respective centralizer orders $q^{12}(q^6-1)$ and $q^{10}(q^2-1)$. Note that $H \in \mathcal{C}$ since the collection $\mathcal{S}$ is empty (see Remark \ref{r:CS}(b)). We proceed by inspecting the cases arising in \cite[Table 8.51]{BHR}. Note that $|\O|$ is odd if $H$ is a parabolic subgroup (see Proposition \ref{t:odd}), while Proposition \ref{p:subfield} applies if $H$ is a subfield subgroup.

If $H_0 = {\rm PGL}_3^{\e}(q)$, then $H_0$ has a unique class of involutions and thus $T$ is not $2$-elusive. Similarly, if $H_0$ is the normalizer of a maximal torus, then $H_0 = (q^2 \pm q +1)^2.{\rm SL}_2(3)$ or $(q^4-q^2+1).4$, and Lemma \ref{l:conj} implies that $H_0$ has a unique class of involutions. So once again, we conclude that $T$ is not $2$-elusive.

Next suppose $H_0 = G_2(q)$ and note that $H_0$ has two classes of involutions, labelled $A_1$ and $\tilde{A}_1$ in Table \ref{table:evenpinvolutions}. We can embed the ambient algebraic group $\bar{H} = G_2$ in $\bar{G}$ by first embedding it in a subgroup $L = C_3$ of $\bar{G}$ (here $\bar{H}$ acts irreducibly on the natural module for $L$, while $L$ is embedded in $\bar{G}$ by restricting one of the two $8$-dimensional spin modules for $\bar{G}$). Under the embedding $\bar{H} < L$, we find that involutions in the $A_1$ class of $H_0$ are of type $a_2$ when viewed as elements of $L$ (with respect to the notation of \cite{AS}), while those in the other class are of type $b_3$. And then by considering the embedding of $L$ in $\bar{G}$, we deduce that $H_0$ meets both $T$-classes of involutions and thus $T$ is $2$-elusive. 

Next suppose $H_0 = {\rm L}_2(q^3) \times {\rm L}_2(q)$. Clearly, if $t = (1,s) \in H_0$ is an involution, then ${\rm L}_2(q^3) \leqs C_{H_0}(t)$ and thus $t$ is in the $T$-class $A_1$. We claim that if $t = (s,1) \in H_0$ is an involution, then $t$ is in the other class of involutions and thus $T$ is $2$-elusive. To see this, it will be helpful to view $H_0$ in terms of the following embedding:
\[
H_0 = \{(x,x^{\psi},x^{\psi^2},y) \,:\, x \in {\rm L}_2(q^3), y \in {\rm L}_2(q)\} < A_1^4 = D_2^2 < D_4 = \bar{G},
\]
where $\psi$ is an order $3$ field automorphism of ${\rm L}_2(q^3)$. Here we view $D_2 = A_1^2$, identifying the natural $4$-dimensional module for $D_2$ with the tensor product of the natural modules for the two $A_1$ factors. Then in terms of this embedding, $t = (s,1) \in H_0$ is  of the form $(J_2,J_2,J_2,I_2)$ and therefore has Jordan form 
\[
(J_2 \otimes J_2) \oplus (J_2 \otimes I_2) = (J_2^4)
\]
on the natural module for $\bar{G}$. By inspecting Table \ref{table:3D4q} we deduce that $t$ is contained in the $T$-class labelled $A_1^3$, as claimed. 

To complete the proof, we may assume $H_0 = J.2$, where $J= ((q^2+\e q+1) \circ {\rm SL}_3^{\e}(q)).f_{\e}$ and $f_{\e} = (3,q^2+\e q+1)$. As noted in \cite[Table II]{K}, we have $J = C_T(x)$ and $H_0 = N_T(J) = N_T(\langle x \rangle)$ for a certain semisimple element $x \in T$ of order $q^2+\e q+1$. Since the centralizer of an involution of type $A_1^3$ has order $q^{10}(q^2-1)$, which is not divisible by $q^2+\e q+1$, it follows that every involution in $J$ is of type $A_1$.

We claim that $H_0$ also contains an involution of type $A_1^3$, which means that $T$ is $2$-elusive. To see this, we will work with a more precise description of the structure of $H_0$. First observe that the representative for the conjugacy class of $x$ used in \cite{K} is $y = s_4$ for $\e = +$ and $y = s_9$ for $\e = -$, where $s_4$ and $s_9$ are defined as follows in \cite[Table 2.1]{DM}: 
\begin{align*} 
s_4 &= h_{\alpha_1}(t) h_{\alpha_3}(t^q) h_{\alpha_4}(t^{q^2}), \mbox{ where $t \ne 1$ and 
$t^{q^2+q+1} = 1$} \\
s_9 &= h_{\alpha_1}(t) h_{\alpha_3}(t^{-q}) h_{\alpha_4}(t^{q-1}), \mbox{ where $t \neq 1$ and  $t^{q^2-q+1} = 1$}
\end{align*} 
with respect to the notation we defined in Section \ref{ss:setup}. In both cases,  $C_{\bar{G}}(y) = T_2A_2$ is a subsystem subgroup that corresponds to a subsystem of $\Phi$ with base $\{\alpha_2, -\alpha_0\}$ (denoted by $J_3$ in \cite[Table 1.0]{DM}), where $\alpha_0$ is the longest root of $\Phi$.

Next let $w \in N_{\bar{G}}(\bar{T})$ be an element with image $-s_{\alpha_1 + 2\alpha_2 + \alpha_3 + \alpha_4}$ in the Weyl group $W$ of $\bar{G}$. (Note that $W$ contains a central involution and we have $-s_{\alpha_1 + 2\alpha_2 + \alpha_3 + \alpha_4} = -w_{1+2}$ in the notation of \cite{DM}.) Since $-s_{\alpha_1 + 2\alpha_2 + \alpha_3 + \alpha_4} = s_{\alpha_1} s_{\alpha_3} s_{\alpha_4}$, we can choose $w = w_{\alpha_1} w_{\alpha_3} w_{\alpha_4}$. Note that $w$ acts on $\Phi$ by mapping $\alpha_1 \mapsto -\alpha_1$, $\alpha_3 \mapsto -\alpha_3$ and $\alpha_4 \mapsto -\alpha_4$, so $y \in \{s_4,s_9\}$ is inverted by $w$. In addition, note that $w_{\alpha}^2 = h_{\alpha}(-1) = 1$ for all $\alpha \in \Phi$ (recall that $p=2$) and thus $w$ is an involution.

It follows that the normalizer of $C_{\bar{G}}(y)$ is equal to $\bar{H} = C_{\bar{G}}(y).2 = C_{\bar{G}}(y){:}\langle w \rangle$. Now following \cite[Table 2.1]{DM}, we have $H = N_{G}(\left(\bar{H}^g\right)_\s)$, where under the bijection of Lemma \ref{l:sigmaclasses}, the $\sigma$-invariant subgroup $\bar{H}^g$ corresponds to the image of $n \in \bar{G}$ in $H^1(\sigma, \bar{H}/\bar{H}^\circ)$, with $$n = \begin{cases} 1, & \text{ if } \e = +, \\ w, & \text{ if } \e = -.\end{cases}$$ Thus by Lemma \ref{l:tildeH0}, it will suffice to consider involutions in $\widetilde{H_0} = \bar{H}_{n \s} \cap \left(\bar{G}_{n \s}\right)' = \bar{H}_{n \s}$. Then $w \in \widetilde{H_0}$, since $w = w_{\alpha_1} w_{\alpha_3} w_{\alpha_4}$ is fixed by both $\s$ and $n \in \{1,w\}$.

Finally, we will now verify that $w$ is an involution of type $A_1^3$, which will complete the proof of the proposition. Let $w'$ be the involution corresponding to $w$ in the simply connected cover of $\bar{G}$. As explained in Section \ref{ss:comp}, we can use {\sc Magma} to show that $w'$ has Jordan form $(J_2^4)$ on the $8$-dimensional natural module for $\bar{G}$. Then by inspecting Table \ref{table:3D4q}, we see that $w$ is in the $T$-class labelled $A_1^3$ and hence $T$ is $2$-elusive.
\end{proof}

\section{The groups with socle $F_4(q)$}\label{s:proof_F4} 

Next we handle the groups with socle $T = F_4(q)$.  The maximal subgroups of $G$ have been determined up to conjugacy by Craven \cite{Craven}, extending earlier work of Norton and Wilson \cite{NW} for $q=2$. In particular, the subgroups comprising the collection $\mathcal{C}$ are listed in \cite[Tables 7 and 8]{Craven}, while those in $\mathcal{S}$ are presented in \cite[Table 1]{Craven}. Referring to Section \ref{ss:invols}, let us recall that $T$ has exactly $2+2\delta_{2,p}$ classes of involutions.

\begin{rem}\label{r:torn}
Here we take the opportunity to point out that $(q^2+1)^2.({\rm SL}_2(3){:}4)$ is the correct structure of the torus normalizer expressed as $(q^2+1)^2.(4 \circ {\rm GL}_2(3))$ in \cite[Table 8]{Craven}; the source of this error is \cite[Table 2]{LSS}. Referring to the bijection in Lemma \ref{l:torus1}, the maximal torus $(q^2+1)^2$ corresponds to a certain conjugacy class of elements of order $4$ in the Weyl group $W$ of $\bar{G} = F_4$, with representative $w_{(22)}$ given in \cite[p.93]{LawtherB4F4}. A computation shows that such an element has centralizer ${\rm SL}_2(3){:}4$ in $W$, which implies that the torus normalizer has structure $(q^2+1)^2.({\rm SL}_2(3){:}4)$ (see Lemma \ref{l:torus2}). In fact, one can check that $W$ does not have a subgroup isomorphic to $4 \circ {\rm GL}_2(3)$.
\end{rem}

We begin by handling the subgroups in $\mathcal{S}$.

\begin{prop}\label{p:f4_1}
Suppose $T = F_4(q)$ and $|\O|$ is even. If $H \in \mathcal{S}$, then $T$ is $2$-elusive if and only if $q=2$ and $H_0 = {\rm L}_4(3).2_2$, or if $q=p \geqs 3$ and $H_0 = {}^3D_4(2).3$.
\end{prop}

\begin{proof}
The possibilities for $H_0$ are listed in \cite[Table 1]{Craven} and we recall that $T$ has $2(1+\delta_{2,p})$ classes of involutions (see Tables \ref{table:evenpinvolutions} and \ref{table:oddpinvolutions}). Of course, if $H_0$ contains a unique class of involutions, then $T$ is not $2$-elusive. And similarly if $q$ is even and $H_0$ has at most three involution classes. So by inspecting \cite{Craven}, we quickly deduce that $T$ is $2$-elusive only if one of the following holds:

\vspace{1mm}

\begin{itemize}\addtolength{\itemsep}{0.2\baselineskip}
\item[(a)] $H_0 = {\rm L}_4(3).2_2$, $q=2$; 
\item[(b)] $H_0 = {}^3D_4(2).3$, $q=p \geqs 3$; 
\item[(c)] $H_0 = {\rm PGL}_2(13)$ and either $q=7$, or $q=p \ne 13$ and $p \equiv \pm 1 \imod{7}$, or $q=p^3$ is odd and $p \equiv \pm 2, \pm 3 \imod{7}$. 
\end{itemize}

Consider case (a). Here the notation indicates that $H_0$ contains involutory graph automorphisms of ${\rm L}_4(3)$ with centralizer ${\rm PGSp}_4(3)$ and we find that $H_0$ has precisely $5$ classes of involutions, with representatives $x_1, \ldots, x_5$. Let $V_1 = L_{\bar{G}}(\varpi_1)$ and $V_2 = L_{\bar{G}}(\varpi_4)$ be the two $26$-dimensional irreducible modules for $\bar{G}$, and note that $V_1$ and $V_2$ are interchanged by a graph automorphism of $\bar{G}$. As recorded in \cite[Table 6.32]{Litt}, $H_0$ acts irreducibly on $V_1$ and $V_2$. Moreover, both modules for $H_0$ can be constructed using {\sc Magma} (see Remark \ref{r:cvx}, for example) and this allows us to compute the respective Jordan forms of each involution in $H_0$. This is detailed in Table \ref{tab:JF4}, up to some choice of labelling of the involution class representatives in $H_0$. By considering Table \ref{table:evenpinvolutions} we immediately deduce that $x_1 \in A_1$, $x_2 \in \tilde{A}_1$ and $x_4,x_5 \in A_1\tilde{A}_1$ (up to a choice of labelling for $x_1$ and $x_2$). In addition, since $x_3$ has the same Jordan form on $V_1$ and $V_2$, it follows that $x_3 \in (\tilde{A}_1)_2$, and we conclude that $T$ is $2$-elusive.

{\scriptsize
\begin{table}
\[
\begin{array}{cccc} \hline
i & |x_i^{H_0}| & \mbox{Jordan form on $V_1$} & \mbox{Jordan form on $V_2$} \\ \hline
1 & 117 & (2^6,1^{14}) & (2^{10},1^{6}) \\
2 & 117 & (2^{10},1^{6}) & (2^6,1^{14}) \\
3 & 2106 & (2^{10},1^{6}) & (2^{10},1^{6}) \\
4 & 5265 & (2^{12},1^{2}) & (2^{12},1^{2}) \\
5 & 10530 & (2^{12},1^{2}) & (2^{12},1^{2}) \\ \hline
\end{array}
\]
\caption{The case $H_0 = {\rm L}_4(3).2_2 < F_4(2)$}
\label{tab:JF4}
\end{table}
}

Next let us turn to case (b) above, so $q = p$ is odd. Here both $H_0$ and $T$ have two classes of involutions, with every involution of $H_0$ contained in $\operatorname{soc}(H_0) = {}^3D_4(2)$. Moreover, we note that $H_0$ acts irreducibly on the $52$-dimensional adjoint module $V = \mathcal{L}(\bar{G})$ (see \cite[6.1.35, Table 6.36]{Litt}). By inspecting the Brauer character table of $H_0$ (see \cite[pp.250--253]{ModularAtlas}), we observe that a $52$-dimensional irreducible $K[H_0]$-module arises as the reduction modulo $p$ of a $52$-dimensional representation in characteristic zero. Thus the dimension of the fixed point spaces on $V$ for each involution in $H_0$ can be read off from the character table of $H_0$. From this, it follows that the two $H_0$-classes are not fused in $T$, and thus $T$ is $2$-elusive.

Finally, let us consider case (c). As discussed in Section \ref{ss:feasible}, we can use the {\sc Magma} code from \cite{LittGithub} to compute the feasible characters of $H_0 = {\rm PGL}_2(13)$ on $V = \mathcal{L}(\bar{G})$. For each feasible character, we check that every involution in $H_0$ has trace $-4$ on $V$, which implies that every involution in $H_0$ is of type $A_1C_3$. In particular, $T$ is not $2$-elusive.
\end{proof}

We now complete the proof of Theorem \ref{t:main} for $T = F_4(q)$ by handling the cases with $H \in \mathcal{C}$. Recall that the possibilities for $H$ are listed in \cite[Tables 7 and 8]{Craven} (see Remark \ref{r:torn}).

\begin{prop}\label{p:f4_2}
Suppose $T = F_4(q)$ and $|\O|$ is even. If $H \in \mathcal{C}$ then $T$ is $2$-elusive unless one of the following holds:

\vspace{1mm}

\begin{itemize}\addtolength{\itemsep}{0.2\baselineskip}
\item[{\rm (i)}] $q$ is odd and $H_0$ is one of the following:
\[
{}^3D_4(q).3, \; {\rm PGL}_2(q) \, (p \geqs 13), \; G_2(q) \, (p=7), \; {\rm ASL}_3(3) \, (q=p \geqs 5).
\]
\item[{\rm (ii)}] $q$ is even and $H_0$ is one of the following:
\[
{}^2F_4(q_0), \; ({\rm SL}_3^{\e}(q) \circ {\rm SL}_3^{\e}(q)).(3,q-\e).2, \; {\rm Sp}_4(q^2).2,
\]
\[
(q^2 \pm q +1)^2.(3 \times {\rm SL}_2(3)), \; (q^4-q^2+1).12, \; (q^2+1)^2.(\operatorname{SL}_2(3){:}4),
\]
where $q=q_0^2$.
\end{itemize}
\end{prop}

\begin{proof}
By Propositions \ref{p:parab} and \ref{p:subfield}, we know that $T$ is $2$-elusive if $H$ is a parabolic or subfield subgroup, while Proposition \ref{p:twisted} shows that $T$ is not $2$-elusive when $H_0 = {}^2F_4(q_0)$ with $q = q_0^2$. For the remainder of the proof, it will be convenient to write $V = \mathcal{L}(\bar{G})$ for the $52$-dimensional adjoint module for $\bar{G}$.

\vs

\noindent \emph{Case 1. $H$ is a maximal rank subgroup.}

\vs

We begin by considering the subgroups of the form $H = N_G(\bar{H}_{\s})$, where $\bar{H}$ is a maximal rank subgroup. By inspecting \cite[Tables 5.1, 5.2]{LSS}, we see that the possibilities for $\bar{H}$ are as follows:
\begin{equation}\label{e:f4list1}
B_4 \, (p \ne 2), \, C_4 \, (p=2), \, D_4.{\rm Sym}_3, \, A_2^2.2, \, A_1C_3 \, (p \ne 2), \, C_2^2.2 \, (p=2), \, T_4.W \, (p=2)
\end{equation}
where $T_4 = \bar{T}$ is a maximal torus of $\bar{G}$ and $W = {\rm O}_4^{+}(3)$ is the Weyl group of $\bar{G}$. 

In each of these cases, we are free to assume that $H = N_G(\left( \bar{H}^g \right)_{\s})$, where $\bar{H}$ is the normalizer of a standard subsystem subgroup 
\[
\bar{H}^\circ = \langle \bar{T}, U_{\alpha} \,:\, \alpha \in \Phi' \rangle,
\]
and $\bar{H}^g$ is a $\s$-invariant conjugate of $\bar{H}$. Here $\bar{T}$ is a maximal torus of $\bar{G}$ as in the setup of Section \ref{ss:setup}, and $\Phi'$ is a root subsystem of $\Phi$. In each case, the precise structure of $H$ is recorded in \cite[Tables 5.1, 5.2]{LSS} and we refer the reader to \cite[p.300]{LSS} for some details on how the structure of $H$ is determined from $\bar{H}$ (also see \cite[Tables 7 and 8]{Craven}). Referring to \eqref{e:f4list1}, we now consider each possibility for $\bar{H}$ in turn. 

\vs

\noindent \emph{Case 1.1. $\bar{H}^{\circ}$ is not a maximal torus of $\bar{G}$.}

\vs

First suppose $\bar{H} = B_4$, so $q$ is odd and the subsystem $\Phi'$ has base $\{-\alpha_0, \alpha_1, \alpha_2, \alpha_3\}$, where $\alpha_0$ is the longest root of $\Phi$. We have $H_0 = \bar{H}_{\s} = {\rm Spin}_9(q) = 2.\O_9(q)$, which is the centralizer in $T$ of a $B_4$-type involution. Moreover, $h_{\alpha_1}(-1) \in H_0$ is an involution of type $A_1C_3$ (Table \ref{table:oddpinvolutions}), so $T$ is $2$-elusive. By inspecting \cite[Table 14]{Law09}, we deduce that the same conclusion holds when $\bar{H} = C_4$ and $q$ is even, in which case $H_0 = {\rm Sp}_8(q)$.

Next assume $\bar{H} = D_4.{\rm Sym}_3$. Here the subsystem $\Phi'$ has base $\{-\alpha_0, \alpha_1, \alpha_2, \beta\}$, where $\beta = \alpha_2 + 2 \alpha_3 + 2 \alpha_4$, and $-\beta$ is the highest root of the $B_4$ subsystem discussed in the previous paragraph. In this case there are two possibilities for $H_0$. If $H_0 = {}^3D_4(q).3$ then either $q$ is odd and $H_0$ has a unique class of involutions, or $q$ is even and $H_0$ has two such classes (see Proposition \ref{l:3d4qclasses}); in both cases it is clear that $T$ is not $2$-elusive. Now assume $H_0$ is of type ${\rm P\O}_{8}^{+}(q)$, so we have  
\[
H_0 = \bar{H}_{\s} = {\rm Spin}_8^{+}(q).{\rm Sym}_3 = (2,q-1)^2.{\rm P\O}_{8}^{+}(q).{\rm Sym}_3.
\]
If $q$ is odd then by inspecting \cite[Table 13]{BT} we see that $H_0$ contains involutions of type $A_1C_3$. In addition the element $t = h_{-\alpha_0}(-1) h_{\alpha_1}(-1)$ is contained in $H_0$. A computation with {\sc Magma} (see Section \ref{ss:comp}) shows that $t$ has fixed point space of dimension $36$ on $V$, and thus is an involution of type $B_4$ (Table \ref{table:oddpinvolutions}). Therefore, $T$ is $2$-elusive when $q$ is odd. By considering the embeddings $D_4 < B_4 < F_4$ and inspecting \cite[Section 4.4]{Law09}, it is easy to check that $T$ is also $2$-elusive when $q$ is even. For example, if $t \in \O_8^{+}(q)$ is a $c_4$-type involution (in the notation of \cite{AS}; see Remark \ref{r:inv_class}(b)), then the $D_4$-class of $t$ is labelled $A_1+D_2$, which embeds in the $B_4$-class labelled $A_1+B_1^{(2)}$ in \cite[Table 4]{Law09}, and this in turn is contained in the $\bar{G}$-class $A_1\tilde{A}_1$.

Next suppose $\bar{H} = A_2^2.2$, in which case the subsystem $\Phi'$ has base $\{-\alpha_0, \alpha_1\} \cup \{\alpha_3,\alpha_4\}$. We have $\bar{H} = \langle \bar{H}^\circ, w \rangle$, where $w \in N_{\bar{G}}(\bar{T})$ corresponds to the central involution in the Weyl group of $F_4$ that acts as $\alpha \mapsto -\alpha$ on $\Phi$. Explicitly, we can choose $w = w_{\alpha_1}w_{\alpha_3}w_{\alpha_{14}}w_{\alpha_{21}}$, where $\alpha_i$ is the $i$-th root respect to the specific ordering of the roots of $\bar{G}$ used by {\sc Magma} (as described in Remark \ref{r:ordering}).

Now under the bijection of Lemma \ref{l:sigmaclasses}, the $\s$-invariant conjugate $\bar{H}^g$ corresponds to the image of $n \in \bar{H}$ in $H^1(\s, \bar{H}/\bar{H}^\circ)$, where $n \in \{1,w\}$. Then
\[
H_0 = ({\rm SL}_3^{\e}(q) \circ {\rm SL}_3^{\e}(q)).2 = e.({\rm L}_3^{\e}(q) \times {\rm L}_3^{\e}(q)).e.2
\]
where $e = (3,q-\e)$, and $\e = +$ or $\e = -$ according to whether $n = 1$ or $n = w$. Now by Lemma \ref{l:tildeH0}, it suffices to consider involutions in $\widetilde{H_0} = \bar{H}_{n\s}$. We have $h_{\alpha_1}(-1), h_{\alpha_4}(-1) \in \widetilde{H_0}$, since both elements are contained in $\bar{H}$ and are fixed by $w$ and $\s$. In other words, $\bar{H}$ contains representatives for the classes $A_1C_3$ and $B_4$ from Table \ref{table:oddpinvolutions}, so $T$ is $2$-elusive when $q$ is odd.  

Now assume $\bar{H} = A_2^2.2$ and $q$ is even, in which case $\bar{H} = \bar{H}^\circ{:}\langle w \rangle$ since $w$ is an involution. By inspecting \cite[Section 4.7]{Law09}, we observe that there are no involutions of type $(\tilde{A}_1)_2$ in $\bar{H}^{\circ}$. Since $w$ acts on $\Phi$ as $\alpha \mapsto -\alpha$, the action of $w$ on both $A_2$ factors of $\bar{H}^\circ = A_2^2$ is via the standard inverse-transpose graph automorphism. Hence it follows from \cite[Lemma 4.4.6]{GLS} that there is a unique $\bar{H}$-class of involutions in $\bar{H} \setminus \bar{H}^\circ$, and so every involution in $\bar{H} \setminus \bar{H}^\circ$ is conjugate to $w$. A computation with {\sc Magma} (see Section \ref{ss:comp}) shows that $w$ has Jordan form $(2^{24},1^4)$ on $V$, and from Table \ref{table:evenpinvolutions} we deduce that $w$ is contained in the class labelled $A_1\tilde{A}_1$. So for $\bar{H} = A_2^2.2$ we conclude that $T$ is $2$-elusive if and only if $q$ is odd.

Next suppose $\bar{H} = A_1C_3$, so $q$ is odd and $H_0$ is the centralizer of an $A_1C_3$-type involution. By inspecting \cite[Table 13]{BT}, we see that $H_0$ also contains involutions of type $B_4$ and thus $T$ is $2$-elusive. 

To complete the proof in Case 1.1, we may assume $q$ is even and $\bar{H} = C_2^2.2$. Here we have $H_0 = {\rm Sp}_4(q)^2.2 = {\rm Sp}_4(q) \wr {\rm Sym}_2$ and the 
maximality of $H$ implies that $G$ contains a graph (or graph-field) automorphism of $T$ (see \cite[Table 5.1]{LSS}). Now $H_0 < {\rm Sp}_8(q) < T$ and it is easy to check that every class of involutions in ${\rm Sp}_8(q)$ has a representative in $H_0$. Since we have already noted that $T$ is $2$-elusive when $H_0 = {\rm Sp}_8(q)$, we deduce that the same conclusion holds when $H_0 = {\rm Sp}_4(q)^2.2$. Similarly, if $H_0 = {\rm Sp}_4(q^2).2$ then by considering the natural embedding of $H_0$ in ${\rm Sp}_8(q)$ we deduce that every involution in $H_0$ is of type $a_4$ or $c_4$ as an element of ${\rm Sp}_8(q)$, which in turn implies that there are no involutions of type $A_1$ or $\tilde{A}_1$ in $H_0$. In particular, $T$ is not $2$-elusive in this case. 

\vs

\noindent \emph{Case 1.2. $\bar{H}^{\circ}$ is a maximal torus of $\bar{G}$.}

\vs

Now assume $H_0$ is the normalizer of a maximal torus of $T$, so $\bar{H} = \bar{T}.W$, $q$ is even and the maximality of $H$ implies that $G$ contains graph or graph-fields automorphisms (see \cite[Table 5.2]{LSS}). For the tori of order $(q^2 \pm q +1)^2$ and $q^4-q^2+1$, Lemma \ref{l:conj} implies that $H_0$ has a unique class of involutions and thus $T$ is not $2$-elusive. Similarly, $H_0 = (q^2+1)^2.({\rm SL}_2(3){:}4)$ has only two classes of involutions and so once again $T$ is not $2$-elusive. 

To complete the proof for normalizers of maximal tori, we may assume $H_0 = (q-\e)^4.W$. Let $w = w_{\alpha_1}w_{\alpha_3}w_{\alpha_{14}}w_{\alpha_{21}}$ be an element corresponding to the central involution of $W$ as in the case $\bar{H} = A_2^2.2$. Then under the bijection of Lemma \ref{l:sigmaclasses}, the $\s$-invariant conjugate $\bar{T}^g$ corresponds to the image of $n \in \bar{H}$ in $H^1(\s, \bar{H}/\bar{H}^\circ)$, where $n = 1$ if $\e = +$ and $n = w$ if $\e = -$.

By Lemma \ref{l:tildeH0}, it will suffice to consider involutions in $\widetilde{H_0} = N_{\bar{G}}(\bar{T})_{n\s}$. Note that $w$ is contained in the subgroup $W_0 = \langle w_{\alpha} \,:\, \alpha \in \Phi \rangle$, which is isomorphic to $W$ by Theorem \ref{t:titsweyl}. In particular, $W_0 \leqs \widetilde{H_0}$. With the aid of {\sc Magma} (see Section \ref{ss:comp}), we can calculate the Jordan form of each involution in $W_0$ on both the minimal module $V_{\operatorname{min}}$ and the adjoint module $V = \mathcal{L}(\bar{G})$. In this way, we can verify that $W_0$ meets every conjugacy class of involutions in $T$ and thus $T$ is $2$-elusive. More specifically, by computing the Jordan forms of the elements 
\[
w_{\alpha_1}, \; w_{\alpha_4}, \; w_{\alpha_1} w_{\alpha_4}, \; w_{\alpha_2 + \alpha_3} w_{\alpha_3}
\]
on $V_{\operatorname{min}}$ and $V$, we can read off from Table \ref{table:evenpinvolutions} that they are contained in the classes labelled $A_1$, $\tilde{A_1}$, $A_1\tilde{A_1}$, $(\tilde{A_1})_2$, respectively. 

\vs

\noindent \emph{Case 2. The remaining cases.}

\vs

By inspecting \cite[Tables 7 and 8]{Craven}, in order to complete the proof of the proposition we may assume $q$ is odd and $H_0$ is one of the following:
\[
{\rm PGL}_2(q) \times G_2(q) \, (q \geqs 5), \, {\rm PGL}_2(q) \, (p \geqs 13), \, G_2(q) \, (p=7), \, {\rm ASL}_3(3) \, (q=p \geqs 5).
\]

In the latter two cases, $H_0$ has a unique class of involutions and thus $T$ is not $2$-elusive. Next assume $H_0 = {\rm PGL}_2(q) \times G_2(q)$. Here $H_0$ contains involutions of type $B_4$ (see \cite[Table 18]{BT}) and from the fact that
\[
V \downarrow A_1G_2 = \mathcal{L}(A_1G_2)/(V_{A_1}(4) \otimes V_{G_2}(\omega_1))
\]
(see \cite[Table 12.2]{Thomas}) we deduce that an involution $t \in H_0$ of the form $(1,s)$ is of type $A_1C_3$. Indeed, $t$ has Jordan form $(-I_8,I_9)$ on $\mathcal{L}(A_1G_2) \cong V_{A_1}(2) \oplus V_{G_2}(\omega_2)$ and it acts on  
$V_{A_1}(4) \otimes V_{G_2}(\omega_1)$ as $I_5 \otimes (-I_4,I_3) = (-I_{20},I_{15})$, so $\dim C_V(t) = 24$ and the claim follows. Therefore, $T$ is $2$-elusive.
 
Finally, suppose $H_0 = {\rm PGL}_2(q)$ with $p \geqs 13$. In this case $H_0 = \bar{H}_{\s}$ where $\bar{H} = \bar{H}^\circ = A_1$. Since $\bar{H}$ has a unique conjugacy class of involutions, it follows that the two $H_0$-classes of involutions are fused in $\bar{G}$. In particular, $T$ is not $2$-elusive in this case.
\end{proof}

\section{The groups with socle $E_6^{\e}(q)$}\label{s:proof_E6}

In this section we establish Theorem \ref{t:main} for the groups with socle $E_6(q)$ or ${}^2E_6(q)$. In order to handle both cases simultaneously, we will write $T = E_6^{\e}(q)$, where $\e = +$ if $T = E_6(q)$, and $\e=-$ if $T = {}^2E_6(q)$. In addition, $V_{{\rm min}}$ will denote the $27$-dimensional minimal module $V_{\bar{G}_{\rm sc}}(\varpi_1)$ for the simply connected cover $\bar{G}_{\rm sc}$ of $\bar{G}$. In this section, we will occasionally determine the class of an involution $g \in \bar{G}$ by computing its action on $V_{{\rm min}}$, and by this we mean the action of the unique involution $g' \in \bar{G}_{\rm sc}$ that $g$ lifts to (see Remark \ref{r:e77}). It may be helpful to recall that $T = E_6^{\e}(q)$ has exactly $2+\delta_{2,p}$ classes of involutions (see Section \ref{ss:invols}).

We begin by considering the groups with $H \in \mathcal{S}$, noting that the possibilities for $H$ are recorded in \cite[Table 2]{Craven} (for $\e=+$) and \cite[Table 3]{Craven} (for $\e=-$).

\begin{prop}\label{p:e6_1}
Suppose $T = E_6^{\e}(q)$ and $|\O|$ is even. If $H \in \mathcal{S}$, then $T$ is $2$-elusive if and only if $H_0 = {}^2F_4(2)$ and $q=p \equiv \e \imod{4}$, or $(\e,q)=(-,2)$ and $H_0 = \O_7(3)$ or ${\rm Fi}_{22}$.
\end{prop}

\begin{proof}
By inspecting \cite[Tables 2 and 3]{Craven} and considering the number of classes of involutions in $H_0$, we immediately deduce that $T$ is $2$-elusive only if one of the following holds:

\vspace{1mm}

\begin{itemize}\addtolength{\itemsep}{0.2\baselineskip}
\item[(a)] $H_0 = {}^2F_4(2)$, $q=p \equiv \e \imod{4}$;
\item[(b)] $H_0 = {\rm M}_{12}$, $(\e,q)=(+,5)$; 
\item[(c)] $H_0 = \O_7(3)$ or ${\rm Fi}_{22}$, $(\e,q)=(-,2)$.
\end{itemize}

\vspace{1mm}

First consider (a). As recorded in \cite[Table 6.109]{Litt}, we note that $H_0$ lifts to a subgroup of $\bar{G}_{\rm sc}$ that acts irreducibly on $V_{{\rm min}}$. By inspecting the Brauer character table of $H_0$ (see \cite[pp.188--191]{ModularAtlas}), we deduce that the $27$-dimensional irreducible $K[H_0]$-modules arise as the reduction modulo $p$ of a $27$-dimensional irreducible representation in characteristic zero. Thus the dimension of the fixed point spaces on $V$ for each involution of $H_0$ can be read from the character table of $H_0$, and from this we deduce that $T$ is $2$-elusive.

Similarly, in (b) we find that $H_0$ acts irreducibly on the adjoint module $V = \mathcal{L}(\bar{G})$ (see \cite[Table 6.61]{Litt}) and using {\sc Magma} we calculate that every involution in $H_0$ acts as $(-I_{40},I_{38})$ on $V$. This means that every involution is of type $A_1A_5$ and thus $T$ is not $2$-elusive. 

Finally, consider the cases in (c). Here $H_0$ acts irreducibly on $V = \mathcal{L}(\bar{G})$ (see \cite[6.2.67, 6.105]{Litt}) and using {\sc Magma} we can compute the Jordan form of every involution in $H_0$. In both cases, $H_0$ has three classes of involutions, labelled \texttt{2A}, \texttt{2B} and \texttt{2C} in the Atlas \cite{Atlas}, and we find that each class representative has the following Jordan form on $V$:
\[
\texttt{2A}{:}\; (2^{22},1^{34}), \;\; \texttt{2B}{:}\; (2^{32},1^{14}), \;\; \texttt{2C}{:}\; (2^{38},1^{2}).
\]
In this way, we conclude that $T$ is $2$-elusive.
\end{proof}

For the remainder of this section, we may assume $H \in \mathcal{C}$. Here the possibilities for $H$ are listed in \cite[Table 9]{Craven} (for $\e=+$) and \cite[Table 10]{Craven} (for $\e=-$).

\begin{prop}\label{p:e6_2}
Suppose $T = E_6^{\e}(q)$, $|\O|$ is even and $H \in \mathcal{C}$. Set $e = (3,q-\e)$. Then $T$ is $2$-elusive unless one of the following holds:

\vspace{1mm}

\begin{itemize}\addtolength{\itemsep}{0.2\baselineskip}
\item[{\rm (i)}] $H_0 = {\rm L}_3^{\e}(q^3).3$, $({}^3D_4(q) \times (q^2+\e q+1)/e).3$ or $G_2(q)$.
\item[{\rm (ii)}] $H_0 = (q^2+\e q+1)^3/e.(3^{1+2}.{\rm SL}_2(3))$.
\item[{\rm (iii)}] $H_0 = {\rm PGL}_3^{\pm}(q).2$, where $p \geqs 5$ and $q \equiv \e \imod{4}$.
\item[{\rm (iv)}] $H_0 = 3^{3+3}{:}{\rm SL}_3(3)$, where $q=p \geqs 5$ and $q \equiv \e \imod{3}$.
\end{itemize}
\end{prop}

\begin{proof}
By Propositions \ref{p:parab}, \ref{p:subfield}, and \ref{p:twisted}, we know that $T$ is $2$-elusive if $H$ is a parabolic subgroup, or a subfield subgroup, or if $\e = +$ and $H_0 = {}^2E_6(q_0)$ with $q=q_0^2$. So to complete the proof we need to work through the remaining cases arising in \cite[Tables 9 and 10]{Craven}. Set $d = (2,q-1)$.

\vs

\noindent \emph{Case 1. $H$ is a maximal rank subgroup.}

\vs

First assume $H = N_G(\bar{H}_{\s})$, where $\bar{H}$ is a maximal rank subgroup. By \cite[Tables 5.1, 5.2]{LSS}, the possibilities for $\bar{H}$ are as follows:
\begin{equation}\label{e:e6list1}
A_1A_5, \; D_5T_1, \; A_2^3.{\rm Sym}_3, \; D_4T_2.{\rm Sym}_3, \; T_6.W,
\end{equation}
where in the latter case, $T_6$ is a maximal torus of $\bar{G}$ and $W = {\rm PGSp}_4(3)$ is the Weyl group. As in the proof of Proposition \ref{p:f4_2}, we are free to assume that $H = N_G(\left( \bar{H}^g \right)_{\s})$, where $\bar{H}$ is the normalizer of a standard subsystem subgroup 
\[
\bar{H}^\circ = \langle \bar{T}, U_{\alpha} \,:\, \alpha \in \Phi' \rangle,
\]
and $\bar{H}^g$ is a $\s$-invariant conjugate of $\bar{H}$. Here $\bar{T}$ is a maximal torus of $\bar{G}$, as in the setup of Section \ref{ss:setup}, and $\Phi'$ is a $\s$-invariant root subsystem of $\Phi$. As in the proof of Proposition \ref{p:f4_2}, we use \cite[Tables 5.1, 5.2]{LSS} to read off the precise structure of $H$ (also see \cite[Tables 9 and 10]{Craven}) and we refer the reader to 
\cite[p.300]{LSS} for further details. We now consider each possibility in \eqref{e:e6list1}.

\vs

\noindent \emph{Case 1.1. $\bar{H}^{\circ}$ is not a maximal torus of $\bar{G}$.}

\vs

Suppose $\bar{H} = A_1A_5$, in which case $H_0 = d.({\rm L}_2(q) \times {\rm L}_6^{\e}(q)).de$. For $q$ even, we deduce that $T$ is $2$-elusive by inspecting \cite[Section 4.8]{Law09}. And for $q$ odd, we first note that $H_0$ is the centralizer in $T$ of an involution of type $A_1A_5$ and by inspecting \cite[Table 13]{BT} we see that $H_0$ also contains involutions of type $D_5T_1$. Therefore, $T$ is $2$-elusive for all $q$.

In the case $\bar{H} = D_5T_1$, the subsystem $\Phi'$ has base $\{-\alpha_0, \alpha_2, \alpha_4, \alpha_3, \alpha_5\}$, and we have $H_0 = \bar{H}_{\s} \cap T$. If $q$ is odd, then $H_0$ is the centralizer in $T$ of an involution of type $D_5T_1$, and by inspecting \cite[Table 13]{BT} we deduce that $T$ is $2$-elusive. Now assume $q$ is even, so $H_0 = \O_{10}^{\e}(q) \times (q-\e)/e$, and $H_0$ contains the involutions 
\[
x_{\alpha_2}(1), \; x_{\alpha_3}(1) x_{\alpha_5}(1),\; x_{\alpha_2}(1) x_{\alpha_3}(1) x_{\alpha_5}(1).
\]
Using {\sc Magma}, as described in Section \ref{ss:comp}, we can calculate the action of these involutions on $V = V_{{\rm min}}$ and we deduce that the dimensions of the respective fixed point spaces are $21$, $17$ and $15$. By inspecting Table \ref{table:evenpinvolutions}, we conclude that these involutions are contained in the $T$-classes labelled $A_1$, $A_1^2$ and $A_1^3$, respectively, and thus $T$ is $2$-elusive. (It is also easy to see that these involutions are conjugate by a Weyl group element to the representatives listed in Table \ref{table:evenpinvolutions}.)

Next suppose $\bar{H} = A_2^3.{\rm Sym}_3$, so $\Phi'$ has base $\{\alpha_1, \alpha_3\} \cup \{\alpha_5,\alpha_6\} \cup \{-\alpha_0,\alpha_2\}$, and $H$ is of type $A_2^{\e}(q)^3$, $A_2(q^2)A_2^{-\e}(q)$ or $A_2^{\e}(q^3)$. In the latter case, we note that $H_0 = {\rm L}_3^{\e}(q^3).3$ has a unique class of involutions and thus $T$ is not $2$-elusive. We claim that $T$ is $2$-elusive in the other two cases.

Suppose $q$ is even and let $t$ be an involution in a group of type $A_2$. Then both $H_0 = A_2^{\e}(q)^3$ and $H_0 = A_2(q^2)A_2^{-\e}(q)$ contain conjugates of $(1,1,t)$, $(t,t,1)$, and $(t,t,t)$ from $\bar{H}^\circ = A_2^3$. And by inspecting \cite[Section 4.9]{Law09}, we deduce that these elements represent the three classes of involutions in $\bar{G}$, whence $T$ is $2$-elusive. 

Now assume $q$ is odd. From \cite[Table 13]{BT} we note that $H_0$ contains involutions of type $D_5T_1$. And as above, $H_0$ contains a conjugate of $(1,1,t) \in \bar{H}^{\circ}$, where $t \in A_2$ is an involution. A simple group of type $A_2$ has a unique conjugacy class of involutions, which for the factor corresponding to the simple roots $\{-\alpha_0,\alpha_2\}$ has $h_{\alpha_2}(-1)$ as a representative. This is precisely the representative of the class $A_1A_5$ presented in Table \ref{table:oddpinvolutions}, so we conclude that $T$ is $2$-elusive.

Now assume $\bar{H} = D_4T_2.{\rm Sym}_3$, in which case $\Phi'$ has base $\{\alpha_2, \alpha_4, \alpha_3, \alpha_5\}$ and $\bar{H}^\circ$ is a Levi factor. If $H$ is of type ${}^3D_4(q) \times (q^2+\e q+1)$ then $H_0$ has $1+\delta_{2,p}$ classes of involutions, so $T$ is not $2$-elusive. For the remainder, we may assume $H_0 = J.{\rm Sym}_3$ with 
\[
J = d^2.({\rm P\O}_{8}^{+}(q) \times ((q-\e)/d)^2/e).d^2.
\]
Let $w \in N_{\bar{H}}( \bar{T} )$ be an element that corresponds to the longest element of the Weyl group of $E_6$, explicitly we choose $w = w_{\alpha_{36}} w_{\alpha_4} w_{\alpha_{15}} w_{\alpha_{23}}$, where $\alpha_i$ is the $i$-th root respect to the specific ordering of the roots of $\bar{G}$ used by {\sc Magma} (see Remark \ref{r:ordering}).

Now under the bijection of Lemma \ref{l:sigmaclasses}, the $\s$-invariant conjugate $\bar{H}^g$ corresponds to the image of $n \in \bar{H}$ in $H^1(\s, \bar{H}/\bar{H}^\circ)$, where $n = 1$ if $\e = +$ and $n = w$ if $\e = -$. Then in view of Lemma \ref{l:tildeH0}, it suffices to consider involutions in $\widetilde{H_0} = N_{\bar{G}}(\bar{H}_{n \s}) \cap \left(\bar{G}_{n\s}\right)'$. Moreover, every involution in $\bar{H}_{n\s}$ is contained in $\left(\bar{G}_{n\s}\right)'$, so we only need to find involutions in $\bar{H}_{n\s}$.

If $q$ is odd, then we observe that the torus $\bar{T}_{n\s} < \bar{H}_{n\s}$ contains the class representatives in Table \ref{table:oddpinvolutions} and thus $T$ is $2$-elusive. Now assume $q$ is even and note that we may assume $w$ is contained in the subgroup $W_0 = \langle w_{\alpha} \, :\,  \alpha \in \Phi \rangle$. We have $W_0 \cong W$ (see Theorem \ref{t:titsweyl}), so the elements $w_{\alpha_2}$, $w_{\alpha_3} w_{\alpha_5}$ and $w_{\alpha_1} w_{\alpha_2} w_{\alpha_6}$ are fixed by $w$ and $\s$ since the corresponding elements in $W$ are centralized by the longest element. Using {\sc Magma} to compute the respective Jordan normal forms on $\mathcal{L}(\bar{G})$ (see Section \ref{ss:comp}) and inspecting Table \ref{table:evenpinvolutions}, we deduce that $w_{\alpha_2}$, $w_{\alpha_3} w_{\alpha_5}$ and $w_{\alpha_1} w_{\alpha_2} w_{\alpha_6}$ are involutions of type $A_1$, $A_1^2$ and $A_1^3$, respectively. Since each of these involutions is contained in $\bar{H}_{n\s}$, we conclude that $T$ is also $2$-elusive when $q$ is even.

\vs

\noindent \emph{Case 1.2. $\bar{H}^{\circ}$ is a maximal torus of $\bar{G}$.}

\vs

To complete the analysis of maximal rank subgroups, we may assume $\bar{H} = T_6.W$ is the normalizer of a maximal torus. As noted in \cite[Table 5.2]{LSS} (also see \cite[Tables 9 and 10]{Craven}), there are two separate cases to consider. First assume $H_0 = ((q-\e)^6/e).W$. We claim that $T$ is $2$-elusive. As in the previous paragraph, to justify the claim it suffices to show that $\bar{H}_{n\s}$ meets every $\bar{G}$-class of involutions, where $n = 1$ if $\varepsilon = +$, and $n$ corresponds to the longest element of $W$ if $\varepsilon = -$. As before, if $q$ is odd, then $\bar{T}_{n\s} < \bar{H}_{n\s}$ contains the class representatives listed in Table \ref{table:oddpinvolutions} and the claim follows. Similarly, if $q$ is even then $\bar{H}_{n\s}$ contains the involutions $w_{\alpha_2}$, $w_{\alpha_3} w_{\alpha_5}$ and $w_{\alpha_1} w_{\alpha_2} w_{\alpha_6}$. As noted in the previous paragraph, these elements represent the three classes of involutions in $\bar{G}$.

Finally, suppose
\[
H_0 = (q^2+\e q+1)^3/e.(3^{1+2}.{\rm SL}_2(3)).
\]
Since $\operatorname{SL}_2(3)$ has a unique conjugacy class of involutions, Lemma \ref{l:conj} implies that $H_0$ also has a unique class of involutions and thus $T$ is not $2$-elusive. 

\vs

\noindent \emph{Case 2. The remaining cases.}

\vs

By inspecting \cite[Tables 9 and 10]{Craven}, in order to complete the proof we may assume $H_0$ is one of the following:

\vspace{1mm}

\begin{itemize}\addtolength{\itemsep}{0.2\baselineskip}
\item[(a)] $H_0 = F_4(q)$;
\item[(b)] $H_0 = {\rm PGSp}_8(q)$, $p \ne 2$; 
\item[(c)] $H_0 = {\rm L}_3^{\e}(q) \times G_2(q)$;
\item[(d)] $H_0 = G_2(q)$;
\item[(e)] $H_0 = {\rm PGL}_3^{\delta}(q).2$, $p \geqs 5$;
\item[(f)] $H_0 = 3^{3+3}{:}{\rm SL}_3(3)$, $p \geqs 5$.
\end{itemize}

\vspace{1mm}

In case (f), Lemma \ref{l:conj} implies that $H_0$ has a unique class of involutions and thus $T$ is not $2$-elusive. The same conclusion holds in (d) since $H_0 = G_2(q)$  has $1+\delta_{2,p}$ classes of involutions. In each of the remaining cases we have $H_0 = N_T(\bar{H}_{\s})$, where $\bar{H}$ is a $\s$-stable subgroup of $\bar{G}$ of rank at most $4$. 

First consider case (a). If $q$ is even then the information in \cite[Table A]{Law} indicates that $T$ is $2$-elusive, so let us assume $q$ is odd and let $V$ be the adjoint module for $\bar{G}$. Then $V \downarrow F_4 = \mathcal{L}(F_4)/V_{\operatorname{min}}(F_4)$ (see \cite[Table 12.3]{Thomas}), so from Table \ref{table:oddpinvolutions} we deduce that $\dim C_V(t) = 38$ for the involutions $t \in H_0$ of type $A_1C_3$, whereas $\dim C_V(t) = 46$ for those of type $B_4$. Therefore, $T$ is $2$-elusive.

Next assume (b) holds, so $\bar{H}$ is of type $C_4$ and $q$ is odd. Let $\{\omega_1, \ldots, \omega_4\}$ be a set of fundamental dominant weights for $\bar{H}$. By lifting $\bar{H}$ to the simply connected cover of $\bar{G}$ and inspecting \cite[Table 12.3]{Thomas}, we note that $V_{\operatorname{min}} \downarrow C_4 = V_{C_4}(\omega_2)$. Then by a {\sc Magma} calculation over $\Q$ (see Lemma \ref{l:rationalcompute}), we see that $\bar{H}$ contains involutions with fixed point spaces of dimensions $11$ and $15$ on $V_{{\rm min}}$. For example, we can take the involutions $h_{\alpha}'(-1)$ and $h_{\beta}'(-1)$, where $\alpha$ and $\beta$ are short and long roots, respectively, in the root system of $C_4$. It follows that $T$ is $2$-elusive.

Now assume we are in case (c), so $\bar{H} = A_2G_2$ and we write $\{\omega_1,\omega_2\}$ and $\{\delta_1,\delta_2\}$ for the fundamental dominant weights of the two simple factors. By appealing to \cite[Section 5.12]{Law09}, we deduce that $T$ is $2$-elusive when $q$ is even. Now assume $q$ is odd. Here \cite[Table 18]{BT} indicates that $H_0$ contains involutions of type $T_1D_5$. If $V = \mathcal{L}(\bar{G})$ is the adjoint module, then \cite[Table 12.3]{Thomas} gives
\[
V \downarrow A_2G_2 = \mathcal{L}(A_2G_2)/(V_{A_2}(\omega_1+\omega_2) \otimes V_{G_2}(\delta_1)).
\]
And with the aid of {\sc Magma} (see Section \ref{ss:comp}) we calculate that an involution $x \in H_0$ of the form $(1,s)$ has Jordan form $(-I_{40},I_{38})$ on $V$, which places $x$ in the class $A_1A_5$. In particular, $T$ is $2$-elusive.

Finally, let us turn to case (e). Here $\bar{H} = A_2.2$ with $p \geqs 5$ and we will show that $T$ is not $2$-elusive by proving that every involution in $\bar{H}$ is of type $A_1A_5$. 

To do this, let $V = \mathcal{L}(\bar{G})$ be the adjoint module and observe that
\[
V \downarrow A_2 = \mathcal{L}(A_2)/V_{A_2}(4\omega_1+\omega_2)/V_{A_2}(\omega_1+4\omega_2)
\]
as in \cite[Table 12.3]{Thomas}, where $\{\omega_1,\omega_2\}$ are fundamental dominant weights for $\bar{H}^{\circ} = A_2$. Now the composition factors of $V \downarrow A_2$ are irreducible Weyl modules by \cite[6.6]{Lubeck}, so it follows from \cite[Lemma II.2.14]{Jantzen} that $V \downarrow A_2$ is completely reducible. Therefore
\[
V \downarrow A_2 = \mathcal{L}(A_2) \oplus V_{A_2}(4\omega_1+\omega_2) \oplus V_{A_2}(\omega_1+4\omega_2).
\]
If $x \in \bar{H} \setminus \bar{H}^\circ$ is an involution, then $x$ acts on $\bar{H}^\circ = A_2$ as a graph automorphism (see \cite[Claim, p.314]{Testerman1989}). Thus $x$ acts as $(-I_5,I_3)$ on $\mathcal{L}(A_2)$, and it interchanges the two $35$-dimensional summands in the above decomposition, which means that $x$ acts as $(-I_{40},I_{38})$ on $V$.

The connected component $\bar{H}^\circ$ has a unique conjugacy class of involutions, represented by $x = h_{\beta}'(-1)$, where $\beta$ is a root of $A_2$. Using {\sc Magma}, we calculate that $x$ acts as $(-I_4,I_4)$ on $\mathcal{L}(A_2)$ and as $(-I_{18},I_{17})$ on the other two summands. It follows that $\dim C_V(t) = 38$ for every involution $t \in \bar{H}$, whence every involution in $\bar{H}$ is of type $A_1A_5$ and we conclude that $T$ is not $2$-elusive in case (e). This completes the proof of the proposition.
\end{proof}

\section{The groups with socle $E_7(q)$}\label{s:e7}

The main goal of this section is to prove Theorem \ref{t:main} for the groups with socle $T = E_7(q)$. As before, we will divide our analysis according to whether or not $H$ is in $\mathcal{C}$ or $\mathcal{S}$. Recall that $T$ has $3+2\delta_{2,p}$ classes of involutions (see Tables \ref{table:evenpinvolutions} and \ref{table:oddpinvolutions}).

We begin by considering the groups where $H$ is an almost simple subgroup contained in the collection $\mathcal{S}$. At the time of writing, the subgroups in $\mathcal{S}$ have not been fully determined, but Craven's recent work in \cite{Craven2} severely restricts the possibilities. More precisely, he proves that $H$ is either one of the almost simple subgroups recorded in \cite[Table 1.1]{Craven2}, all of which are known to be maximal, or $H$ is a putative maximal subgroup with socle ${\rm L}_2(r)$ for $r \in \{7,8,9,13\}$ arising in part (iii) of \cite[Theorem 1.1]{Craven2} (with the restrictions on $q$ listed in \cite[Table 1.2]{Craven2}). By considering each  possibility in turn, we will show that $T$ always contains a derangement of order $2$ (see Proposition \ref{p:e7_as}).

\begin{rem}
Typically, we find that $T$ is $2$-elusive if $H$ is of the form $N_{G}(\bar{H}_{\s})$ and $\bar{H}$ is a maximal rank subgroup. And as in previous cases, we will often verify this by identifying explicit representatives in $H_0$ for each $T$-class of involutions. Now if $p = 2$, then $T = (\bar{G}_{\s})' = \bar{G}_{\s}$ and so it suffices to find suitable representatives in $\bar{H}_{\s} \leqs H_0$. However, if $p$ is odd then $[\bar{G}_{\s} : T] = 2$ and so it is not sufficient to find representatives in $\bar{H}_{\s}$. In the latter case, in order to verify that a given element $x \in \bar{H}$ is contained in $H_0$, we need to check that $x$ lifts to an element in the simply connected cover $\bar{G}_{\rm sc}$ that is fixed by $\s$. We refer the reader to Remark \ref{r:e7_inv} for further details.
\end{rem}

\begin{prop}\label{p:e7_as}
Suppose $T = E_7(q)$, $|\O|$ is even and $H \in \mathcal{S}$. Then $T$ is not $2$-elusive.
\end{prop}

\begin{proof}
Let $S$ denote the socle of $H$ and assume for now that $H$ is one of the  subgroups recorded in \cite[Table 1.1]{Craven2}. Just by considering the number of classes of involutions in $H_0$, we see that $T$ is $2$-elusive only if $H_0 = {\rm M}_{12}.2$ and $q=p=5$. Here both $T$ and $H_0$ have three classes of involutions and by considering the composition factors of the adjoint module $V = \mathcal{L}(\bar{G})$ restricted to $S$ (see \cite[Table 6.147]{Litt}) we deduce that every involution in $S$ has trace $5$ on $V$. This implies that every involution in $S$ is of type $A_1D_6$ (see Table \ref{table:oddpinvolutions}). And since there is a unique class of involutions in $H_0 \setminus S$, we conclude that $T$ is not $2$-elusive.

To complete the proof, we may assume $S = {\rm L}_2(r)$ with $r \in \{7,8,9,13\}$, as in part (iii) of \cite[Theorem 1.1]{Craven2}. Recall that it remains an open problem to determine whether or not $G$ has a maximal subgroup of this form. In any case, we will prove that if such a subgroup $H$ does exist, then the corresponding action of $T$ on $G/H$ is not $2$-elusive. As before, by considering the number of involution classes in $H_0$, we deduce that $T$ is $2$-elusive only if $H_0 = {\rm L}_2(9).2 \cong {\rm Sym}_6$ or ${\rm L}_2(9).2^2$. Here $p \geqs 5$ and $H_0$ has three classes of involutions. As before, let $V = \mathcal{L}(\bar{G})$ be the adjoint module.

Suppose $p \geqs 7$. Then \cite[Section 6.3]{Craven2} gives
\[
V \downarrow S = 10^2 \oplus 9^3 \oplus 8_1^4 \oplus 8_2^3 \oplus 5_1^3 \oplus 5_2^3,
\]
where the irreducible $8$-dimensional modules $8_1$ and $8_2$ are interchanged by an involution in $H_0 \setminus S$ (corresponding to a transposition in ${\rm Sym}_6$). But the multiplicities of $8_1$ and $8_2$ as composition factors of $V \downarrow S$ are unequal, so $H_0$ does not act on $V$ and we have a contradiction. We conclude that if $p \geqs 7$ and $H$ is a maximal subgroup with socle ${\rm L}_2(9)$, then $H_0$ has at most two conjugacy classes of involutions and so $T$ is not $2$-elusive.

Finally, let us assume $p=5$. First suppose $H_0 \cong {\rm Sym}_6$. By using {\sc Magma} to compute the feasible characters of $H_0$ on $V$ (see Section \ref{ss:feasible}), we deduce that every involution in $H_0$ has trace $-8$ on $V$. This is in fact true for all feasible characters of $H_0$, not just the ones with property (\textbf{P}). In any case, every involution in $H_0$ is of type $A_1D_6$ and so $T$ is not $2$-elusive. This also proves that $T$ is not $2$-elusive in the case $H_0 = {\rm L}_2(9).2^2$, since $H_0$ has three conjugacy classes of involutions, and two of them are contained in ${\rm L}_2(9).2 \cong {\rm Sym}_6$.
\end{proof}

For the remainder of this section we may assume $H \in \mathcal{C}$ is one of the maximal subgroups recorded in \cite[Table 4.1]{Craven2}. Our main result is the following.

\begin{prop}\label{p:e7_ns}
Suppose $T = E_7(q)$, $|\O|$ is even and $H \in \mathcal{C}$. Then $T$ is $2$-elusive unless $H_0$ is one of the following:
\[
P_m \, (\mbox{$m \in \{2,5,7\}$ {\rm and} $q \equiv 3 \imod{4}$}),
\]
\[
({\rm L}_2(q^3) \times {}^3D_4(q)).3, \, {\rm L}_2(q^7).7, \, {\rm L}_2(q) \times {\rm PGL}_2(q) \, (p \geqs 5), 
\]
\[
{\rm PGL}_3^{\pm}(q).2 \, (p \geqs 5), \, {}^3D_4(q).3 \, (p \geqs 3), \, {\rm L}_2(q) \, (\mbox{{\rm two classes;} $p \geqs 17,19$})
\]
\end{prop}

We divide the proof of Proposition \ref{p:e7_ns} into a sequence of lemmas, recalling that we have already proven the result when $H$ is a parabolic or a subfield subgroup (see Propositions \ref{p:parab} and \ref{p:subfield}). Set $d = (2,q-1)$.

\begin{rem}\label{r:e7note}
Here we take the opportunity to correct an error in the original arXiv version of \cite{Craven2}. This concerns the precise structure of a $2$-local maximal subgroup, which is presented as $(2^2 \times {\rm P\Omega}_8^{+}(q).2^2).3$ in \cite[Table 4.1]{Craven2}. Here $q$ is odd and 
the correct structure is 
\[
H_0 = \begin{cases} 
(2^2 \times {\rm P\Omega}_8^{+}(q).2^2).{\rm Sym}_3 & \mbox{if $q \equiv \pm 1 \imod{8}$} \\ 
(2^2 \times {\rm P\Omega}_8^{+}(q).2^2).3 & \mbox{if $q \equiv \pm 3 \imod{8}$.}
\end{cases}
\]
This error originates from an inaccuracy in one of the entries in \cite[Table 1]{AnDietrich} and we refer the reader to \cite{Kor} for more details.
\end{rem}

We begin by handling the remaining maximal rank subgroups of $G$.

\begin{lem}\label{l:e7_mr}
The conclusion to Proposition \ref{p:e7_ns} holds if $H = N_G(\bar{H}_{\s})$ and $\bar{H}$ is a reductive $\s$-stable maximal rank subgroup of $\bar{G}$.
\end{lem}

\begin{proof}
By inspecting \cite[Tables 5.1, 5.2]{LSS}, we see that $\bar{H}$ is one of the following types:
\begin{equation}\label{e:e7list1}
A_1D_6, \, E_6T_1.2, \, A_7.2, \, A_2A_5.2, \,  A_1^3D_4.{\rm Sym}_3, \, A_1^7.{\rm L}_3(2), \, T_7.W.
\end{equation}
Here $T_7$ is a maximal torus of $\bar{G}$ and $W = 2 \times {\rm Sp}_6(2)$ is the Weyl group of $\bar{G}$. As in the proof of Propositions \ref{p:f4_2} and \ref{p:e6_2}, we are free to assume that $H = N_G(\left( \bar{H}^g \right)_{\s})$, where $\bar{H}$ is the normalizer of a standard subsystem subgroup 
\[
\bar{H}^\circ = \langle \bar{T}, U_{\alpha} \,:\, \alpha \in \Phi' \rangle
\]
and $\bar{H}^g$ is a $\s$-invariant conjugate of $\bar{H}$. Here $\bar{T}$ is a maximal torus of $\bar{G}$ as in the setup of Section \ref{ss:setup}, and $\Phi'$ is a root subsystem of $\Phi$. As in the proof of Propositions \ref{p:f4_2} and \ref{p:e6_2}, we refer to \cite[Tables 5.1, 5.2]{LSS} for the precise structure of $H$ in each case (also see \cite[Table 4.1]{Craven2}).

Note that the Weyl group of $\bar{H}^\circ$ can be identified with the subgroup $\langle s_{\alpha} : \alpha \in \Phi' \rangle$ of $W$. Then $\bar{H}$ is generated by $\bar{H}^\circ$, together with the elements of $N_{\bar{G}}(\bar{T})$ that normalize the Weyl group of $\bar{H}^\circ$. In particular, $\bar{H}$ contains an element $w \in N_{\bar{G}}(\bar{T})$ which maps to the central involution of $W$. Explicitly, we have
\begin{equation}\label{e:wdef}
w = w_{\alpha_1} w_{\alpha_2} w_{\alpha_5} w_{\alpha_7} w_{\alpha_{37}} w_{\alpha_{55}} w_{\alpha_{61}} \in \bar{H},
\end{equation}
where $\alpha_i$ denotes the $i$-th positive root in $\Phi$ with respect to the specific ordering of the roots of $\bar{G}$ used by {\sc Magma} (see Remark \ref{r:ordering}). Then $w$ acts as $\alpha \mapsto -\alpha$ on $\Phi$, and a computation shows that 
\[
w^2 = h_{\alpha_2}(-1) h_{\alpha_5}(-1) h_{\alpha_7}(-1) = 1,
\]
so $w$ is an involution in $\bar{G}$. We will write $\a_0$ for the longest root of $\Phi$.

Let $n \in N_{\bar{G}}(\bar{T})$ be an element of $\bar{H}$ such that under the bijection of Lemma \ref{l:sigmaclasses}, the $\s$-invariant conjugate $\bar{H}^g$ corresponds to the image of $n$ in $H^1(\s, \bar{H}/\bar{H}^\circ)$. Then in view of Lemma \ref{l:tildeH0}, for determining $2$-elusivity it suffices to consider involutions in $\widetilde{H_0} = N_{\bar{G}}(\bar{H}_{n \s}) \cap \left(\bar{G}_{n\s}\right)'$. 

We now make an observation which will be useful in the cases where $n \in \{1,w\}$. A computation with {\sc Magma} (which can be done over $\mathbb{Q}$, as described in Section \ref{ss:comp}) shows that the lift of $w$ to $\bar{G}_{\rm sc}$ centralizes the lifts of $h_{\alpha}(-1)$ and $w_{\alpha}$ for all $\alpha \in \Phi$. Moreover $h_{\alpha}(-1)$ and $w_{\alpha}$ are both fixed by $\s$, which leads to the following conclusion:
\begin{equation}\label{e:e7centralize}
h_{\alpha}(-1), \; w_{\alpha} \in \left(\bar{G}_{\s}\right)' \cap \left(\bar{G}_{w\s}\right)'\ \text{ for all } \alpha \in \Phi.
\end{equation}

Let us now consider the possibilities for $\bar{H}$, which we listed above in \eqref{e:e7list1}. First assume $\bar{H} = T_7.W$ is the normalizer of a maximal torus of $\bar{G}$. We claim that $T$ is $2$-elusive. Here $H_0 = S.W = N_T(S)$ with $S = (q -\e)^7/d$ (see \cite[Table 4.1]{Craven2}), and we can take $n = 1$ if $\e = +$ and $n = w$ if $\e = -$. We begin by assuming $q$ is odd. In view of~\eqref{e:e7centralize}, we have $h_{\alpha}(-1), w_{\alpha} \in H_0$ for all $\alpha \in \Phi$. Hence $\widetilde{H_0}$ contains the following involutions:
\[
h_{\alpha_1}(-1), \; w_{\alpha_2} w_{\alpha_5} w_{\alpha_7},\; h_{\alpha_1}(-1) w_{\alpha_2} w_{\alpha_5} w_{\alpha_7},
\]
which are of type $A_1D_6$, $E_6T_1$ and $A_7$, respectively (see Table \ref{table:oddpinvolutions}). We conclude that $T$ is $2$-elusive.

Now assume $q$ is even. Here Theorem \ref{t:titsweyl} implies that $H_0 = S{:}W$ is a split extension and we may identify $W$ with the subgroup of $E_7(2)$ generated by the set $\{ w_{\alpha} \, : \, \alpha \in \Phi \}$. Then using {\sc Magma} (see Section \ref{ss:comp}), we can calculate the Jordan form of each involution class representative in $W$ on the adjoint module $\mathcal{L}(\bar{G})$ and by inspecting Table \ref{table:evenpinvolutions} we conclude that $T$ is $2$-elusive.

We now turn to the remaining possibilities for $\bar{H}$, dividing the analysis according to the parity of $q$. To begin with, we will assume $q$ is odd.

\vs

\noindent \emph{Case 1.1. $\bar{H} = A_1D_6$, $q$ odd.}
 
\vs

Here $\bar{H} = \bar{H}^\circ$ corresponds to a subsystem $\Phi'$ with base $\{-\alpha_0, \alpha_2, \alpha_3, \alpha_4, \alpha_5, \alpha_6, \alpha_7\}$. Then $H_0 = \bar{H} \cap T$ contains the elements $h_{\alpha_3}(-1)$, $w_{\alpha_2} w_{\alpha_5} w_{\alpha_7}$ and $h_{\alpha_3}(-1)w_{\alpha_2} w_{\alpha_5} w_{\alpha_7}$, and it is easy to check that they are conjugate to the three class representatives listed in Table \ref{table:oddpinvolutions}. (This can also be verified by computing the dimension of each fixed point space on the adjoint module $\mathcal{L}(\bar{G})$, as discussed in Section \ref{ss:comp}.) We conclude that $T$ is $2$-elusive.

\vs

\noindent \emph{Case 1.2. $\bar{H} = A_7.2$, $q$ odd.}

\vs

In this case, $\bar{H}^\circ = A_7$ corresponds to a subsystem $\Phi'$ with base 
\[
\{\beta_1, \ldots, \beta_7\} = \{-\alpha_0, \alpha_1, \alpha_3, \alpha_4, \alpha_5, \alpha_6, \alpha_7\}.
\]
In addition, the element $w$ defined in \eqref{e:wdef} is contained in $\bar{H} \setminus \bar{H}^\circ$ (since the Weyl group of $A_7$ does not contain $-1$), so we have a split extension $\bar{H} = \bar{H}^\circ{:}\langle w \rangle$. Thus in this case we can take $n \in \{1,w\}$ and it suffices to consider involutions in $\widetilde{H_0} = \bar{H} \cap \left(\bar{G}_{n\s}\right)'$.

Define $t = h_{\alpha_4}(-1) \in \bar{H}$, which is an involution of type $A_1D_6$. Next we choose an involution $g \in \bar{H}^{\circ}$ corresponding to the longest element in the Weyl group of $\bar{H}^\circ$. Specifically, we take 
\[
g = w_{\beta_1} (w_{\beta_2} w_{\beta_1}) \cdots (w_{\beta_7} w_{\beta_6} \cdots w_{\beta_1}),
\]
which one can verify is an involution by direct calculation. 

In view of~\eqref{e:e7centralize}, we have $t, w, g \in \left(\bar{G}_{n\s}\right)'$, so we deduce that $t, w, g \in \widetilde{H_0}$. Then by computing the actions of $t$, $tgw$ and $w$ on $\mathcal{L}(\bar{G})$, we see that they are involutions with fixed point spaces of dimensions $69$, $79$ and $63$, respectively. So by inspecting Table \ref{table:oddpinvolutions}, we deduce that $T$ is $2$-elusive.

\vs

\noindent \emph{Case 1.3. $\bar{H} = A_2A_5.2$, $q$ odd.}

\vs

Here $\bar{H}^\circ = A_2A_5$ corresponds to a subsystem of $\Phi$ with base $\{-\alpha_0, \alpha_1, \alpha_2, \alpha_4, \alpha_5, \alpha_6, \alpha_7\}$. As in the previous case, we have $\bar{H} = \bar{H}^\circ{:}\langle w \rangle$, so we can take $n \in \{1,w\}$ and it suffices to consider involutions in $\widetilde{H_0} = \bar{H} \cap \left(\bar{G}_{n\s}\right)'$. 

Now the elements $t = h_{\alpha_1}(-1)$ and $g = w_{\alpha_2} w_{\alpha_5} w_{\alpha_7}$ are contained in $\bar{H}$, and moreover $t,g \in \widetilde{H_0}$ by~\eqref{e:e7centralize}. Then $t$, $g$ and $tg$ coincide precisely with the representatives given in Table \ref{table:oddpinvolutions}, so $T$ is $2$-elusive.

\vs

\noindent \emph{Case 1.4. $\bar{H} = A_1^3D_4.{\rm Sym}_3$, $q$ odd.}

\vs

First observe that 
\[
\bar{H}^\circ = A_1^3 D_4  = A_1D_2D_4 < A_1D_6,
\]
where $A_1D_6$ is the maximal rank subgroup we treated in Case 1.1. In addition, 
$\bar{H}^\circ$ corresponds to a subsystem of $\Phi$ with base $\{-\alpha_0, \alpha_7, -\beta, \alpha_2, \alpha_3, \alpha_4, \alpha_5\}$, where $\beta = \alpha_{49}$ is the longest root of the $D_6$ factor in $A_1D_6$.

There are two possibilities for $H_0$, as recorded in \cite[Table 4.1]{Craven2}. The first is 
\[
H_0 = 2^2.({\rm L}_2(q)^3 \times {\rm P\Omega}_8^{+}(q)).2^2.{\rm Sym}_3.
\]
In this case we can take $n = 1$, so $\widetilde{H_0} = H_0 = \bar{H} \cap T$. Then $H_0$ contains the elements $h_{\alpha_3}(-1)$, $w_{\alpha_2} w_{\alpha_5} w_{\alpha_7}$ and $h_{\alpha_3}(-1)w_{\alpha_2} w_{\alpha_5} w_{\alpha_7}$ from Case 1.1 and thus $T$ is $2$-elusive.

The other possibility is $H_0 = J.3$, where $J = {\rm L}_2(q^3) \times {}^3D_4(q)$. Clearly, every involution in $H_0$ is contained in $J$ and we note that there are three such classes. As in the other cases, it suffices to consider involutions in $\widetilde{H_0} = \bar{H} \cap \left(\bar{G}_{n\s}\right)'$.  Here we can choose $n \in N_{\bar{G}}(\bar{T})$ such that $n$ transitively permutes the three roots $\{-\alpha_0, \alpha_7, -\beta \}$ of the $A_1^3$ factor and induces a triality automorphism on the $D_4$ factor, transitively permuting the roots $\{\alpha_2, \alpha_3, \alpha_5\}$.

Then the unique conjugacy classes of involutions in the ${\rm L}_2(q^3)$ and ${}^3D_4(q)$ factors are represented by the elements $t,t' \in \widetilde{H_0}$, respectively, where
\[
t = h_{\alpha_7}(-1)h_{\alpha_0}(-1)h_{\beta}(-1), \;\; t' = h_{\alpha_2}(-1)h_{\alpha_3}(-1)h_{\alpha_5}(-1),
\]
and the third class of involutions in $\widetilde{H_0}$ is represented by their product $tt' = t't$. But this means that every involution in $\widetilde{H_0}$ lifts to an involution in the simply connected cover $\bar{G}_{{\rm sc}}$, and hence every involution in $\widetilde{H_0}$ is of type $A_1D_6$. So in this case, we conclude that $T$ is not $2$-elusive.

\vs

\noindent \emph{Case 1.5. $\bar{H} = A_1^7.{\rm L}_3(2)$, $q$ odd.}

\vs

Here the connected component $\bar{H}^\circ = A_1^7$ is embedded in $\bar{G}$ via
\[
\bar{H}^\circ = A_1^7 = A_1(A_1^2)^3 < A_1(A_1^2D_4) < A_1D_6 < \bar{G}.
\]
In particular, $\bar{H}^{\circ}$ corresponds to a subsystem with base $\{\alpha_0, \alpha_2, \alpha_3, \alpha_5, \gamma, \beta, \alpha_7\}$, where $\gamma = \alpha_{28}$ is the longest root in the $D_4$ factor of $A_1^2 D_4 < D_6$, and $\beta = \alpha_{49}$ is the longest root in the $D_6$ factor of $A_1D_6$. As noted in \cite[Table 4.1]{Craven2}, there are two possibilities for $H_0$. 

If $H_0 = {\rm L}_2(q^7).7$, then $H_0$ has a unique class of involutions and thus $T$ is not $2$-elusive. The other possibility is $H_0 = 2^3.{\rm L}_2(q)^7.2^3.{\rm L}_3(2)$, in which case $H_0 = \bar{H} \cap T$ and it follows that $H_0$ contains the involutions listed in Case 1.1. Therefore, $T$ is $2$-elusive.

\vs

\noindent \emph{Case 1.6. $\bar{H} = E_6T_1.2$, $q$ odd.}

\vs

To complete the proof of the lemma for $q$ odd, we may assume $\bar{H} = E_6T_1.2$, so $\bar{H}^\circ$ is a standard Levi factor corresponding to a root system with base $\{\alpha_1, \ldots, \alpha_6\}$. As in Case 1.2, we have $\bar{H} = \bar{H}^\circ{:}\langle w \rangle$, so we can take $n \in \{1,w\}$ and it suffices to consider involutions in $\widetilde{H_0} = \bar{H} \cap \left(\bar{G}_{n\s}\right)'$. 

Set $t = h_{\alpha_1}(-1) \in \bar{H}^{\circ}$, which is an involution of type $A_1D_6$, and define 
\[
h = w_{\alpha_2} w_{\alpha_{28}} w_{\alpha_{38}} w_{\alpha_{46}} \in \bar{H}^{\circ},
\] 
which corresponds to the longest element of the Weyl group of $E_6$. As in Case 1.2, we have $t, w, h \in \widetilde{H_0}$. A computation (discussed also in Remark \ref{r:e7_inv}) shows that $t$, $w$, $hw$ are involutions in $\widetilde{H_0}$ from classes $A_1D_6$, $A_7$, $E_6T_1$, respectively. Therefore $T$ is $2$-elusive.

\vs

In order to complete the proof of the lemma, we may assume $q$ is even. Note that in this case $\bar{G}_{\s} = \left( \bar{G}_{\s} \right)'$ and thus $\widetilde{H_0} = \bar{H}_{n\s}$. As above, we partition the analysis into a number of subcases.

\vs

\noindent \emph{Case 2.1. $\bar{H} = A_1D_6$ or $A_2A_5.2$, $q$ even.}

\vs

In both cases, the $\bar{G}$-class of each unipotent element in $\bar{H}^{\circ}$ has been determined by Lawther, see \cite[Sections 4.10, 4.12]{Law09}, and it follows that $\bar{H}^{\circ}$ meets every $\bar{G}$-class of involutions. Since $H_0$ meets every unipotent conjugacy class in $\bar{H}^{\circ}$, we deduce that $T$ is $2$-elusive.

\vs

\noindent \emph{Case 2.2. $\bar{H} = A_7.2$, $q$ even.}

\vs

By inspecting \cite[Section 4.1]{Law09}, we deduce that $\bar{H}^{\circ}$ contains involutions of type $A_1$, $A_1^2$ and $(A_1^3)^{(2)}$. Since $H_0 = \operatorname{PGL}_8^{\e}(q).2$ meets every unipotent conjugacy class of $\bar{H}^{\circ}$, it follows that $H_0$ contains involutions in each of these classes.

As noted in Case 1.2, we have $\bar{H} = \bar{H}^{\circ}{:}\langle w \rangle$. Thus we can take $n \in \{1,w\}$ and it suffices to consider involutions in $\widetilde{H_0} = \bar{H}_{n\s}$. We now define $g \in \bar{H}^{\circ}$ as in Case 1.2; by~\eqref{e:e7centralize} we have $g, w \in \widetilde{H_0}$. Another computation with {\sc Magma} (as described in Section \ref{ss:comp}) shows that $w$ and $gw$ are involutions with respective Jordan forms $(2^{63},1^7)$ and $(2^{53},1^{27})$ on $\mathcal{L}(\bar{G})$, so they are in the $\bar{G}$-classes labelled $A_1^4$ and $(A_1^3)^{(1)}$, respectively. Since $w, gw \in \widetilde{H_0}$, we deduce that $T$ is $2$-elusive.

\vs

\noindent \emph{Case 2.3. $\bar{H} = A_1^3D_4.{\rm Sym}_3$, $q$ even.}

\vs

As in Case 1.4, $\bar{H}^\circ$ corresponds to a subsystem with base $\{-\alpha_0, \alpha_7, -\beta, \alpha_2, \alpha_3, \alpha_4, \alpha_5\}$, where $\beta$ is the longest root in the root system of $D_6$. There are two possibilities for $H_0$.

First assume $H_0 = ({\rm L}_2(q)^3 \times \Omega_8^{+}(q)).{\rm Sym}_3$, in which case $H_0 = \bar{H}_{\s}$. Setting $x_i = x_{\a_i}$, we see that $H_0$ contains the following involutions
\begin{equation}\label{e:inv_list}
x_{7}(1),\; x_{7}(1)x_{5}(1),\; x_{7}(1)x_{5}(1)x_{3}(1),\; x_{7}(1)x_{5}(1)x_{2}(1),\;
x_{0}(1)x_{7}(1)x_{5}(1)x_{2}(1).
\end{equation}
By using {\sc Magma} to calculate the Jordan form of each of these elements on the adjoint module $\mathcal{L}(\bar{G})$, we deduce that the corresponding $\bar{G}$-classes are $A_1$, $A_1^2$, $(A_1^3)^{(2)}$, $(A_1^3)^{(1)}$ and $A_1^4$, respectively, and it follows that $T$ is $2$-elusive.

Now assume $H_0 = ({\rm L}_2(q^3) \times {}^3D_4(q)).3$. As in Case 1.4, it suffices to consider involutions in $\widetilde{H_0} = \bar{H}_{n\s}$, where $n$ fixes $\a_4$ and it acts as a $3$-cycle on the sets $\{-\alpha_0, \alpha_7, -\beta \}$ and $\{\alpha_2, \alpha_3, \alpha_5\}$. Explicitly, we can choose $$n = w_{\alpha_{20}} w_{\alpha_{48}} w_{\alpha_{47}} w_{\alpha_{30}}.$$ Then the ${\rm L}_2(q^3)$ factor has a unique class of involutions, with a representative given by $t = x_{{-\alpha_0}}(1) x_{-\beta}(1) x_{\alpha_7}(1)$. There are two classes of involutions in ${}^3D_4(q)$, with representatives given by $s = x_{\alpha_2}(1) x_{\alpha_3}(1) x_{\alpha_5}(1)$ and $s' = x_{\alpha_4}(1)$ (see Table \ref{table:3D4q}). We have $t,s,s' \in \widetilde{H_0}$, and any involution in $\widetilde{H_0}$ is conjugate to $t$, $s$, $s'$, $ts$ or $ts'$. By computing the respective Jordan forms on the adjoint module $\mathcal{L}(\bar{G})$, we find that the corresponding $\bar{G}$-classes are $(A_1^3)^{(1)}$, $(A_1^3)^{(2)}$, $A_1$, $A_1^4$ and $A_1^4$, respectively. Therefore, $\widetilde{H_0}$ does not contain any involutions of type $A_1^2$ and thus $T$ is not $2$-elusive.

\vs

\noindent \emph{Case 2.4. $\bar{H} = A_1^7.{\rm L}_3(2)$, $q$ even.}

\vs

As in Case 1.5, $\bar{H}^\circ$ corresponds to a subsystem with base $\{\alpha_0, \alpha_2, \alpha_3, \alpha_5, \gamma, \beta, \alpha_7\}$, where $\gamma$ is the longest root for the $D_4$ factor of $A_1^2 D_4 < D_6$, and $\beta$ is the longest root for the $D_6$ factor of $A_1D_6$. If $H_0 = {\rm L}_2(q)^7.{\rm L}_3(2)$, then $H_0 = \bar{H}_{\s}$ contains all the involutions in \eqref{e:inv_list} and thus $T$ is $2$-elusive by the argument in Case 2.3. On the other hand, if $H_0 = {\rm L}_2(q^7).7$ then $H_0$ has a unique class of involutions and thus $T$ is not $2$-elusive. 

\vs

\noindent \emph{Case 2.5. $\bar{H} = E_6T_1.2$, $q$ even.}

\vs

Finally, to complete the proof we may assume $q$ is even and $\bar{H} = E_6T_1.2$, which is the normalizer of a standard Levi factor $\bar{H}^\circ = E_6T_1$ with base $\{\alpha_1, \ldots, \alpha_6\}$ for the root system of $(\bar{H}^\circ)' = E_6$. As in Case 1.6, it suffices to consider involutions in $\widetilde{H_0} = \bar{H}_{n\s}$ with $n \in \{1,w\}$. It follows from~\eqref{e:e7centralize} that $\widetilde{H_0}$ contains the following involutions 
\[
w_{\alpha_2}, \; w_{\alpha_2} w_{\alpha_3}, \; w_{\alpha_2} w_{\alpha_3} w_{\alpha_5}.
\]
By computing their action on $\mathcal{L}(\bar{G})$, we see that these involutions are contained in the $\bar{G}$-classes labelled $A_1$, $A_1^2$ and $(A_1^3)^{(2)}$, respectively. 

Next define $h \in \bar{H}^\circ$ as in Case 1.6, which corresponds to the longest element of the Weyl group of $E_6$. As before, we have $h, w \in \widetilde{H_0}$ and with the aid of {\sc Magma} we can show that $w$ and $hw$ are involutions with respective Jordan forms $(2^{63},1^7)$ and $(2^{53},1^{27})$ on $\mathcal{L}(\bar{G})$. By inspecting Table \ref{table:evenpinvolutions}, we deduce that $w$ and $hw$ are in the $\bar{G}$-classes labelled $A_1^4$ and $(A_1^3)^{(1)}$, respectively, and this allows us to conclude that $T$ is $2$-elusive. 
\end{proof}

To complete the proof of Proposition \ref{p:e7_ns}, which in turn completes the proof of Theorem \ref{t:main} for the groups with socle $E_7(q)$, it just remains to consider the cases where $H = N_G(\bar{H}_{\s})$ and $\bar{H}$ is a maximal positive-dimensional closed subgroup of $\bar{G}$ that does not contain a maximal torus. Note that this includes the family of exotic $2$-local subgroups from \cite[Table 1]{CLSS}, which we treat separately in the following lemma.

\begin{lem}\label{l:e7_exotic}
Suppose $p$ is odd and $H = N_G(E)$, where $E = 2^2$ and
\[
N_{\bar{G}_{\s}}(E) = (E \times {\rm Inndiag}({\rm P\O}_8^{+}(q)).{\rm Sym}_3.
\]
Then $T$ is $2$-elusive.
\end{lem}

\begin{proof}
We begin by recalling some basic properties of $E$, as described in the proof of \cite[Lemma 2.15]{CLSS}. First, let us record the fact that the three involutions in $E$ are all of type $A_7$. In addition, $C_{\bar{G}}(E) = E \times D_4$, where the $D_4$ factor is of adjoint type. Since $N_{\bar{G}_{\s}}(E)/C_{\bar{G}_{\s}}(E) \cong \operatorname{Sym}_3$ acts transitively on the set of involutions in $E$, it follows that $E \leqs T$ and thus $H_0 = N_T(E)$ contains involutions of type $A_7$. Since the lift of the $D_4$ factor to $\bar{G}_{{\rm sc}}$ is also of adjoint type, it follows that 
\[
C_T(E) = C_{\bar{G}_{\s}}(E) = E \times {\rm Inndiag}({\rm P\O}_{8}^{+}(q)),
\]
(also see \cite[Theorem 1.1]{Kor}). We claim that $C_T(E)$ also contains involutions of type $A_1D_6$ and $E_6T_1$, which implies that $T$ is $2$-elusive. 

To see this, first write $E = \langle e, f \rangle$. Here $e \in C_{\bar{G}}(e)^\circ$ and $f \in C_{\bar{G}}(e) \setminus C_{\bar{G}}(e)^\circ$, since lifts of $e$ and $f$ to the simply connected cover $\bar{G}_{{\rm sc}}$ do not commute, as observed in the proof of \cite[Lemma 2.15]{CLSS}. Recall that $C_{\bar{G}}(e)^\circ$ is of type $A_7$. More precisely, we can identify $C_{\bar{G}}(e)^\circ$ with ${\rm SL}_8(K) / \langle \zeta^2 I_8 \rangle$, where $\zeta \in K$ is a primitive $8$th root of unity. Then the central involution $e \in C_{\bar{G}}(e)^\circ$ corresponds to the image of the scalar matrix $\zeta I_8$. Furthermore, we can assume that $f$ acts as the inverse-transpose graph automorphism on $C_{\bar{G}}(e)^\circ$.

Let $t$ be an involution in $C_{\bar{G}}(e)^\circ$ corresponding to the diagonal matrix $(-I_2, I_6)$ in ${\rm SL}_8(K)$ and note that 
\[
\mathcal{L}(\bar{G}) \downarrow A_7 = \mathcal{L}(A_7)/\Lambda^4(W),
\]
where $W$ is the natural module for $A_7$ (see \cite[Table 12.4]{Thomas}).
Here $t$ has trace $15$ on the first summand and trace $-10$ on the second, so $t$ has trace $5$ on $\mathcal{L}(\bar{G})$ and it is therefore an involution of type $A_1D_6$. Now the central involution $e$ acts as $I_{63}$ on $\mathcal{L}(A_7)$ and as $-I_{70}$ on $\Lambda^4(V)$. Therefore, the involution $et$ has trace $15 + 10 = 25$ on $\mathcal{L}(\bar{G})$ and so it belongs to the class $E_6T_1$. 

Therefore, in order to prove that $T$ is $2$-elusive, it suffices to show that the subgroup ${\rm Inndiag}({\rm P\O}_8^{+}(q))$ of $C_T(E)$ contains an involution $x$ that is conjugate to $t$. But this is clear since the $D_4$ subgroup of $C_{\bar{G}}(E)$ is the image of the orthogonal subgroup ${\rm SO}_8(K) < {\rm SL}_8(K)$. The result now follows since $x$ is in the class $A_1D_6$ and the product $ex$ is an involution of type $E_6T_1$.
\end{proof}

\begin{lem}\label{l:e7_mc}
The conclusion to Proposition \ref{p:e7_ns} holds if $H = N_G(\bar{H}_{\s})$ and $\bar{H}$ is a positive-dimensional non-maximal rank subgroup of $\bar{G}$.
\end{lem}

\begin{proof}
By inspecting \cite[Table 4.1]{Craven2}, we see that the possibilities for $\bar{H}$ are as follows: 
\begin{equation}\label{e:e7list2}
A_1F_4, \, G_2C_3, \, A_1G_2 \, (p \geqs 3), A_1^2 \, (p \geqs 5), A_2.2 \, (p \geqs 5), A_1 \, (\mbox{two classes; $p \geqs 17,19$}),
\end{equation}
together with $(2^2 \times D_4).{\rm Sym}_3$ for $p \geqs 3$.
	
First assume $q$ is even, so $\bar{H} = A_1F_4$ or $G_2C_3$. As explained in \cite[Section 5.12]{Law09}, the $\bar{G}$-class of each involution in $\bar{H}$ can be read off from \cite[Table 38]{Law09} and we quickly deduce that $T$ is $2$-elusive. For the remainder we may assume $q$ is odd. Let $V = \mathcal{L}(\bar{G})$ be the adjoint module. We will consider each possibility for $\bar{H}$ in turn (see \eqref{e:e7list2}).

Suppose $\bar{H} = A_1F_4$, in which case $H_0 = {\rm L}_2(q) \times F_4(q)$ and
\[
V \downarrow A_1F_4 = (V_{A_1}(2) \otimes V_{F_4}(\delta_4)) / (V_{A_1}(2) \otimes 0) / (0 \otimes V_{F_4}(\delta_1))
\]
by \cite[Table 12.4]{Thomas}, where $0$ is the trivial module and $\{\delta_1, \ldots, \delta_4\}$ is a set of fundamental dominant weights for the $F_4$ factor. Let $t \in H_0$ be an involution in the ${\rm L}_2(q)$ factor, and let $s,s'$ be representatives of the two classes of involutions in $F_4(q)$. A {\sc Magma} calculation (which can be done over $\Q$; see Lemma \ref{l:rationalcompute}) shows that $s$, $ts$ and $ts'$ have respective fixed point spaces of dimension $69$, $63$ and $79$ on $V$. It follows that $T$ is $2$-elusive.

Next assume $\bar{H} = G_2C_3$, so $H_0 = G_2(q) \times {\rm PSp}_6(q)$ and  
\[
V \downarrow G_2C_3 = (V_{G_2}(\omega_1) \otimes V_{C_3}(\delta_2)) / (V_{G_2}(\omega_2) \otimes 0) / (0 \otimes V_{C_3}(2\delta_1))
\]
as recorded in \cite[Table 12.4]{Thomas}, where $\{\omega_1,\omega_2\}$ and $\{\delta_1, \delta_2,\delta_3\}$ are fundamental dominant weights for the two factors of $\bar{H}^{\circ}$. Here the $G_2(q)$ factor has a unique class of involutions, which is represented by $t = h_{\alpha}'(-1)$ for some root $\alpha$ for $G_2$. Next let $s \in {\rm PSp}_6(q)$ be an involution which lifts to an element of order $4$ in ${\rm Sp}_6(q)$. Here a representative is given by $s = w_{\beta_1}'' w_{\beta_3}''$, where $\{\beta_1, \beta_2, \beta_3\}$ are the simple roots of $C_3$ and $w_{\beta}''$ is a standard Chevalley generator of $C_3$, as defined in Section \ref{ss:setup}. 

Working with {\sc Magma}, we calculate that $t$, $s$ and $ts$ have fixed point spaces of dimensions $69$, $79$ and $63$ on $V$. So once again we conclude that $T$ is $2$-elusive. The case $\bar{H} = A_1G_2$ is entirely similar and the same conclusion holds.

Now suppose $\bar{H} = A_1^2$. Here $p \geqs 5$, $H_0 = {\rm L}_2(q) \times {\rm PGL}_2(q)$ and \cite[Table 12.4]{Thomas} gives 
\begin{align*}
V \downarrow A_1^2 = & \, (V_{A_1}(6) \otimes V_{A_1}(4))/(V_{A_1}(4) \otimes V_{A_1}(6))/(V_{A_1}(4) \otimes V_{A_1}(2))/ \\ 
& \, (V_{A_1}(2) \otimes V_{A_1}(8))/(V_{A_1}(2) \otimes V_{A_1}(4))/(V_{A_1}(2) \otimes 0)/(0 \otimes V_{A_1}(2)).
\end{align*} 
Note that both $A_1$ factors in $\bar{H}$ are of adjoint type. We claim that $\bar{H}$ does not contain any involutions of type $E_6T_1$ and thus $T$ is not $2$-elusive. 

To see this, let $t \in \bar{H}$ be an involution in the first $A_1$ factor and let $s$ be an involution in the second. Then every involution in $\bar{H}$ is conjugate to $t$, $s$ or $ts$. Let $c \geqs 0$ be an even integer and consider the Weyl module $V_{A_1}(c)$ for an $A_1$ group of adjoint type. Then it is easy to check that the dimension of the fixed point space of an involution on $V_{A_1}(c)$ is given by the expression
\[
\frac{1}{2}\left(c+1+(-1)^{c/2}\right)
\]
and we deduce that $t$, $s$ and $ts$ have respective fixed point spaces of dimensions $63$, $69$ and $63$ on $V$. By inspecting Table \ref{table:oddpinvolutions}, it follows that $\bar{H}$ does not contain involutions of type $E_6T_1$, as claimed above.

For the two cases with $\bar{H} = A_1$ we note that $H_0 = {\rm L}_2(q)$ has a unique class of involutions and thus $T$ is not $2$-elusive. 

If $\bar{H} = A_2.2$, then Craven notes in \cite[Section 4]{Craven2} that $\bar{H}_{\s} = \PGL_3^{\varepsilon}(q).2$, and either $H_0 = \PGL_3^{\varepsilon}(q)$ or $H_0 = \PGL_3^{\varepsilon}(q).2$, without giving the precise structure in all cases. In any case, $H_0$ has at most two classes of involutions and thus $T$ is not $2$-elusive.

To complete the proof, we may assume $\bar{H} = (2^2 \times D_4).{\rm Sym}_3$, in which case there are two possibilities to consider (see \cite[Table 4.1]{Craven2} and the discussion on p.16 of \cite{Craven2} for further details). If $H_0 = {}^3D_4(q).3$ then $H_0$ has a unique class of involutions and thus $T$ is not $2$-elusive. The other possibility was handled in Lemma \ref{l:e7_exotic}.
\end{proof}

\vs

This completes the proof of Theorem \ref{t:main} for the groups with socle $T = E_7(q)$.

\section{The groups with socle $E_8(q)$}\label{s:e8}

Finally, we turn to the groups with socle $T = E_8(q)$. As usual, we write $\mathcal{M} = \mathcal{C} \cup \mathcal{S}$ for the set of core-free maximal subgroups of $G$ and we refer the reader to Remark \ref{r:CS}(g) for a brief description of the subgroups in the $\mathcal{C}$ and $\mathcal{S}$ collections. In particular, we recall that it remains an open problem to determine the subgroups in $\mathcal{S}$, even up to isomorphism. 
As noted in Section \ref{ss:invols}, $T$ has exactly $2+2\delta_{2,p}$ classes of involutions.

Our main result for the groups with $H \in \mathcal{C}$ is the following. This completes the proof of Theorem \ref{t:main}(i).

\begin{prop}\label{p:e8}
Suppose $T = E_8(q)$, $|\O|$ is even and $H \in \mathcal{C}$. Then $T$ is $2$-elusive unless one of the following holds:

\vspace{1mm}

\begin{itemize}\addtolength{\itemsep}{0.2\baselineskip}
\item[{\rm (i)}] $H_0 = {\rm SU}_5(q^2).4$, ${\rm PGU}_5(q^2).4$, ${\rm U}_3(q^2)^2.8$, or ${\rm U}_3(q^4).8$.
\item[{\rm (ii)}] $H_0 = (q^4 \pm q^3 + q^2 \pm q +1)^2.(5 \times {\rm SL}_2(5))$ or $(q^8 \pm q^7 \mp q^5 - q^4 \mp q^3 \pm q+1).30$.
\item[{\rm (iii)}] $p=2$ and $H_0$ is one of $\O_{8}^{+}(q^2).({\rm Sym}_3 \times 2)$, ${}^3D_4(q^2).6$, $(q^2 \pm q +1)^4.2.(3 \times {\rm U}_4(2))$, $(q^4-q^2+1)^2.(12 \circ {\rm GL}_2(3))$ or $(q^2+1)^4.(4 \circ 2^{1+4}).{\rm Alt}_6.2$.
\item[{\rm (iv)}] $H_0 = F_4(q)$ and $p =3$.
\item[{\rm (v)}] $H_0 = {\rm SO}_5(q)$ and $p \geqs 5$.
\item[{\rm (vi)}] $H_0 = {\rm PGL}_2(q)$ {\rm (three classes; $p \geqs 23,29,31$)}.
\item[{\rm (vii)}] $H_0 = {\rm ASL}_3(5)$ and $p \ne 2,5$.
\end{itemize}
\end{prop}

We divide the proof of Proposition \ref{p:e8} into several cases. In view of Propositions \ref{p:parab} and \ref{p:subfield}, we may assume that $H$ is not a parabolic nor a subfield subgroup. We start by dealing with the remaining subgroups arising in part (I) of Definition \ref{d:c}, so $H = N_G(\bar{H}_{\s})$ and one of the following holds:

\vspace{1mm}

\begin{itemize}\addtolength{\itemsep}{0.2\baselineskip}
\item[{\rm (a)}] $\bar{H}$ is a maximal torus;
\item[{\rm (b)}] $\bar{H}$ is reductive of maximal rank and $\bar{H}^{\circ}$ is not a maximal torus;
\item[{\rm (c)}] $\bar{H}$ is not a maximal rank subgroup.
\end{itemize}

\vspace{1mm}

Throughout, we will write $V = \mathcal{L}(\bar{G})$ for the adjoint module of $\bar{G}$. We begin by considering the cases where $\bar{H}$ is a maximal torus of $\bar{G}$, noting that the possibilities for $H_0$ are recorded in \cite[Table 5.2]{LSS}.

\begin{lem}\label{l:e8_2}
The conclusion to Proposition \ref{p:e8} holds if $H = N_G(\bar{H}_{\s})$ and $\bar{H}$ is a maximal torus of $\bar{G}$. 
\end{lem}

\begin{proof}
We can assume that $H = N_G(\left(\bar{T}^g\right)_{\s})$, where $\bar{T}$ is the $\s$-stable maximal torus from the setup of Section \ref{ss:setup}. Let $W = N_{\bar{G}}(\bar{T})/\bar{T} = 2.{\rm O}_{8}^{+}(2)$ be the Weyl group of $\bar{G}$.

By Lemma \ref{l:torus1}, the $\s$-stable torus $\bar{T}^g$ corresponds to a conjugacy class $x^W$ in $W$, and the latter determines $H$ up to $\bar{G}_{\s}$-conjugacy. Let $n \in N_{\bar{G}}(\bar{T})$ be a lift of $x$. By Lemmas \ref{l:tildeH0} and \ref{l:torus1}, for the purpose of determining the $2$-elusivity of $T$ it suffices to consider involutions in \begin{equation}\label{e:torus}
\widetilde{H_0} = N_{\bar{G}}(\bar{T})_{n\s} = S.F,
\end{equation}
where $S = \bar{T}_{n\s}$ and $F = C_W(x)$ (see Lemma \ref{l:torus2}). As noted above, the possibilities for $H_0$ are recorded in \cite[Table 5.2]{LSS}. 

First notice that if 
\[
S = (q^4 \pm q^3 + q^2 \pm q + 1)^2 \mbox{ or } q^8 \pm q^7 \mp q^5 - q^4 \mp q^3 \pm q + 1
\]
then $|S|$ is odd and $F \in \{5 \times {\rm SL}_2(5), Z_{30}\}$ has a unique class of involutions, so Lemma \ref{l:conj} implies that $T$ is not $2$-elusive.

We now consider the remaining cases. To begin with, we will assume that $q$ is odd. By \cite[Corollary 4.4]{BT}, $H_0$ contains involutions in the $D_8$-class, so it remains to determine whether or not $H_0$ also contains $A_1E_7$ involutions. 

Suppose $S = (q-\e)^8$, in which case we can assume that $n \in \{1,w\}$, where $w \in N_{\bar{G}}(\bar{T})$ corresponds to the central involution in the Weyl group. Since $w$ acts on $\Phi$ as $\alpha \mapsto -\alpha$, it follows that $w$ centralizes $h_{\alpha}(-1)$ for all $\alpha \in \Phi$. Thus $S = \bar{T}_{n\s}$ contains the representatives $h_{\alpha_1}(-1)$ and $h_{\alpha_1}(-1)h_{\alpha_2}(-1)$ from Table \ref{table:oddpinvolutions}. In particular, $T$ is $2$-elusive.

For $q$ odd, it remains to consider the following four possibilities for $S$:
\[
(q^2 \pm q + 1)^4, \, (q^2 + 1)^4, \, (q^4 - q^2 + 1)^2.
\]
In order to describe the element $n \in N_{\bar{G}}(\bar{T})$ in these cases (see \eqref{e:torus}), set $h_i = h_{\alpha_i}(-1)$ and $w_j = w_{\beta_j}$, where $1 \leqs i \leqs 8$, $1 \leqs j \leqs 240$ and $\beta_j$ is the $j$-th root in $\Phi$ with respect to the specific ordering of roots used by {\sc Magma} (see Remark \ref{r:ordering}).

First assume $S = (q^2+1)^4$. By inspecting \cite{GaltStaroletov}, we see that $H_0 = S.F$ is a non-split extension. Moreover, \cite[Table 8]{GaltStaroletov} gives $S = \bar{T}_{n \sigma}$ with $n = w_{2}w_{3}w_{4}w_{7}w_{120}w_{18}w_{8}w_{74}$ and a {\sc Magma} computation (which can be done over $\Q$, as explained in  Section \ref{ss:comp}) shows that $a = h_2 w_2 w_5$ and $b = w_4w_{17}$ centralize $n$. Therefore $a,b \in \widetilde{H_0}$. Moreover, we can use {\sc Magma} to show that $a^2$ and $(ab)^2$ are involutions, with fixed point spaces on $V$ with dimensions $120$ and $136$, respectively. Therefore, $T$ is $2$-elusive in this case.

% MAGMA code:
%G := GroupOfLieType("E8", RationalField());
%r := AdjointRepresentation(G);
%
%R := RootDatum(G);
%n := [elt<G|i> : i in [1..NumberOfPositiveRoots(R)]];
%h := [TorusTerm(G,i,-1) : i in [1..NumberOfPositiveRoots(R)]];
%w := [Transpose(i) : i in ReflectionMatrices(R)];
%n0:= h[2]*h[5]*h[7]*n[1]*n[2]*n[5]*n[7]*n[44]*n[71]*n[89]*n[120];
%
%// torus #56: (q^2+q+1)^4 
%// 2 classes of involutions
%// representatives: a^3, b
%x:=n[1]*n[2]*n[3]*n[5]*n[6]*n[8]*n[120]*n[69];
%
%a:=x^2*n0;
%b:=h[1]*h[4]*n[1]*n[4]*n[18]*n[44];
%c:=h[1]*h[3]*h[5]*h[6]*h[7]*h[8]*n[1]*n[2]*n[64]*n[116]*n[26]*n[28]*n[32]*n[120];
%
%x*a*x^-1*a^-1 eq 1; x*b*x^-1*b^-1 eq 1; x*c*x^-1*c^-1 eq 1;
%
%g1 := Matrix(r(a));
%g2 := Matrix(r(b));
%g3 := Matrix(r(c));

%// torus #67 (q^4-q^2+1)^2
%// 2 classes of involutions
%// representatives a^6, c^2
%x:=n[2]*n[32]*n[5]*n[7]*n[1]*n[4]*n[6]*n[65];
%
%a:=x;
%b:=h[3]*h[5]*h[6]*n[18]*n[45]*n[92]*n[112];
%c:=h[2]*h[3]*h[4]*n[2]*n[29]*n[4]*n[17]; 
%
%g1 := Matrix(r(a));
%g2 := Matrix(r(b));
%g3 := Matrix(r(c));

In each of the remaining cases (with $q$ odd), the maximal torus $S$ has a complement $R$ in $\widetilde{H_0}$, with explicit generators for $R$ presented in \cite[Table 7]{GaltStaroletov}. 

For example, suppose $S = (q^2 + q + 1)^4$, which is labelled as torus 56 in \cite{GaltStaroletov}. By inspecting \cite[Table 7]{GaltStaroletov}, we can take $n = w_{1}w_{2}w_{3}w_{5}w_{6}w_{8}w_{120}w_{69}$ and we set $R = \la a,b,c\ra$, where
\[
a  = n^2w_0, \;\; b = h_{1}h_{4}w_{1}w_{4}w_{18}w_{44}, \;\; c = h_{1}h_{3}h_{5}h_{6}h_{7}h_{8}w_{1}w_{2}w_{64}w_{116}w_{26}w_{28}w_{32}w_{120}
\]
and
\[
w_0 = h_{2}h_{5}h_{7}w_{1}w_{2}w_{5}w_{7}w_{44}w_{71}w_{89}w_{120}.
\]
A computation with {\sc Magma} shows that $R$ centralizes $n$, so $R \leqs \widetilde{H_0}$. In addition, we can use {\sc Magma} to show that $a^3$ and $b$ have order two, with fixed point spaces on $V$ of dimensions $120$ and $136$ respectively, so once again $T$ is $2$-elusive. By appealing to \cite[Lemma 2.5]{GaltStaroletov}, we deduce that the same conclusion holds when $S = (q^2-q+1)^4$. 

For $q$ odd, the final case $S = (q^4 - q^2 + 1)^2$ is very similar (this is torus 67 in \cite{GaltStaroletov}). Here we take $n = w_{2}w_{32}w_{5}w_{7}w_{1}w_{4}w_{6}w_{65}$ and $R = \la a,b,c \ra$, where
\[
a = n,\;\; b = h_{3}h_{5}h_{6}w_{18}w_{45}w_{92}w_{112},\;\; c = h_{2}h_{3}h_{4}w_{2}w_{29}w_{4}w_{17}.
\]
Then $R$ centralizes $n$ and we have $\widetilde{H_0} = SR$. Moreover, a {\sc Magma} computation shows that $a^6$ and $c^2$ are involutions with fixed point spaces on $V$ of dimensions $120$ and $136$ respectively, so $T$ is $2$-elusive.

% MAGMA code:
%G := GroupOfLieType("E8", RationalField());
%r := AdjointRepresentation(G);
%
%R := RootDatum(G);
%n := [elt<G|i> : i in [1..NumberOfPositiveRoots(R)]];
%h := [TorusTerm(G,i,-1) : i in [1..NumberOfPositiveRoots(R)]];
%w := [Transpose(i) : i in ReflectionMatrices(R)];
%n0:= h[2]*h[5]*h[7]*n[1]*n[2]*n[5]*n[7]*n[44]*n[71]*n[89]*n[120];
%
%// torus #56: (q^2+q+1)^4 
%// 2 classes of involutions
%// representatives: a^3, b
%x:=n[1]*n[2]*n[3]*n[5]*n[6]*n[8]*n[120]*n[69];
%
%a:=x^2*n0;
%b:=h[1]*h[4]*n[1]*n[4]*n[18]*n[44];
%c:=h[1]*h[3]*h[5]*h[6]*h[7]*h[8]*n[1]*n[2]*n[64]*n[116]*n[26]*n[28]*n[32]*n[120];
%
%x*a*x^-1*a^-1 eq 1; x*b*x^-1*b^-1 eq 1; x*c*x^-1*c^-1 eq 1;
%
%g1 := Matrix(r(a));
%g2 := Matrix(r(b));
%g3 := Matrix(r(c));

%// torus #67 (q^4-q^2+1)^2
%// 2 classes of involutions
%// representatives a^6, c^2
%x:=n[2]*n[32]*n[5]*n[7]*n[1]*n[4]*n[6]*n[65];
%
%a:=x;
%b:=h[3]*h[5]*h[6]*n[18]*n[45]*n[92]*n[112];
%c:=h[2]*h[3]*h[4]*n[2]*n[29]*n[4]*n[17]; 
%
%g1 := Matrix(r(a));
%g2 := Matrix(r(b));
%g3 := Matrix(r(c));
 
To complete the proof, we may assume $p=2$. Here Theorem \ref{t:titsweyl} implies that $H_0 = S{:}F$ is a split extension and we may identify $F$ with a subgroup of $\la w_{\alpha} \, : \, \alpha \in \Phi \ra$. Since the involutions in the class labelled $A_1$ are long root elements, \cite[Proposition 1.13(iii)]{LLS} implies that $H_0$ contains such an involution if and only if $F$ contains a reflection in $W$. By inspecting \cite[Table 5.2]{LSS}, we see that $F$ contains such an element if and only if $H_0 = (q-\e)^8{:}W$, so $T$ is not $2$-elusive in all of the remaining cases. Finally, using {\sc Magma} one can check that the subgroup $W < E_8(2)$ contains a representative of every $E_8(2)$-class of involutions and this allows us to conclude that $T$ is $2$-elusive when $S = (q-\e)^8$.
\end{proof}

\begin{lem}\label{l:e8_3}
The conclusion to Proposition \ref{p:e8} holds if $H = N_G(\bar{H}_{\s})$ and $\bar{H}$ is a maximal rank subgroup.
\end{lem}

\begin{proof}
In view of Lemma \ref{l:e8_2}, we may assume $\bar{H}^\circ$ is reductive and not a maximal torus. By Theorem \ref{t:odd}, we see that $|\O|$ is odd if $H_0$ is of type $D_8(q)$, $D_4(q)^2$ or $A_1(q)^8$, so we can exclude these cases for the remainder of the proof. Then by inspecting \cite[Table 5.1]{LSS}, the possibilities for $\bar{H}$ are as follows:
\begin{equation}\label{e:e8list1}
A_1E_7, \; A_8.2, \; A_2E_6.2, \; A_4^2.4, \; D_4^2.({\rm Sym}_3 \times 2), \; A_2^4.{\rm GL}_2(3).
\end{equation}
As before, we are free to assume that $H = N_G(\left( \bar{H}^g \right)_{\s})$, where $\bar{H}$ is the normalizer of a standard subsystem subgroup 
\[
\bar{H}^\circ = \langle \bar{T}, U_{\alpha} \,:\, \alpha \in \Phi' \rangle,
\]
and $\bar{H}^g$ is a $\s$-invariant conjugate of $\bar{H}$. Here $\bar{T}$ is a maximal torus of $\bar{G}$, as in the setup of Section \ref{ss:setup}, and $\Phi'$ is a root subsystem of $\Phi$. As usual, we will denote the longest root of $\Phi$ by $\alpha_0$. 

Recall that the Weyl group of $\bar{H}^\circ$ can be identified with the subgroup $\langle s_{\alpha} : \alpha \in \Phi' \rangle$ of $W$. Then $\bar{H}$ is generated by $\bar{H}^\circ$, together with the elements of $N_{\bar{G}}(\bar{T})$ that normalize the Weyl group of $\bar{H}^\circ$. In particular, $\bar{H}$ contains an element $w \in N_{\bar{G}}(\bar{T})$ which maps to the central involution of $W$. Explicitly, we can choose $$w = w_{\alpha_{1}} w_{\alpha_{2}} w_{\alpha_{5}} w_{\alpha_{7}} w_{\alpha_{44}} w_{\alpha_{71}} w_{\alpha_{89}} w_{\alpha_{120}} \in \bar{H},$$ where $\alpha_i$ denotes the $i$-th positive root in $\Phi$ with respect to the specific ordering of the roots of $\bar{G} = E_8$ used by {\sc Magma} (see Remark \ref{r:ordering}). Then $w$ acts as $\alpha \mapsto -\alpha$ on $\Phi$, and a computation shows that $w^2 = 1$, so $w$ is an involution in $\bar{G}$.

For later use, we note that when $q$ is odd, a computation with {\sc Magma} shows that $w$ has a fixed point space of dimension $120$ on $V$. Therefore $w$ is an involution of type $D_8$ (Table \ref{table:oddpinvolutions}) when $q$ is odd.

Let $n \in N_{\bar{G}}(\bar{T})$ be an element of $\bar{H}$ such that under the bijection of Lemma \ref{l:sigmaclasses}, the $\s$-invariant conjugate $\bar{H}^g$ corresponds to the image of $n$ in $H^1(\s, \bar{H}/\bar{H}^\circ)$. Then in view of Lemma \ref{l:tildeH0}, for determining $2$-elusivity it suffices to consider involutions in $\widetilde{H_0} = \bar{H}_{n\s}$. 

As in the proof of Lemma \ref{l:e7_mr}, we will partition the analysis according to the parity of $q$. As before, we refer to \cite[Table 5.1]{LSS} for the precise structure of $H$ in each case. To begin with, we will assume $q$ is odd, and we will consider each possibility for $\bar{H}$ in turn (see \eqref{e:e8list1}).

\vs

\noindent \emph{Case 1.1. $\bar{H} = A_1E_7$, $q$ odd.}

\vs

Suppose $\bar{H} = A_1E_7$, in which case $\Phi'$ has base $\{\alpha_1, \alpha_2, \ldots, \alpha_7\} \cup \{ -\alpha_0\}$ and $H_0 = \bar{H}_{\s}$. Thus $H_0$ contains the representatives $h_{\alpha_1}(-1)$ and $h_{\alpha_1}(-1)h_{\alpha_2}(-1)$ from Table \ref{table:oddpinvolutions}, and we conclude that $T$ is $2$-elusive in this case.

\vs

\noindent \emph{Case 1.2. $\bar{H} = A_8.2$ or $A_2E_6.2$, $q$ odd.}

\vs

Next suppose $\bar{H} = A_8.2$, in which case $\Phi'$ has base $\{\alpha_1, \alpha_3, \alpha_4, \ldots, \alpha_8, -\alpha_0\}$. Moreover $\bar{H} = \bar{H}^\circ{:} \langle w \rangle$, so we can take $n \in \{1,w\}$. Now $w$ centralizes $h_{\alpha}(-1)$ for all $\alpha \in \Phi$, so $h_{\alpha}(-1) \in \widetilde{H_0}$ for all $\alpha \in \Phi$. Thus $\widetilde{H_0}$ contains the representatives $h_{\alpha_1}(-1)$ and $h_{\alpha_1}(-1)h_{\alpha_2}(-1)$ from Table \ref{table:oddpinvolutions}, and so $T$ is $2$-elusive.

Now assume $\bar{H} = A_2E_6.2$, in which case $\Phi'$ has base $\{\alpha_1, \ldots, \alpha_6\} \cup \{\alpha_8, -\alpha_0\}$. Here we also have $\bar{H} = \bar{H}^\circ{:}\langle w \rangle$ and once again we conclude that $T$ is $2$-elusive.

\vs

\noindent \emph{Case 1.3. $\bar{H} = A_4^2.4$, $q$ odd.}

\vs

Next assume $\bar{H} = A_4^2.4$, in which case $\Phi'$ has base $\{\alpha_1, \alpha_3, \alpha_4, \alpha_2\} \cup \{\alpha_6, \alpha_7, \alpha_8, -\alpha_0\}$. We can choose $n \in \{1,w,x,x^{-1}\}$, where the image of $x$ in $\bar{H}/\bar{H}^\circ$ has order $4$.

If $n \in \{1,w\}$, then as in Case 1.2 it is easy to see that the representatives from Table \ref{table:oddpinvolutions} are contained in $\widetilde{H_0}$, whence $T$ is $2$-elusive. Now assume $n \in \{x,x^{-1}\}$, in which case $H_0 = J.4$ with $J = {\rm SU}_5(q^2)$ or ${\rm PGU}_5(q^2)$. We will prove that every involution in $H_0$ is of type $D_8$.

First we consider the types of involutions in $\bar{H}^\circ = A_4^2$. There are two conjugacy classes of involutions in each $A_4$ factor, with representatives given by $$t_1 = h_{\alpha_1}(-1),\ t_2 = h_{\alpha_1}(-1)h_{\alpha_4}(-1)$$ in the first factor, and $$t_1' = h_{\alpha_6}(-1),\ t_2' = h_{\alpha_6}(-1)h_{\alpha_8}(-1)$$ in the second factor. A computation with {\sc Magma} shows that for $i = 1,2$ the involution $t_it_i'$ has a fixed point space of dimension $120$ on $V$, and thus belongs to class $D_8$ (Table \ref{table:oddpinvolutions}). We conclude that the involutions in $\bar{H}^\circ$ of the form $(t,t)$ belong to the class $D_8$. Now an involution in $J < H_0$ embeds into $\bar{H}^\circ$ as $(t,t)$, so it follows that every involution in $J$ is of type $D_8$.

Next we will prove that every involution in $\bar{H} \setminus \bar{H}^\circ$ is of type $D_8$. To this end, note that $w \in \bar{H} \setminus \bar{H}^\circ$, so all involutions in $\bar{H} \setminus \bar{H}^\circ$ are contained in $\bar{H}^\circ w$. Since $w$ acts on $\Phi$ as $\alpha \mapsto -\alpha$, the action of $w$ on both $A_4$ factors of $\bar{H}^\circ = A_4^2$ is via the standard inverse-transpose graph automorphism. Hence it follows from \cite[Lemma 4.4.6]{GLS} that there is a unique $\bar{H}$-class of involutions in $\bar{H}^\circ w$. Since $w$ is an involution of type $D_8$, it follows that every involution in $\bar{H} \setminus \bar{H}^\circ$ is of type $D_8$. Consequently, the involutions in $H_0 \setminus J$ are of type $D_8$, so we conclude that every involution in $H_0$ is of type $D_8$, and $T$ is not $2$-elusive.

\vs

\noindent \emph{Case 1.4. $\bar{H} = D_4^2.({\rm Sym}_3 \times 2)$, $q$ odd.}

\vs

In this case $\Phi'$ has base $\{\alpha_3, \alpha_4, \alpha_2, \alpha_5\} \cup \{\alpha_7, \alpha_8, \beta, -\alpha_0\}$, where $-\beta$ is the longest root in the $D_8$ root subsystem with base $\{\alpha_2, \alpha_3, \ldots, \alpha_8, -\alpha_0\}$. Explicitly we have $\beta = \alpha_{97}$.

Here $H_0$ is of type $D_4(q)^2$, $D_4(q^2)$, ${}^3D_4(q)^2$ or ${}^3D_4(q^2)$. As noted in the proof of \cite[Lemma 2.5]{LSS}, the image of $n$ in $\bar{H}/\bar{H}^{\circ} = {\rm Sym}_3 \times 2$ is either the central involution, or it has order $1$, $3$ or $6$. In other words, the image of $n$ is contained in the unique cyclic subgroup of order $6$ in ${\rm Sym}_3 \times 2$. A computation shows that the image of $$g = w_{\alpha_{92}} w_{\alpha_{89}} w_{\alpha_{90}} w_{\alpha_{60}} w_{\alpha_{72}} w_{\alpha_{56}}$$ in the Weyl group of $E_8$ has order $6$ and it normalizes the Weyl group of $\bar{H}^\circ$. Another computation shows that $g$ is also an element of order $6$ in $\bar{G}$. Hence $g \in \bar{H}$ and $\bar{H}^\circ.6 = \bar{H}^\circ{:}\langle g \rangle$, so we can assume that $n \in \langle g \rangle$. In particular, $n$ centralizes $g$ and thus $g \in \widetilde{H_0} = \bar{H}_{n\s}$. 

A computation with {\sc Magma} shows that $g^3$ has a fixed point space of dimension $136$ on $V$, so $g^3$ is an involution of type $A_1E_7$. Moreover $g$ centralizes $t = h_{\alpha_2}(-1)h_{\alpha_7}(-1)$, so $t \in \widetilde{H_0}$. Yet another computation shows that $t$ is an involution of type $D_8$. So we conclude that $T$ is $2$-elusive in every case.

\vs

\noindent \emph{Case 1.5. $\bar{H} = A_2^4.{\rm GL}_2(3)$, $q$ odd.}

\vs

In this case $\Phi'$ has base $\{\alpha_1, \alpha_3\} \cup \{\alpha_5, \alpha_6\} \cup \{-\alpha_{69}, \alpha_2\} \cup \{\alpha_8,-\alpha_0\}$, where $\alpha_{69}$ is the longest root in a root subsystem of type $E_6$. By the proof of \cite[Lemma 2.5]{LSS}, the image of $n$ in $\bar{H}/\bar{H}^\circ \cong \GL_2(3)$ is contained in a cyclic subgroup of order $8$. Thus we can take $n \in \{1,w,x^2,x\}$, where the image of $x$ has order $8$ in $\GL_2(3)$. Then $\widetilde{H_0}$ is contained in $\bar{H}^\circ \langle x \rangle$.

If $n \in \{1,w\}$, then as in Case 1.2, the representatives from Table \ref{table:oddpinvolutions} are contained in $\widetilde{H_0}$, and so $T$ is $2$-elusive. Now assume $n \in \{x^2,x\}$, in which case $H_0 = J.8$ and $J = {\rm U}_3(q^2)^2$ or $J = {\rm U}_3(q^4)$. We will prove that every involution in $H_0$ is of type $D_8$.

To this end, let $t \in \widetilde{H_0}$ be an involution. If $t \in \bar{H}^\circ$, then up to conjugacy, $t$ embeds in $\bar{H}^{\circ}$ as the image of $(z,z,z,z)$, $(z,z,1,1)$ or $(1,1,z,z)$, where $z$ is an involution in $A_2$ (note that $(z,z,z,z)$ is the only possibility when $J = {\rm U}_3(q^4)$). In any case, representatives in $\bar{H}^{\circ}$ for these involutions are given by \begin{align*} t_1 &= h_{\alpha_1}(-1) h_{\alpha_5}(-1) h_{\alpha_2}(-1) h_{\alpha_8}(-1), \\ t_2 &= h_{\alpha_1}(-1) h_{\alpha_5}(-1), \\ t_3 &= h_{\alpha_2}(-1) h_{\alpha_8}(-1).\end{align*} A computation with {\sc Magma} shows that $t_1$, $t_2$, $t_3$ all have a fixed point space of dimension $120$ on $V$, and so they are all of type $D_8$.

Suppose then that $t \in \bar{H} \setminus \bar{H}^\circ$, in which case $t \in \bar{H}^\circ w$. As in Case 1.3, it follows from \cite[Lemma 4.4.6]{GLS} that $t$ must be $\bar{H}$-conjugate to $w$. Since $w$ is an involution of type $D_8$, we conclude that every involution in $H_0$ is of type $D_8$, and $T$ is not $2$-elusive.

\vs

To complete the proof, we may assume $q$ is even.

\vs

\noindent \emph{Case 2.1. $\bar{H} = A_1E_7$, $A_8.2$ or $A_2E_6.2$, $q$ even.}

\vs

In each of these cases, we can use \cite{Law09} to show that $\bar{H}^{\circ}$ contains a representative of each $\bar{G}$-class of involutions, whence $T$ is $2$-elusive.

\vs

\noindent \emph{Case 2.2. $\bar{H} = A_4^2.4$ or $D_4^2.({\rm Sym}_3 \times 2)$, $q$ even.}

\vs

First assume $\bar{H} = A_4^2.4$. As above, by inspecting \cite{Law09} we see that $\bar{H}^{\circ}$ contains a representative of each $\bar{G}$-class of involutions and this immediately implies that $T$ is $2$-elusive when $H_0$ is of type ${\rm L}_5^{\e}(q)^2$. On the other hand, if $H_0 = {\rm SU}_5(q^2).4$ or ${\rm PGU}_5(q^2).4$, then $H_0$ only has three classes of involutions, so $T$ is not $2$-elusive.

Now suppose $\bar{H} = D_4^2.({\rm Sym}_3 \times 2)$. If $H_0 = \O_8^{+}(q^2).({\rm Sym}_3 \times 2)$ or ${}^3D_4(q^2).6$, then \cite[Proposition 1.13]{LLS} implies that $H_0$ does not contain any $A_1$-type involutions, so $T$ is not $2$-elusive. Now assume $H_0 = \O_8^{+}(q)^2.({\rm Sym}_3 \times 2)$ or ${}^3D_4(q)^2.6$. There are two classes of involutions in ${}^3D_4(q)$, which embed in the algebraic group $D_4$ as involutions of type $a_2$ or $c_4$ (in terms of the notation in \cite{AS}). So in both cases, by considering the natural embedding of $H_0$ in $D_4^2$, we deduce that $H_0$ contains involutions of the form 
\[
(a_2,1), \; (a_2,a_2), \; (c_4,1), \; (c_4,c_4),
\] 
which embed in $D_8<\bar{G}$ as involutions of type $a_2$, $a_4$, $c_4$ and $c_8$. With respect to Lawther's notation for involution classes in $D_8$ (see \cite[Table 8]{Law09}), these elements are contained in the respective $D_8$-classes $A_1$, $2A_1$, $A_1+D_2$ and $3A_1+D_2$. Then by inspecting \cite[Section 4.13]{Law09} we deduce that the corresponding $\bar{G}$-classes are $A_1$, $A_1^2$, $A_1^3$ and $A_1^4$, respectively, and thus $T$ is $2$-elusive.

\vs

\noindent \emph{Case 2.3. $\bar{H} = A_2^4.{\rm GL}_2(3)$, $q$ even.}

\vs

First assume $H_0 = {\rm U}_3(q^2)^2.8$ or ${\rm U}_3(q^4).8$. In view of \cite[Proposition 1.13]{LLS} and the embedding of $H_0$ in $\bar{H}$, it is easy to see that $H_0$ does not contain any $A_1$-type involutions and thus $T$ is not $2$-elusive. Now assume $H_0$ is of type ${\rm L}_3^{\e}(q)^4$. We may embed 
\[
\bar{H}^{\circ} = A_2(A_2^3) < A_2E_6 < \bar{G}
\]
and we note that the maximal rank subgroup $A_2^3 < E_6$ contains a representative of all three involution classes in $E_6$. Then by considering the embedding of $A_2E_6$ in $\bar{G}$ (see \cite[Section 4.15]{Law09}) we deduce that $\bar{H}^{\circ}$ contains a representative of all four $\bar{G}$-classes of involutions. This implies that $T$ is $2$-elusive and the proof of the lemma is complete.
\end{proof}

\begin{lem}\label{l:e8_4}
The conclusion to Proposition \ref{p:e8} holds if $H = N_G(\bar{H}_{\s})$ and $\bar{H}$ is a positive-dimensional non-maximal rank subgroup of $\bar{G}$.
\end{lem}

\begin{proof}
According to \cite[Theorem 1]{LS04}, the possibilities for $\bar{H}$ are as follows:
\begin{equation}\label{e:e8list2}
\renewcommand*{\arraystretch}{1.3}
\begin{array}{c}
G_2F_4, \; A_1G_2^2.2 \, (p \geqs 3), \; F_4 \, (p=3), \; A_1A_2.2 \, (p \geqs 5), \\
B_2 \, (p \geqs 5), \, A_1 \, (\mbox{$3$ classes; $p \geqs 23,29,31$}), \, A_1 \times {\rm Sym}_5 \, (p \geqs 7).
\end{array}
\end{equation}
Note that we have included the special case $(\bar{H},p) = (F_4,3)$ described in \cite{CST}, which was incorrectly omitted in \cite{LS04}.

We now consider each possibility for $\bar{H}$ in \eqref{e:e8list2} and we will begin by assuming $\bar{H}^{\circ} = A_1$. If $\bar{H} = A_1$ then $H_0 = {\rm PGL}_2(q)$ and $T$ is not $2$-elusive since $\bar{H}$ has a unique class of involutions. So we may assume $\bar{H} = A_1 \times {\rm Sym}_5$, in which case $H_0 = {\rm PGL}_2(q) \times {\rm Sym}_5$ and $p \geqs 7$. We claim that $T$ is $2$-elusive. 

To see this, first note that the ${\rm Sym}_5$ factor contains involutions of type $A_1E_7$ (see \cite[Lemma 5.6]{BT}). Let $z \in {\rm PGL}_2(q) < H_0$ be the image of $(-I_1,I_1) \in {\rm GL}_2(q)$. Then as explained in the proof of \cite[Lemma 5.6]{BT}, we can embed $z$ in a maximal rank subgroup $J = A_4^2$ of $\bar{G}$, where $z = z_1z_2$ and each $z_i \in A_4$ acts as $(-I_2,I_3)$ on the natural module for $A_4$. Now 
\[
V \downarrow J = \mathcal{L}(A_4^2) / (U_1 \otimes \Lambda^2(U_2)) / (\Lambda^2(U_1) \otimes U_2^*) / (\Lambda^2(U_1)^* \otimes U_2) / (U_1^* \otimes \Lambda^2(U_2)^*)
\]
where $U_1$ and $U_2$ are the natural modules for the two $A_4$ factors of $J$ (see \cite[Table 12.5]{Thomas}). We calculate that $z$ has a $24$-dimensional $1$-eigenspace on each summand, so $\dim C_V(z) = 120$ and therefore $z$ is a $D_8$-type involution. This justifies the claim.

Now suppose $\bar{H} = B_2$, so $H_0 = {\rm SO}_5(q)$, $p \geqs 5$ and \cite[Table 12.5]{Thomas} gives 
\[
V \downarrow B_2 = V_{B_2}(3\omega_1+2\omega_2) / V_{B_2}(6\omega_2) / V_{B_2}(2\omega_2).
\]
The involutions in $H_0$ are of type $(-I_4,I_1)$ and $(-I_2,I_3)$, which we can write as $h_{\beta_1}'(-1)$ and $w_{\beta_2}'$ with respect to a set of simple roots $\{\beta_1, \beta_2\}$ for $B_2$. A calculation with {\sc Magma} (see Section \ref{ss:comp}) shows that both involutions have trace $-8$ on $V$. Therefore, every involution in $H_0$ is of type $D_8$ and thus $T$ is not $2$-elusive.

Next assume $\bar{H} = A_1A_2.2$, so $p \geqs 5$. By inspecting \cite[Table 18]{BT}, we see that $H_0$ contains involutions of type $A_1E_7$. Let $t = (z,1) \in {\rm L}_2(q) \times {\rm L}_3^{\e}(q) < H_0$ be an involution. Now $z$ has trace $(-1)^{\ell}$ on the Weyl module $V_{A_1}(2\ell)$ and so by considering the restriction of $V$ to $\bar{H}^{\circ}$ (see \cite[Table 12.5]{Thomas}) we calculate that $\dim C_V(t)=120$ and thus $t$ is an involution of type $D_8$. 

The case $\bar{H} = A_1G_2^2.2$ is very similar. Here $p \geqs 3$ and the socle of $H_0$ is either ${\rm L}_2(q) \times G_2(q)^2$ or ${\rm L}_2(q) \times G_2(q^2)$. By \cite[Table 18]{BT}, $H_0$ contains involutions of type $A_1E_7$. Let $t$ be an involution in the ${\rm L}_2(q)$ factor of ${\rm soc}(H_0)$. Then by considering the restriction of $V$ to $\bar{H}^{\circ}$ we calculate that $\dim C_V(t)=120$ and thus $T$ is $2$-elusive.

Next suppose that $\bar{H} = G_2F_4$, so $H_0 = G_2(q) \times F_4(q)$. If $p=2$ then the information in \cite[Table 38]{Law09} immediately implies that $T$ is $2$-elusive, so let us assume $p \geqs 3$. By inspecting \cite[Table 18]{BT}, we see that $H_0$ contains involutions of type $A_1E_7$. Let $t$ be a $B_4$-type involution in the $F_4(q)$ factor and note that $t$ has trace $-6$ and $20$ on the respective Weyl modules $V_{F_4}(\delta_4)$ and $V_{F_4}(\delta_1)$. From \cite[Table 12.5]{Thomas} we see that 
\[
V \downarrow G_2F_4 = \mathcal{L}(G_2F_4)/(V_{G_2}(\omega_1) \otimes V_{F_4}(\delta_4))/(V_{G_2}(\omega_2) \otimes 0)/(0 \otimes V_{F_4}(\delta_1))
\]
and this allows us to deduce that $\dim C_V(t) = 120$. Therefore, $t$ is an involution of type $D_8$ and we conclude that $T$ is $2$-elusive.

To complete the proof, it remains to consider the special case arising in \cite{CST}, where $\bar{H} = F_4$ and $p=3$. Let $t_1,t_2 \in H_0 = F_4(q)$ be representatives of the two $H_0$-classes of involutions as in Table \ref{table:oddpinvolutions}, say $t_1 = h_{\b_1}'(-1)$ and $t_2 = h_{\b_4}'(-1)$, where $C_{\bar{H}}(t_1) = A_1C_3$ and $C_{\bar{H}}(t_2) = B_4$.  In \cite[Section 3]{CST}, expressions for the generators $x_{\beta_i}'(c), h_{\beta_i}'(c)$ of $\bar{H}$ are presented in terms of the generators for $\bar{G}$, namely 
\begin{align*}
t_1 &= h_{\b_1}'(-1) = h_{\alpha_4}(-1)h_{\alpha_6}(-1), \\
t_2 &= h_{\b_4}'(-1) = h_{\alpha_1}(-1)h_{\alpha_3}(-1)h_{\alpha_4}(-1)h_{\alpha_5}(-1)h_{\alpha_7}(-1).
\end{align*} 
With the aid of {\sc Magma} (see Section \ref{ss:comp}), we can use these expressions to show that both $t_1$ and $t_2$ have trace $-8$ on $V$. Therefore, every involution in $H_0$ is of type $D_8$ and thus $T$ is not $2$-elusive.
\end{proof}

In order to complete the proof of Proposition \ref{p:e8}, it just remains to consider the groups where $H$ is either an exotic local subgroup or the Borovik subgroup.

\begin{lem}\label{l:e8_5}
The conclusion to Proposition \ref{p:e8} holds if $H$ is an exotic local subgroup.
\end{lem}

\begin{proof}
The exotic local subgroups were classified in \cite{CLSS}, and from \cite[Theorem 1]{CLSS} we see that there are two possibilities:

\vspace{1mm}

\begin{itemize}\addtolength{\itemsep}{0.2\baselineskip}
\item[(a)] $H_0 = 5^3.{\rm SL}_3(5)$, $p \ne 2,5$ and $q \in \{p,p^2\}$;
\item[(b)] $H_0 = 2^{5+10}.{\rm L}_5(2)$ and $q=p \geqs 3$.
\end{itemize}

\vspace{1mm}

If $H_0 = 5^3.{\rm SL}_3(5)$, then $p \ne 2,5$ and it follows from \cite[Lemma 5.2]{CLSS} that $H_0$ is isomorphic to the affine group ${\rm ASL}_3(5)$. This implies that $H_0$ has a unique class of involutions and thus $T$ is not $2$-elusive. 

Now assume $H_0 = 2^{5+10}.{\rm L}_5(2)$. Here $H_0 = N_T(E)$, where $E = 2^5$ is elementary abelian, and we inspect the proof of \cite[Lemma 2.17]{CLSS}. As explained in ``Part A" of the proof, every involution in $E$ is of type $D_8$. And in ``Part B" we see that $E$ is centralized by an element $e_1 \in T$ (in the notation of \cite{CLSS}), which is of type $A_1E_7$. Therefore, $T$ is $2$-elusive.
\end{proof}

\begin{lem}\label{l:e8_bor}
Suppose $T = E_8(q)$, $p \geqs 7$ and $H_0 = ({\rm Alt}_5 \times {\rm Sym}_6).2$ is the Borovik subgroup. Then $T$ is $2$-elusive.
\end{lem}

\begin{proof}
We will refer to Borovik's original paper \cite{Bor} for various properties of this subgroup. Write $H_0 = (L_1 \times L_2.2).2$, where $L_1 = {\rm Alt}_5$ and $L_2 = {\rm Alt}_6$. By \cite[Lemma 6.8]{Bor}, every involution in $L_1$ and $L_2$ is of type $D_8$. 

In order to establish the existence of involutions in the $A_1E_7$ class, let $z \in L_1$ be an element of order $3$ and set $J = C_{\bar{G}}(z)$. Following \cite[p.177]{Bor}, we have $J = A_8$ and $L_2$ acts irreducibly on the natural module $U$ for $J = A_8$. Therefore, $L_2.2 = {\rm Sym}_6$ also acts irreducibly on $U$ and from the character table of ${\rm Sym}_6$ we deduce that a transposition $x \in L_2.2$ acts as $(-I_6,I_3)$ on $U$. Finally, we observe that
\[
V \downarrow J = \mathcal{L}(J)/\Lambda^3(U)/\Lambda^6(U)
\]
(see \cite[Table 12.5]{Thomas}, noting that $\l_5$ should be $\l_6$ in the case labelled 62) and we calculate that $x$ has trace $8$ on each summand. Therefore, $x$ has trace $24$ on $V$ and we conclude that $x$ is an involution of type $A_1E_7$.
\end{proof}

To complete the proof of Theorem \ref{t:main}, we may assume $T = E_8(q)$ and $H \in \mathcal{S}$, in which case $H$ is almost simple with socle $S$. Recall that ${\rm Lie}(p)$ denotes the set of simple groups of Lie type defined over a field of characteristic $p$. 

We begin by considering the groups with $S \not\in {\rm Lie}(p)$. Recall from Remark \ref{r:CS}(g) that $S$ is one of the following (it remains an open problem to determine the precise list of possibilities for $S$):

\vspace{1mm}

\begin{itemize}\addtolength{\itemsep}{0.2\baselineskip}
\item[{\rm (a)}] $S = {\rm Alt}_6$ $(p \ne 5)$, ${\rm Alt}_7$ $(p \ne 2)$, ${\rm M}_{11}$ $(p =3,11)$, ${\rm J}_3$ $(p=2)$ or ${\rm Th}$ $(p=3)$;

\item[{\rm (b)}] $S = {\rm L}_2(q')$ with $q' \in \{7,8,11,13,16,17,19,25,29,31,32,41,49,61\}$; or

\item[{\rm (c)}] $S = {\rm L}_3(3)$, ${\rm L}_3(5)$, ${\rm L}_4(5)$, ${\rm U}_3(3)$, ${\rm U}_4(2)$, ${\rm PSp}_4(5)$, ${}^2B_2(8)$, ${}^2B_2(32)$, ${}^3D_4(2)$ or ${}^2F_4(2)'$.
\end{itemize}

\begin{prop}\label{p:e8_as_1}
Suppose $T = E_8(q)$, $|\O|$ is even and $H \in \mathcal{S}$ with socle $S \not\in {\rm Lie}(p)$. Then $T$ is $2$-elusive only if one of the following holds:

\vspace{1mm}

\begin{itemize}\addtolength{\itemsep}{0.2\baselineskip}
\item[{\rm (i)}] $H_0 = {\rm Sym}_6$, ${\rm Alt}_6.2 \cong {\rm PGL}_2(9)$ or ${\rm Alt}_6.2^2$, with $p \ne 2,5$.  
\item[{\rm (ii)}] $H_0 = {\rm PGL}_2(r)$ with $(r,p) = (7,3)$, $(11,5)$ or $(13,7)$.
\item[{\rm (iii)}] $H_0 = {\rm L}_3(3).2$ and $p=13$.
\end{itemize}
\end{prop}

\begin{proof}
We will divide the proof into a number of separate cases, according to the socle $S$. As usual, let $V = \mathcal{L}(\bar{G})$ be the adjoint module.

\vs

\noindent \emph{Case 1. $S = {\rm Alt}_n$, $n \geqs 5$.}

\vs

First assume $S = {\rm Alt}_n$ is an alternating group. It is not known whether or not $T$ has a maximal subgroup with socle $S$, but the main theorem of \cite{Craven_alt} tells us that this can only happen if $S = {\rm Alt}_6$ and $p \ne 5$, or if $S = {\rm Alt}_7$ and $p \geqs 3$. Then by considering the number of conjugacy classes of involutions in $H_0$, it follows that $T$ is $2$-elusive only if

\vspace{1mm}

\begin{itemize}\addtolength{\itemsep}{0.2\baselineskip}
\item[(a)] $H_0 \in \{{\rm Sym}_6, {\rm PGL}_2(9), {\rm Alt}_6.2^2\}$ and $p \ne 2,5$; or
\item[(b)] $H_0 = {\rm Sym}_7$ and $p \geqs 3$.
\end{itemize}

The possibilities in (a) are recorded in part (i) of the proposition. As a comment, we observe that computations with {\sc Magma} show that for all the groups listed in (a), there are several feasible characters with property (\textbf{P}) such that $H_0$ intersects every $T$-class of involutions, and also several such that every involution $H_0$ belongs to the same $T$-class. So this case remains inconclusive.

Now assume $H_0 = {\rm Sym}_7$, so $G = T$ and $H = {\rm Sym}_7$. Then by inspecting \cite[Theorem 4]{Craven_alt}, we see that the composition factors of $V \downarrow S$ are as follows:
\[
\begin{array}{ll}
15^4, 13^6, 10^5, (10^*)^5, 1^{10} & p = 3 \\
35^4, 15^4, 10, 10^*, 8^2, 6^2 & p = 5 \\
\hspace{-4.8mm} \left\{\begin{array}{l}
\mbox{$35^3, 21, 14_a, 14_b^2, 10^7, 5^2$ or} \\
35^4, 14_a^2, 10^6, 5^4 \end{array}\right. & p = 7 \\
35^4, 15^4, 14_a^2, 10, 10^* & p \geqs 11.
\end{array}
\]

Using the computational approach described in Section \ref{ss:feasible}, we can work with Litterick's {\sc Magma} code in \cite{LittGithub} to find all the feasible characters of ${\rm Sym}_7$ on $V$. As a consequence, we find that if $p \not\in \{3,7\}$, then ${\rm Sym}_7$ does not have a feasible character that is consistent with the composition factors of $V\downarrow S$ given above. 

Finally, for $p \in \{3,7\}$ we claim that every involution in $H_0$ is of type $D_8$ and thus $T$ is not $2$-elusive. Suppose $p=3$. Using {\sc Magma}, we find that the only feasible character of $H_0 = {\rm Sym}_7$, which is consistent with $V \downarrow S$, has composition factors 
\[
20^5, 15_a^3, 15_b, 13_a, 13_b^5, 1_a^6, 1_b^4.
\]
Here $V \downarrow S$ has property (\textbf{P}) and the composition factors of $V \downarrow S$ are described in the row labelled 1) in \cite[Table 6.250]{Litt}. In any case, with these composition factors every involution in $H_0$ has trace $-8$ on $V$ and the claim follows. Similarly, if $p = 7$ then there are only three feasible characters of $H_0 = {\rm Sym}_7$ that are consistent with $V \downarrow S$, namely 
\[
\left\{\begin{array}{l}
35_a^4, 14_a^2, 10_a^6, 5_b^4    \\
35_a^2,35_b, 21_b, 14_b,14_c,14_d, 10_a^4,10_b^3, 5_a,5_b \\
35_a^2,35_b^2, 14_b^2, 10_a^4,10_b^2, 5_a^2,5_b^2.
\end{array}\right.
\]
Here the respective composition factors of $V \downarrow S$ are as in entries 3), 2), 3) of \cite[Table 6.248]{Litt}. Once again, for each possibility we calculate that every involution in $H_0$ is of type $D_8$ and the proof of the claim is complete.

\vs

\noindent \emph{Case 2. $S$ is a sporadic group.}

\vs

Here the possibilities for $S$ can be read off from \cite[Table 10.2]{LS99} and by applying \cite[Theorem 8]{Litt} we deduce that $S = {\rm M}_{11}$ (with $p=3$ or $11$), ${\rm Th}$ ($p=3$) or ${\rm J}_3$ ($p=2$). In the first two cases, $H= H_0 = S$ has a unique class of involutions and thus $T$ is not $2$-elusive. The same conclusion holds when $S = {\rm J}_3$ and $p=2$ since $H_0$ has at most two classes of involutions, whereas $T$ has four.

\vs

\noindent \emph{Case 3. $S$ is a group of Lie type.}

\vs

By inspecting \cite[Tables 10.3, 10.4]{LS99} and \cite[Theorem 8]{Litt} we see that $S$ is one of the following:
\begin{itemize}\addtolength{\itemsep}{0.2\baselineskip}
\item[(a)] $S = {\rm L}_2(q')$ with $q' \in \{7,8,11,13,16,17,19,25,29,31,32,41,49,61\}$;
\item[(b)] $S = {\rm L}_3(3)$, ${\rm L}_3(5)$, ${\rm U}_3(3)$ $(p=2,7)$, ${\rm U}_4(2)$, ${}^3D_4(2)$, ${}^2F_4(2)'$ $(p=3)$, ${}^2B_2(8)$, ${\rm L}_4(5)$ $(p=2)$, ${\rm PSp}_4(5)$ $(p=2)$ or ${}^2B_2(32)$ $(p=5)$.
\end{itemize}
Note that in (a) we have excluded the case $q'=9$, since ${\rm L}_2(9) \cong {\rm Alt}_6$.

\vs

\noindent \emph{Case 3(a). $S = {\rm L}_2(q')$.}

\vs

We begin by considering the groups in (a). In each case, any almost simple group with socle $S$ has at most $3$ classes of involutions, so we can assume that $p$ is odd. In addition, we may assume $H_0 \ne S$ since $S$ has a unique class of involutions. In particular, if $q'$ is a prime, then $H_0 = {\rm PGL}_2(q')$ is the only possibility. Note that if $q' \in \{8,32\}$ then $H_0$ has a unique class of involutions, so $T$ is not $2$-elusive in these cases. For the remainder, we may assume $q' \ne 8,32$. 

To handle the remaining cases, we proceed as in Case 1, using Litterick  \cite{LittGithub} to analyze feasible characters with the aid of {\sc Magma}, as discussed in Section \ref{ss:feasible}. Recall the definition of property (\textbf{P}) (see Definition \ref{d:P}) and recall that $H_0 \in \mathcal{S}$ only if there exists a feasible character of $H_0$ on $V \downarrow H_0$ with property (\textbf{P}). For the remainder of the proof, we will refer to such a character as a \emph{compatible feasible character}.

First assume $q'$ is a prime, in which case we may assume $H_0 = {\rm PGL}_2(q')$. If we take $q' \in \{17,29,41,61\}$, then one can check that $H_0$ has no compatible feasible characters. Now assume $q' \in \{7,11,13,19,31\}$. In these cases, we have used {\sc Magma} to determine all the compatible feasible characters (see Example \ref{ex:e8cfc}) and by computing traces we deduce that $T$ is $2$-elusive only if $(H_0,p)$ is one of the following:
\[
({\rm PGL}_2(7),3), \; ({\rm PGL}_2(11),5), \; ({\rm PGL}_2(13),7),
\]
which correspond to the cases listed in part (ii) of the proposition. Therefore, to complete our analysis of the groups with $S = {\rm L}_2(q')$, we may assume $q' \in \{16,25,49\}$. In each case, we claim that $T$ is not $2$-elusive.

Suppose $q' = 16$, in which case $H_0 = {\rm L}_2(16).2$ or $H_0 = {\rm Aut}(S) = {\rm P\Gamma L}_2(16) = {\rm L}_2(16).4$. Using {\sc Magma} to calculate all the feasible characters of ${\rm L}_2(16).2$ on $V$, we deduce that every involution in ${\rm L}_2(16).2$ has trace $-8$ on $V$. Since every involution in ${\rm Aut}(S)$ is contained in ${\rm L}_2(16).2$, we conclude that $T$ is not $2$-elusive when $q' = 16$.

Next assume $q' = 25$. Once again we claim that every involution in $H_0$ is of type $D_8$. Since every involution in ${\rm Aut}(S)$ is contained in ${\rm PGL}_2(25)$ or ${\rm P\Sigma L}_2(25)$, we may assume that $H_0 = S.2$ is one of these two groups. For $H_0 = {\rm P\Sigma L}_2(25)$ we find that $H_0$ has a compatible feasible character if and only if $p=13$, in which case every involution in $H_0$ has trace $-8$ on $V$ and the claim follows. And a very similar argument  gives the same conclusion when $H_0 = {\rm PGL}_2(25)$ (here $H_0$ has compatible feasible characters for all $p \geqs 3$ with $p \ne 5$).

Finally, suppose $q' = 49$. As in the previous case, we may assume $H_0 = {\rm PGL}_2(49)$ or ${\rm P\Sigma L}_2(49)$, and with the aid of {\sc Magma} we can show that $H_0$ does not admit a compatible feasible character.

\vs

\noindent \emph{Case 3(b). The remaining Lie type groups.}

\vs

To complete the proof of the proposition, we may assume $S$ is one of the groups in (b). Just by considering the number of classes of involutions in $H_0$ we see that $T$ is $2$-elusive only if $S = {\rm U}_4(2)$ $(p \geqs 5)$, ${}^3D_4(2)$ or ${}^2F_4(2)'$ $(p=3)$, or if $H_0$ is one of the following:
\[
{\rm L}_{3}(3).2 \; (p \geqs 5), \; {\rm L}_3(5).2 \; (p \geqs 3), \; {\rm U}_3(3).2 \; (p=7),
\]
\[
{\rm L}_4(5).2^2 \; (p=2), \; {\rm L}_4(5).{\rm D}_8 \; (p=2), \; {\rm PSp}_4(5).2 \; (p=2).
\]
We now need to consider each of these possibilities in turn. 

First assume $H_0 = {\rm L}_3(3).2$ and $p \geqs 5$, noting that $H_0$ has two classes of involutions. If $p \ne 13$, then $p$ does not divide $|H_0|$. Computing the feasible characters of $H_0$ with {\sc Magma} shows that every feasible character of $H_0$ on $V$ has a trivial composition factor, and so there are no compatible feasible characters in this case (see Lemma \ref{l:liefixed}). However, the same approach is inconclusive for $p = 13$. Indeed, there are four feasible characters and we find that $T$ is $2$-elusive with respect to exactly three of them.

Now suppose $H_0 = {\rm U}_3(3).2$ and $p = 7$. Here we use {\sc Magma} to determine the list of compatible feasible characters and in each case we find that every involution in $H_0$ is of type $D_8$. Hence, $T$ is not $2$-elusive. We can also eliminate the case $H_0 = {\rm L}_3(5).2$ with $p \geqs 3$ since $H_0$ does not admit a compatible feasible character.

Next assume $S  = {\rm U}_4(2)$ and $p \geqs 5$. If $H_0 = S$, then by examining the feasible characters labelled (\textbf{P}) in \cite[Tables 6.334, 6.335]{Litt}, and by inspecting the Brauer character table of $H_0$ (see \cite[pp.60--62]{ModularAtlas}), we conclude that every involution in $H_0$ has trace $-8$ on $V$. Therefore $T$ is not $2$-elusive when $H_0 = S$. Now assume $H_0 = {\rm U}_4(2).2$. Here we determine that $H_0$ has a unique compatible feasible character and by calculating traces we conclude once again that $T$ is not $2$-elusive.

For $S = {}^3D_4(2)$ we first note that ${\rm Aut}(S) = S.3$ and so every involution in $H_0$ is contained in $S$. Now $S$ has two classes of involutions, and by inspecting the compatible feasible characters in \cite[Tables 6.343, 6.344]{Litt}, we deduce that every involution in $H_0$ is of type $D_8$. Therefore, $T$ is not $2$-elusive in this case.

Next suppose $S = {}^2F_4(2)'$ with $p=3$, noting that every involution in $H_0$ is contained in $S$. For $H_0 = {\rm Aut}(S) = S.2$ we use {\sc Magma} to show that $H_0$ has no compatible feasible characters, so we must have $H_0 = S$. By inspecting \cite[Table 6.347]{Litt} we deduce that every compatible feasible character of $H_0$ corresponds to an embedding with the property that $V \downarrow H_0$ has composition factors $124_a$ and $124_b$ (note that the label (\textbf{P}) has been incorrectly omitted in the first row of \cite[Table 6.347]{Litt}). Then both classes of involutions in $H_0$ have trace $-8$ on $V$ and thus $T$ is not $2$-elusive.

Next assume $S = {\rm L}_4(5)$ and $p=2$. First observe that $H_0$ has at most $3$ classes of involutions, unless $H_0 = S.2^2$ or $H_0 = {\rm Aut}(S) = S.{\rm D}_8$. In particular, we may assume $H_0$ contains $J = S.2$, which is the extension isomorphic to the unique subgroup of index $2$ in ${\rm PGL}_4(5)$. In addition, $S$ acts irreducibly on $V$ (see \cite[Table 6.329]{Litt}), so the same must be true for $H_0$ and $J = S.2$. As discussed in Example \ref{e:psl4-5}, every $248$-dimensional irreducible $K[J]$-module $W$ can be constructed with {\sc Magma}. In this way, we find that there is an involution $x \in J$ with $\dim C_W(x) = 124$. But since there is no such involution in $T$, it follows that $J$ does not embed into $T$. We conclude that $T$ is not $2$-elusive if $S = {\rm L}_4(5)$ and $p = 2$.

Finally, let us assume $S = {\rm PSp}_4(5)$ and $p=2$. Since $S$ has only two classes of involutions, we may assume $H_0 = {\rm Aut}(S) = S.2$, in which case $H_0$ has four such classes, with representatives $t_1,t_2,t_3$ and $t_4$ such that $|C_{H_0}(t_i)| = 31200, 28800, 1440$ and $960$ for $i = 1,2,3$ and $4$, respectively. We calculate that $H_0$ has $6$ absolutely irreducible modules of dimension at most $248$ in characteristic $p = 2$. In Table \ref{tab:dims}, we record $\dim C_W(t_i)$ for each such module $W$ and each involution $t_i$. 

{\scriptsize
\begin{table}
\[
\begin{array}{ccccc} \hline
W & \dim C_W(t_1) & \dim C_W(t_2) & \dim C_W(t_3) & \dim C_W(t_4) \\ \hline
1 & 1 & 1 & 1 & 1 \\
24 & 12 & 16 & 12 & 12 \\
40 & 26 & 24 & 22 & 20 \\
64 & 38 & 36 & 32 & 34 \\
104_a & 65 & 60 & 55 & 54 \\
104_b & 52 & 56 & 56 & 52 \\ \hline
\end{array}
\]
\caption{The case $T = E_8(q)$, $p=2$, ${\rm soc}(H) = {\rm PSp}_4(5)$} 
\label{tab:dims}
\end{table}
}

Using {\sc Magma} to compute the feasible characters of $H_0$, we see that there are two possibilities for the composition factors of $V \downarrow H_0$, namely
\[
104_a^2, 40 \; \mbox{ and } \; 64, 40^2, 24^4, 1^8.
\]
(As noted in Remark (iii) on p.316 of \cite{PS}, if $q=2$ then $V \downarrow H_0$ has composition factors $104_a^2, 40$.) If we write $U_1, \ldots, U_k$ for the composition factors in each case, then we can read off $\dim C_{U_j}(t_i)$ from Table \ref{tab:dims} and the trivial bound
\[
\dim C_V(t_i) \leqs \sum_{j=1}^k \dim C_{U_j}(t_i) 
\]
implies that $\dim C_V(t_i) \leqs 156$ for all $i$. This means that $H_0$ does not contain any $A_1$-type involutions (since $\dim C_V(y) = 190$ for each involution  $y$ in the class $A_1$) and thus $T$ is not $2$-elusive.
\end{proof}

Finally, in order to complete the proof of Theorem \ref{t:main}, we may assume $T = E_8(q)$ and $H \in \mathcal{S}$ is an almost simple subgroup with socle $S \in {\rm Lie}(p)$.

\begin{prop}\label{p:e8_as_2}
Suppose $T = E_8(q)$, $|\O|$ is even, and $H \in \mathcal{S}$ with socle $S \in {\rm Lie}(p)$. Then $T$ is $2$-elusive only if $S = {\rm L}_2(q_0)$, where $p$ is odd and $q_0$ is a power of $p$ in the range $7 \leqs q_0 \leqs 2621$.
\end{prop}

\begin{proof}
As discussed in Remark \ref{r:CS}(g), it remains an open problem to determine the maximal subgroups of this form (even up to isomorphism). However, there has been significant progress towards this goal and at the time of writing, the possibilities for $S$ are as follows:

\vspace{1mm}
 
\begin{itemize}\addtolength{\itemsep}{0.2\baselineskip}
\item[(a)] $S = {\rm L}_2(q_0)$ with $7 \leqs q_0 \leqs (2,q-1) \cdot 1312$; or
\item[(b)] $S = {\rm L}_3^{\e}(3)$, ${\rm L}_3^{\e}(4)$, ${\rm U}_3(8)$, ${\rm U}_4(2)$ or ${}^2B_2(8)$.
\end{itemize}

\vspace{1mm}

\noindent Let us also note that there is not a single known example of a maximal subgroup $H \in \mathcal{S}$ of this form.

First assume $S = {\rm L}_2(q_0)$, where $q_0 = p^e \leqs (2,q-1)\cdot 1312$. If $p=2$ then $q_0 \leqs 2^{10}$ and it is easy to check that $H_0$ has at most three conjugacy classes of involutions, hence $T$ is not $2$-elusive. Now assume $p$ is odd, so we have $7 \leqs q_0 \leqs 2621$ as in the statement of the proposition (note that there is no maximal subgroup $H \in \mathcal{S}$ with socle ${\rm L}_2(5) \cong {\rm Alt}_5$ by \cite[Theorem 2]{Craven_alt}). Of course, if $H_0 = S$ then $H_0$ has a unique class of involutions and $T$ is not $2$-elusive. So we may assume $H_0 \ne S$ has at least two classes of involutions (for example, we could have $H_0 = {\rm PGL}_2(q_0)$) but we have not been able to rule out any of these possibilities. One of the main obstacles here is the very large number of possibilities for the composition factors of $V \downarrow H_0$, which typically leads to a situation where there exists a compatible feasible character for which there are involutions $x,y \in H_0$ with respective traces $-8$ and $24$ on $V$. So to resolve this situation, we would essentially have to rule out the existence of such a maximal subgroup, which seems to be a very difficult problem. There is important ongoing work of Craven \cite{CravenE8} in this direction, but there is still a lot more to do.

To complete the proof, let us turn to the cases in (b). If $S = {\rm L}_3^{\e}(3)$ then $T$ is $2$-elusive only if $H_0 = S.2$, in which case both $T$ and $H_0$ have two classes of involutions. Similarly, if $p=2$ then we may assume $H_0 = {\rm L}_3(4).2^2$, ${\rm L}_3(4).{\rm D}_{12}$ or ${\rm U}_4(2).2$.

Suppose $H_0 = {\rm L}_3(3).2$ and note that the $3$-modular Brauer character table of $H_0$ is available in \textsf{GAP} \cite{GAP}. By inspecting \cite[Proposition 8.1(2)]{Craven0} we deduce that the set of composition factors of $V \downarrow H_0$ is one of the following:
\[
(30^5, 12, 7^8, 6^3, 1^{12}),\; (30^3, 12^3, 7^{10}, 6^7, 1^{10}), \; (30^4, 27, 12, 7^8, 6^4, 1^{9})
\]
noting that $H_0$ has two irreducible modules over $\bar{\mathbb{F}}_{3}$ of dimension $7$, and also two of dimension $27$. For each possibility, by inspecting the Brauer character table, we deduce that the involutions in $S$ have trace $-8$ on $V$, so they are of type $D_8$. However, if $x \in H_0 \setminus S$ is an involution, then the trace of $x$ on $V$ is neither $-8$ nor $24$, no matter which modules of dimension $7$ or $27$ we choose. This means that $H_0 = S$ is the only possibility and so $T$ is not $2$-elusive in this case.

A very similar argument applies when $H_0 = {\rm U}_3(3).2$. Here \cite[Proposition 8.2(1)]{Craven0} implies that the set of composition factors of $V \downarrow H_0$ is one of the following:
\[
(30^4, 12^3, 7^8, 6^5, 1^6),\; (30^4, 27, 12^2, 7^{6}, 6^5, 1^5).
\]
Then by inspecting the $3$-modular Brauer character table of $H_0$ we deduce that the involutions in $S$ are of type $D_8$, but the trace of an involution in $H_0 \setminus S$ is incompatible with containment in $T$. So as above, this case can be discarded.

Finally, let us assume $p=2$ and $H_0 = {\rm L}_3(4).2^2$, ${\rm L}_3(4).{\rm D}_{12}$ or ${\rm U}_4(2).2$. Suppose $H_0 = {\rm U}_4(2).2$, in which case $H_0$ has four conjugacy classes of involutions. The $2$-modular Brauer character table of $H_0$ is available in \textsf{GAP} and by applying \cite[Proposition 7.2]{Craven0} we deduce that $V \downarrow H_0$ has composition factors 
\[
40^2, 14^4, 8^7, 6^8, 1^8.
\]
Let $x \in H_0$ be an involution. We can use {\sc Magma} to compute $\dim C_{V_i}(x)$ for each composition factor $V_i$ of $V \downarrow H_0$ and by summing these dimensions we deduce that $\dim C_V(x) \leqs 158$ for every involution $x \in H_0$. So this implies that $H_0$ does not contain any involutions in the $T$-class $A_1$ and thus $T$ is not $2$-elusive. 

An entirely similar argument, applying \cite[Proposition 8.1]{Craven0}, shows that the same conclusion holds when $H_0 = {\rm L}_3(4).2^2$ or ${\rm L}_3(4).{\rm D}_{12}$ (in fact, we find that $\dim C_V(x) \leqs 144$ for every involution in $x \in H_0$, so there are no $A_1$ or $A_1^2$ involutions in $H_0$).
\end{proof}

\vs

This completes the proof of Theorem \ref{t:main}. 

\section{The tables}\label{s:tab}

In this final section we present Tables \ref{tab:main}, \ref{tab:main2} and \ref{tab:main3} from Theorem \ref{t:main}. Note that in Table \ref{tab:main} with $T = F_4(q)$ and $p=2$, we write ``graphs" to indicate that $H$ is maximal only if $G$ contains a graph (or graph-field) automorphism of $T$.

\clearpage 

{\scriptsize
\begin{table}
\renewcommand\thetable{A}
\[
\begin{array}{lll} \hline
T & H_0 & \mbox{Conditions}  \\ \hline
{}^2F_4(q)' & 
3^{1+2}{:}{\rm D}_8,\; 13{:}6 & \mbox{$q=2$, $G = T.2$} \\
& {\rm PGU}_3(q).2, \, {\rm SU}_3(q).2,\, (q+1)^2{:}{\rm GL}_2(3) & q \geqs 8 \\
& (q^2 \pm \sqrt{2q^3}+q \pm \sqrt{2q}+1){:}12 & q \geqs 8 \\
& &  \\
{}^3D_4(q) & (q^4-q^2+1).4,\, (q^2 \pm q +1)^2.{\rm SL}_2(3) & p=2  \\
& {\rm PGL}_3^{\e}(q) & \mbox{$p=2$, $q \geqs 4$, $q \equiv \e \imod{3}$} \\
& &  \\
F_4(q) & {}^3D_4(q).3 & p \geqs 3 \\
& {\rm PGL}_2(q) & p \geqs 13 \\
& G_2(q) & p = 7 \\
& {\rm ASL}_3(3) & q = p \geqs 5 \\
& {}^2F_4(q_0) & \mbox{$q = q_0^2$, $q_0 = 2^a$, $a \geqs 1$ odd} \\
& ({\rm SL}_3^{\e}(q) \circ {\rm SL}_3^{\e}(q)).e.2 & p = 2,\, e = (3,q-\e) \\
& {\rm Sp}_4(q^2).2 & \mbox{$p=2$, graphs} \\
& (q^2 +\e q +1)^2.(3 \times {\rm SL}_2(3)) & \mbox{$p = 2$, graphs, $q\geqs 4$ if $\e=-$} \\
& (q^4-q^2+1).12, \, (q^2+1)^2.({\rm SL}_2(3){:}4)  & \mbox{$p = 2$, graphs, $q\geqs 4$} \\
& &  \\
E_6^{\e}(q) & {\rm L}_3^{\e}(q^3).3, \, G_2(q),\, ({}^3D_4(q) \times (q^2+\e q+1)/e).3  & \\
& (q^2+\e q+1)^3/e.(3^{1+2}.{\rm SL}_2(3)) & \\
& {\rm PGL}_3^{\pm}(q).2 & p \geqs 5, \, q \equiv \e \imod{4} \\
& 3^{3+3}{:}{\rm SL}_3(3) & q = p \geqs 5, \, q \equiv \e \imod{3} \\
& &  \\
E_7(q) & P_2, \, P_5, \, P_7 & q \equiv 3 \imod{4} \\
& ({\rm L}_2(q^3) \times {}^3D_4(q)).3, \, {\rm L}_2(q^7).7 &  \\
& {\rm L}_2(q) \times {\rm PGL}_2(q), \, {\rm PGL}_3^{\pm}(q).2 & p \geqs 5 \\
& {}^3D_4(q).3 & p \geqs 3 \\
& {\rm L}_2(q) & \mbox{2 classes; $p \geqs 17,19$} \\
& &  \\
E_8(q) & {\rm SU}_5(q^2).4, \, {\rm PGU}_5(q^2).4, \, {\rm U}_3(q^2)^2.8, \, {\rm U}_3(q^4).8 &  \\
& (q^4 \pm q^3 + q^2 \pm q +1)^2.(5 \times {\rm SL}_2(5)) & \\
&  (q^8 \pm q^7 \mp q^5 - q^4 \mp q^3 \pm q+1).30 & \\
& \O_{8}^{+}(q^2).({\rm Sym}_3 \times 2), \, {}^3D_4(q^2).6, \, (q^2 + q +1)^4.2.(3 \times {\rm U}_4(2)) & p = 2 \\
& (q^4-q^2+1)^2.(12 \circ {\rm GL}_2(3)), \, (q^2+1)^4.(4 \circ 2^{1+4}).{\rm Alt}_6.2 & p=2 \\
& (q^2 - q +1)^4.2.(3 \times {\rm U}_4(2)) & \mbox{$p = 2$, $q \geqs 4$} \\
& F_4(q) & p = 3 \\
& {\rm SO}_5(q) & p \geqs 5 \\
& {\rm PGL}_2(q) & \mbox{3 classes; $p \geqs 23,29,31$} \\
& {\rm ASL}_3(5) & p \ne 2,5 \\
\hline
\end{array}
\]
\caption{The pairs $(T,H_0)$ in Theorem \ref{t:main}(i): $H \in \mathcal{C}$, $T$ is not $2$-elusive}
\label{tab:main}
\end{table}
}

{\scriptsize
\begin{table}
\renewcommand\thetable{B}
\[
\begin{array}{lll} \hline
T & H_0 & \mbox{Conditions}  \\ \hline
G_2(q)' & {\rm J}_2 & q = 4 \\ 
& {\rm J}_1 & q = 11 \\
& {\rm U}_3(3).2 & q = p \geqs 5 \\
& {\rm L}_2(13) & \mbox{$q=p \equiv \pm 1, \pm 3, \pm 4 \imod{13}$, or} \\
& & \mbox{$q=p^2$, $p \ne 2$ and $p \equiv \pm 2, \pm 5, \pm 6 \imod{13}$} \\ 
& {\rm L}_2(8) & \mbox{$q=p \equiv \pm 1 \imod{9}$, or} \\
& & \mbox{$q=p^3$, $p \ne 2$ and $p \equiv \pm 2, \pm 4 \imod{9}$} \\ 
& & \\
{}^2F_4(q)' & {\rm Alt}_6.2^2 & \mbox{$q = 2$, $G=T$} \\
& & \\
F_4(q) & {\rm L}_4(3).2_2 & q = 2 \\
& {}^3D_4(2).3 & q = p \geqs 3 \\
& & \\
E_6^{\e}(q) & {}^2F_4(2) & \mbox{$q = p \equiv \e \imod{4}$, $G=T$} \\
& \O_7(3) & \mbox{$(\e,q) = (-,2)$, $G = T.2$} \\ 
& {\rm Fi}_{22} & (\e,q) = (-,2) \\ \hline 
\end{array}
\]
\caption{The pairs $(T,H_0)$ in Theorem \ref{t:main}(ii): $H \in \mathcal{S}$, $T$ is $2$-elusive}
\label{tab:main2}
\end{table}
}

{\scriptsize
\begin{table}
\renewcommand\thetable{C}
\[
\begin{array}{ll} \hline
H_0 & \mbox{Conditions}  \\ \hline
{\rm Sym}_6, \, {\rm Alt}_6.2 \cong {\rm PGL}_2(9),\, {\rm Alt}_6.2^2 & p \ne 2,5 \\
{\rm PGL}_2(r) & (r,p) = (7,3), \, (11,5),\, (13,7) \\
{\rm L}_3(3).2 & p=13 \\ \hline
\end{array}
\]
\caption{The subgroups $H_0$ in Theorem \ref{t:main}(iii): $T = E_8(q)$, $H \in \mathcal{S}$}
\label{tab:main3}
\end{table}
}

\end{document}